\documentclass{amsart}
\usepackage{amssymb}
\usepackage[curve,matrix,arrow,frame,tips]{xy}
\usepackage{color}

\theoremstyle{plain}
\newtheorem{thm}{Theorem}[section]
\newtheorem{prop}[thm]{Proposition}
\newtheorem{lemma}[thm]{Lemma}
\newtheorem{cor}[thm]{Corollary}

\theoremstyle{definition}
\newtheorem{dfn}[thm]{Definition}
\newtheorem{note}[thm]{Notation}

\newtheorem{hypothesis}[thm]{Hypothesis}

\theoremstyle{remark}
\newtheorem{rem}[thm]{Remark}

\newcommand{\rH}{\mathrm{H}}
\newcommand{\rab}{\mathrm{ab}}
\newcommand{\rWeil}{\mathrm{Weil}}
\newcommand{\rth}{\mathrm{th}}
\newcommand{\ps}[1]{[\![ #1 ]\!]}

\begin{document}

\title{Cup products on curves over finite fields}

\date{\today}

\author[Bleher]{ Frauke M. Bleher}
\address{Frauke M. Bleher, Dept. of Mathematics\\Univ. of Iowa\\Iowa City, IA 52242, USA}
\email{frauke-bleher@uiowa.edu}
\thanks{The first author was supported in part by NSF Grant No.\ DMS-1801328
and Simons Foundation Grant 960170.
The second author was supported in part by NSF SaTC Grant No.
CNS-1701785 and Simons Foundation Grant MP-TSM-00002279.
}

\author[Chinburg]{Ted Chinburg} 
\address{Ted Chinburg, Dept. of Mathematics\\ Univ. of Pennsylvania \\ Philadelphia, PA 19104, USA}
\email{ted@math.upenn.edu}


\begin{abstract}
Suppose $k$ is a finite field, that $C$ is a smooth projective geometrically irreducible curve over $k$, and that $n$ is a positive integer not divisible by the characteristic of $k$. In this paper we compute cup products  of elements of the \'etale cohomology groups $\rH^1(C,\mathbb{Z}/n)$ and  $\rH^1(C,\mu_n)$. Over the algebraic closure $\overline{k}$ of $k$, such cup products are connected to values of the Weil pairing on the $n$-torsion of the Jacobian of $\overline{C} = \overline{k} \otimes_k C$ by using a fixed isomorphism between $\mathbb{Z}/n$ and $\mu_n$ over  $\overline{C}$.  Over $k$, such cup products are more subtle due to the fact that they take values in the group $\rH^2(C,\mu_n)=\mathrm{Pic}(C)/n\cdot \mathrm{Pic}(C)$ rather than in the group $\rH^2(\overline{C},\mu_n) = \mathbb{Z}/n$.
\end{abstract}

\maketitle

\section{Introduction}
\label{s:intro}
\setcounter{equation}{0}

Let $\overline{C}$ be a smooth projective irreducible curve of positive genus over an algebraically closed field $\overline{k}$.  Let $n$ be a positive integer not divisible by the characteristic of $\overline{k}$ and let $\tilde{\mu}_n$ be the group of $n^{\rth}$ roots of unity in $\overline{k}$. The Weil pairing 
\begin{equation}
\label{eq:Weildef2}
\langle\  , \ \rangle_{\rWeil}: \quad \mathrm{Jac}(\overline{C})[n] \times \mathrm{Jac}(\overline{C})[n] \to \tilde{\mu}_n
\end{equation}
on the $n$-torsion of the Jacobian of $\overline{C}$ is a classical topic in algebraic geometry.  It arises from the principal polarization of this Jacobian, and has applications to the study of Galois actions on 
$\mathrm{Jac}(\overline{C})[n]$ (see \cite{Serre}, \cite{MayleWang}) and to pairing-based cryptography (see \cite{Joux},\cite{BS},\cite{Miller}).

  It follows from \cite[\S20]{Mumford} and \cite[Dualit\'e \S3]{SGA4.5} that  $\langle\  , \ \rangle_{\rWeil}$
may be identified with the cup product 
\begin{equation}
\label{eq:Weildef}
\rH^1(\overline{C},\mu_n) \times \rH^1(\overline{C},\mu_n) \to \rH^2(\overline{C},\mu_n^{\otimes 2}) = \tilde{\mu}_n.
\end{equation}
Let $\mathcal{F}_n$ be $\mathbb{Z}/n$ or $\mu_n$. Since $\rH^1(\overline{C},\mu_n)=\rH^1(\overline{C},\mathbb{Z}/n)\otimes\tilde{\mu}_n$, we can use any choice of isomorphism $\mathbb{Z}/n\to \tilde{\mu}_n$ to compute the cup product 
\begin{equation}
\label{eq:ohWeil!}
\  \cup_{\overline{C}} \ : \quad\;\, \rH^1(\overline{C},\mathbb{Z}/n) \times \rH^1(\overline{C},\mathcal{F}_n) \to \rH^2(\overline{C},\mathcal{F}_n) 
\end{equation}
from the Weil pairing. Note that $\rH^2(\overline{C},\mu_n) = \mathbb{Z}/n$.

The object of this paper is to generalize classical formulas for the Weil pairing used to compute the pairing (\ref{eq:ohWeil!}) by replacing $\overline{C}$ by a smooth projective geometrically irreducible curve $C$ of positive genus $g(C)$ over a finite field $k$. We give expressions for the cup product and the triple product
\begin{eqnarray}
\label{eq:cupprodbi}
\rH^1(C,\mathbb{Z}/n) \times \rH^1(C,\mu_n) & \xrightarrow{\cup} & \rH^2(C,\mu_n)=\mathrm{Pic}(C)/n\cdot \mathrm{Pic}(C) \quad \mbox{and}\\
\label{eq:cupprodtri}
 \rH^1(C, \mathbb{Z}/n) \times  \rH^1(C,\mathbb{Z}/n) \times \rH^1(C,\mu_n) &\xrightarrow{\cup}& \rH^3(C,\mu_n) =  \mathbb{Z}/n
\end{eqnarray}
when $n$ is relatively prime to $\#k$.

 In our first result, Theorem \ref{thm:MSformula}, we will povide general formulas for (\ref{eq:cupprodbi}) and (\ref{eq:cupprodtri}) based on work of McCallum and Sharifi in \cite{MS}.  
The computation of the terms arising in Theorem \ref{thm:MSformula} is in general 
more difficult than that of various expressions for the Weil pairing  \cite{Howe,Miller,LuRo2015}.  

The computation of the above pairings can be reduced to the case in which $n$ is a power of a prime $\ell$.  There are then infinitely many closed points  $O$ of $C$ of degree $d(O)$ prime to $\ell \cdot (\ell - 1)$.     For $\mathcal{F}_n = \mathbb{Z}/n$ or $\mathcal{F}_n = \mu_n$, we describe below a direct sum decomposition 
\begin{equation}
\label{eq:sumonetwo}
\rH^1(C,\mathcal{F}_n) = \rH^1(k,\mathcal{F}_n) \oplus \rH^1(C,\mathcal{F}_n)_O
\end{equation}
in which $\rH^1(C,\mathcal{F}_n)_O$ consists of classes which are normalized at $O$ in an appropriate sense. 
The pairing of $\rH^1(k,\mathbb{Z}/n)$ with $\rH^1(k,\mu_n)$ is trivial since $\rH^2(k,\mu_n) = 0$.  We will describe below how to compute the pairing (\ref{eq:cupprodbi}) when exactly one of the terms belongs to $\rH^1(k,\mathcal{F}_n)$ using  a new homomorphism which we call the Legendre derivative of Frobenius.  Finally suppose that the terms in (\ref{eq:cupprodbi})  belong to $\rH^1(C,\mathbb{Z}/n)_O$ and $\rH^1(C,\mu_n)_O$. 
We show in Theorem \ref{thm:EllThm} that when $C$ has genus $1$, the pairing on normalized classes can be computed using only the Weil pairing of the restrictions of these classes to $\rH^1(\overline{C},\mathcal{F}_n)$.  The proof does not in fact use the arithmetic approach in Theorem \ref{thm:MSformula}, but is instead a delicate homological calculation. In \S  \ref{s:genustwo} we prove that the counterpart of Theorem \ref{thm:EllThm} for higher genus curves need not hold by constructing an infinite family of genus $2$ counterexamples.   

To state our results precisely, we will  identify
 $\rH^1(C,\mu_n)$ with the quotient of 
$$D(C) = \{b \in k(C)^*:  \mathrm{div}_C(b) \in n \cdot \mathrm{Div}(C)\}$$ 
by the subgroup $(k(C)^*)^n$;  see   Lemma \ref{lem:secondgroup}.  Let $[b] \in \rH^1(C,\mu_n)$ be the cohomology class determined by $b \in D(C)$.  There is a surjection $\rH^1(C,\mu_n) \to \mathrm{Pic}(C)[n]$ which sends $[b]$ to the $n$-torsion divisor class $[\mathrm{div}_C(b)/n]$ of $\mathrm{div}_C(b)/n$.   We will view $\mathrm{Pic}(C)[n]$ as a subgroup of $\mathrm{Pic}(\overline{C})[n] = \mathrm{Jac}(\overline{C})[n]$ when $\overline{C} = \overline{k} \otimes_k C$ and $\overline{k}$ is a fixed algebraic closure of $k$.  
Let $\mathfrak{b} = \mathrm{div}_C(b)/n \in \mathrm{Div}(C)$.

Suppose $\alpha\in \rH^1(C,\mathbb{Z}/n) = \mathrm{Hom}(\pi_1(C),\mathbb{Z}/n) = \mathrm{Hom}(\mathrm{Pic}(C),\mathbb{Z}/n)$.  Let $\pi_\alpha:C_\alpha \to C$ be the \'etale morphism of smooth projective curves over $k$ associated to $\alpha$.  Define $d_\alpha =[k(C_\alpha):k(C)]$ and let $\alpha'\in \mathrm{Hom}(\mathrm{Gal}(k(C_\alpha)/k(C)),\mathbb{Z}/n)$ be the injection associated to $\pi_\alpha$.

\begin{thm}
\label{thm:MSformula}  
Suppose $\alpha\in \rH^1(C,\mathbb{Z}/n)$  and $b\in D(C)$ are as above. 
\begin{enumerate}
\item[(i)]  Let $\sigma \in \mathrm{Gal}(k(C_\alpha)/k(C))$ be an element such that $\alpha'(\sigma)\in\mathbb{Z}/n$ generates the images of $\alpha'$ and $\alpha$.   There is an element $c \in k(C_\alpha)$ such that $b = \mathrm{Norm}_{k(C_\alpha)/k(C)}(c)$. There is a divisor $\mathfrak{c} \in \mathrm{Div}(C_\alpha)$ such that 
$$\mathrm{div}_{C_\alpha}(c) = \frac{n}{d_\alpha} \pi_\alpha^* (\mathfrak{b}) + (1 - \sigma) \cdot \mathfrak{c}.$$
Writing $[\mathfrak{d}]$ for the class in $\mathrm{Pic}(C)$ of a divisor $\mathfrak{d}$, we have
$$\alpha \cup [b] = \alpha'(\sigma) \cdot \left([\mathrm{Norm}_{k(C_\alpha)/k(C)}(\mathfrak{c})] + \frac{n}{2} [\mathfrak{b}]\right)$$
in $\mathrm{Pic}(C)/n\cdot \mathrm{Pic}(C) = \rH^2(C,\mu_n)$.

\item[(ii)]  Under the assumptions of $\mathrm{(i)}$, suppose $\tau \in \rH^1(C,\mathbb{Z}/n) = \mathrm{Hom}(\mathrm{Pic}(C),\mathbb{Z}/n)$. Then
$$\tau \cup \alpha \cup [b] = \tau\left(\alpha'(\sigma) \cdot\left([\mathrm{Norm}_{k(C_\alpha)/k(C)}(\mathfrak{c})] + \frac{n}{2} [\mathfrak{b}]\right) \right)$$
in $\mathbb{Z}/n = \rH^3(C,\mu_n)$.
\end{enumerate}
\end{thm}

Because cup products are bilinear, to compute them we make the following assumption for the remainder of this introduction.

\begin{hypothesis}
\label{hyp:primepower}
The integer $n=\ell^z$ is a positive power of a prime $\ell$ that is relatively prime to $\#k$. 
\end{hypothesis}

We next define normalized classes in $\rH^1(C,\mathbb{Z}/n)$ and in $\rH^1(C,\mu_n)$.  These arise in Miller's work in \cite{Miller} on the efficient computation of the Weil pairing associated to $\overline{C} = \overline{k} \otimes_k C$.

Since $C$ is geometrically irreducible over $k$,  there is a divisor of  $C$ that has degree $1$ by \cite[Cor. 5, \S VII.5]{Weil}. This implies there must be a closed point $O$ for which the degree $d(O)$ is prime to $\ell\cdot(\ell-1)$. 
An element $\alpha\in \mathrm{Hom}(\mathrm{Pic}(C),\mathbb{Z}/n)$ will be said to be normalized at $O$ if $O$ splits to the curve $C_\alpha$ from part (i) of Theorem \ref{thm:MSformula}. On the other hand, an element $b \in D(C)$ will be said to be normalized at $O$  if the Laurent expansion of $b$  with respect to a uniformizing parameter at $O$ has leading coefficient in $(k(O)^*)^n$.  We will show in Lemma \ref{lem:normalize} that this definition does not depend on the choice of uniformizer.  We denote the $\mathbb{Z}/n$-module of classes $\alpha$ in $\rH^1(C,\mathbb{Z}/n)$ that are normalized  at $O$ by $\rH^1(C,\mathbb{Z}/n)_O$, and we denote the $\mathbb{Z}/n$-module of classes $[b]$ represented by $b \in D(C)$ that are normalized at $O$ by $\rH^1(C,\mu_n)_{O}$.

We now fix a closed point $O$ of $C$ for which the degree $d(O)$ is prime to $\ell\cdot(\ell-1)$.
The structure morphism $C \to \mathrm{Spec}(k)$ induces homomorphisms  $\rH^1(k,\mathbb{Z}/n) \to \rH^1(C,\mathbb{Z}/n)$ and $ \rH^1(k,\mu_n) \to \rH^1(C,\mu_n)$.  
Here $\rH^1(k,\mathbb{Z}/n) = \mathrm{Hom}(\mathrm{Gal}(\overline{k}/k),\mathbb{Z}/n)$ is isomorphic to $\mathbb{Z}/n$ whereas the group $\rH^1(k,\mu_n)$ is isomorphic to $k^*/(k^*)^n$. We will show in Lemma \ref{lem:normalize} that there are direct sum decompositions (\ref{eq:sumonetwo}) for both $\mathcal{F}_n=\mathbb{Z}/n$ and $\mathcal{F}_n=\mu_n$.

We now consider the restriction of the cup product pairing (\ref{eq:cupprodbi}) to these summands.
As noted earlier, the restriction of the arguments of $(\ref{eq:cupprodbi})$ to $\rH^1(k,\mathbb{Z}/n)$ and
$\rH^1(k,\mu_n)$ is trivial since $\rH^2(k,\mu_n) = 0$.

Under Hypothesis \ref{hyp:primepower}, let $T_\ell(C)$ be the $\ell$-adic Tate module of $C$ and let $\Phi_{k,C}$ be the arithmetic Frobenius acting on $T_\ell(C)$.
There exists a unique automorphism $A$ of $\mathbb{Q}_\ell \otimes_{\mathbb{Z}_\ell} T_\ell(C)$ such that $\Phi_{k,C} = 1 + n A $.  Multiplication by $A^{-1}$ defines a homomorphism
\begin{equation}
\label{eq:legendre}
d\mathcal{L}: \mathrm{Pic}^0(C)[n] = \frac{T_\ell(C) \cap A T_\ell(C)}{n A T_\ell(C)} \to \frac{T_\ell(C)}{n T_\ell(C) + nA T_\ell(C)} = \mathrm{Pic}^0(C)/n\cdot\mathrm{Pic}^0(C) 
\end{equation}
which we call the Legendre derivative of Frobenius.
In Theorem \ref{thm:legendre1}, we will show how the Legendre derivative $d\mathcal{L}$ determines the value of $(\ref{eq:cupprodbi})$ when the first argument is restricted to $\rH^1(k,\mathbb{Z}/n)$. Similarly, in Theorem \ref{thm:legendre2}, we will discuss the case when the second argument is restricted to $\rH^1(k,\mu_n)$. 

We finally consider the restriction of both arguments of $(\ref{eq:cupprodbi})$ to normalized classes in $\rH^1(C,\mathbb{Z}/n)_O$ and $\rH^1(C,\mu_n)_O$. 
For all $\alpha \in \rH^1(C,\mathbb{Z}/n)_O$ and $[b] \in \rH^1(C,\mu_n)_{O}$, we define
\begin{equation}
\label{eq:almostWeil}
(\alpha,[b])_{[O]} = \frac{1}{d(O)} \cdot(\overline{\alpha} \cup_{\overline{C}} \overline{[b]}) \cdot   [O] 
\quad \mbox{in}\quad \mathrm{Pic}(C)/n \cdot\mathrm{Pic}(C) = \rH^2(C,\mu_n)
\end{equation}
where $[O]$ is the class in $\mathrm{Pic}(C)$ of the divisor of degree $d(O)$ defined by the point $O$ and $\overline{\alpha} \cup_{\overline{C}} \overline{[b]} \in \mathbb{Z}/n$ is the value of the pairing in $(\ref{eq:ohWeil!})$ when $\overline{\alpha}$ and $\overline{[b]}$ are the restrictions of $\alpha$ and $[b]$ to $\rH^1(\overline{C},\mathbb{Z}/n)$ and $\rH^1(\overline{C},\mu_n)$, respectively.
 
We prove the following result for curves $C$ of arbitrary positive genus.

\begin{thm}
\label{thm:cupsizeresult}  
Suppose $n  = \ell^z$ and that $O$ is a closed point of degree $d(O)$ prime to $\ell\cdot (\ell-1)$.   The following two conditions are equivalent:
\begin{enumerate}
\item[(i)]  For all $\alpha \in \rH^1(C,\mathbb{Z}/n)_O$ and $[b] \in \rH^1(C,\mu_n)_{O}$, we have $\alpha\cup[b]=(\alpha,[b])_{[O]}$ from $(\ref{eq:almostWeil})$.
\item[(ii)] The triple product on normalized classes 
$$\rH^1(C, \mathbb{Z}/n)_O \times  \rH^1(C,\mathbb{Z}/n)_O \times \rH^1(C,\mu_n)_O \to \rH^3(C,\mu_n) =  \mathbb{Z}/n$$
is zero.
 \end{enumerate}
\end{thm}

When the genus of $C$ is one, we show that condition (ii) of Theorem \ref{thm:cupsizeresult} is satisfied. In other words, we prove the following result.

\begin{thm}
\label{thm:EllThm}  
Under the hypotheses of Theorem $\ref{thm:cupsizeresult}$, assume further that $g(C)=1$. Then for all $\alpha \in \rH^1(C,\mathbb{Z}/n)_O$ and $[b] \in \rH^1(C,\mu_n)_{O}$, we have
$$\alpha\cup[b]=(\alpha,[b])_{[O]}\quad \mbox{ from $(\ref{eq:almostWeil})$.}$$
\end{thm}

One key ingredient in the proof of Theorem \ref{thm:EllThm} is to show that when $\alpha,[b],\overline{\alpha} ,\overline{[b]}$ are as in (\ref{eq:almostWeil}) and $g(C)=1$, then $\overline{\alpha} \cup_{\overline{C}} \overline{[b]}=0$ in $\rH^2(\overline{C},\mu_n)$ implies $\alpha\cup[b]=0$ in $\rH^2(C,\mu_n)$; see Theorem \ref{thm:WeilCup}.

In \S \ref{s:genustwo}, we will give infinitely many curves $C$ of genus $2$ for which the equivalent conditions of Theorem \ref{thm:cupsizeresult}  do not hold when $n = \ell = 3$.  The size of the image of the cup product in an analogous number theoretic situation arising from the theory of cyclotomic fields is discussed by McCallum and Sharifi in \cite{MS}.  

\medbreak

We now outline the contents of the sections of this paper. In \S \ref{s:sharififormula}, we will set up the notation and assumptions for the remainder of the paper and we will collect some results on \'etale cohomology groups. Moreover, we will show Theorem \ref{thm:MSformula}. For the remainder of the article, we assume Hypothesis \ref{hyp:primepower}. In \S \ref{s:complements}, we will prove the direct sum decomposition (\ref{eq:sumonetwo}) in Lemma \ref{lem:normalize} and we will show in Theorem \ref{thm:basechange} how to reduce to the case when $d(O)=1$. In \S \ref{s:arithcovers}, we will prove various results that are necessary for analyzing the formulas in Theorem \ref{thm:MSformula}.  In \S \ref{s:restrictedFrobenius}, we will consider restrictions of the cup product maps (\ref{eq:cupprodbi}) and  (\ref{eq:cupprodtri})  that are connected to the derivative of the arithmetic Frobenius. In particular, in Proposition \ref{prop:Aut}, we will introduce the Legendre derivative of Frobenius, and we will prove Theorems \ref{thm:legendre1} and \ref{thm:legendre2}.  In \S \ref{s:arbitrarygenus}, we will consider curves of arbitrary positive genus and we will prove Theorem \ref{thm:cupsizeresult}. In \S \ref{s:ellcurves} and \S \ref{s:cohomological}, we will focus on the genus one case and we will prove Theorem \ref{thm:EllThm} by proving that condition (ii) of Theorem \ref{thm:cupsizeresult} holds in this case.  In \S  \ref{s:genustwo}, we will construct an infinite family of curves of genus two for which the equivalent conditions of  Theorem \ref{thm:cupsizeresult} do not hold when $n = \ell = 3$;  see Theorem \ref{thm:genus2examples}.

\section*{Acknowledgements} 
The authors would like to thank D. Boneh, M. Bright, H. W. Lenstra Jr., R. Sharifi, A. Silverberg and A. Venkatesh for conversations related to this article.

\section{\'Etale cohomology groups}
\label{s:sharififormula}
\setcounter{equation}{0}

Throughout this paper we will assume that $C$ is a smooth projective geometrically irreducible curve
of genus $g(C)\ge 1$ over a finite field $k$ of order $q$. We will suppose $n$ is a positive integer that is relatively prime to $q$. Let $k(C)$ be the function field of $C$ and let $\overline{k(C)}$ be a separable closure of $k(C)$ containing a fixed algebraic closure $\overline{k}$ of $k$.  Let $\mathrm{Div}(C)$ be the divisor group of $C$, let $\mathrm{Pic}(C)$ be the Picard group of $C$, and let $\mathrm{Pic}^0(C)$ be the group of divisor classes of degree $0$. 

Let  $\overline{C}=\overline{k} \otimes_k C$ and let $\eta$ be a geometric point of $\overline{C}$, which can then also be viewed as a geometric point of $C$. We have an exact sequence 
\begin{equation}
\label{eq:fundamental}
1 \to \pi_1(\overline{C},\eta) \to \pi_1(C,\eta) \to \mathrm{Gal}(\overline{k}/k) \to 1
\end{equation}
of \'etale fundamental groups in which $\mathrm{Gal}(\overline{k}/k)$ is isomorphic to the profinite completion $\hat{\mathbb{Z}}$ of $\mathbb{Z}$.  There are natural isomorphisms
\begin{equation}
\label{eq:naturall1}
\rH^1(k,\mathbb{Z}/n) = \mathrm{Hom}(\mathrm{Gal}(\overline{k}/k),\mathbb{Z}/n)
\end{equation}
and 
\begin{equation}
\label{eq:naturall2}
\rH^1(C,\mathbb{Z}/n) = \mathrm{Hom}(\pi_1(C,\eta),\mathbb{Z}/n)\quad\mbox{and}\quad\rH^1(\overline{C},\mathbb{Z}/n) = \mathrm{Hom}(\pi_1(\overline{C},\eta),\mathbb{Z}/n).
\end{equation}
In fact, we have from  \cite[\S2.1.2]{AchingerThesis} (see also \cite[\S3]{Achinger2015}) the following result:

\begin{lemma}
\label{lem:achinger}  
$\mathrm{(Achinger)}$ For all $i\ge 0$, and all locally constant constructible sheaves $\mathcal{F}$ of $\mathbb{Z}/n$-modules, the natural homomorphisms 
$$\rH^i(\pi_1(C,\eta),\mathcal{F}_{\eta}) \to \rH^i(C,\mathcal{F})\quad \mbox{and} \quad
\rH^i(\pi_1(\overline{C},\eta),\mathcal{F}_{\eta}) \to \rH^i(\overline{C},\mathcal{F})$$
are isomorphisms when $\mathcal{F}_{\eta}$ is the stalk of $\mathcal{F}$ at  $\eta$.  
\end{lemma}

Let $\mathcal{F}$ be a constructible sheaf of finite groups of order dividing $n$ on $C$, and let $\overline{\mathcal{F}}$ denote its restriction to $\overline{C}$. The following lemma results directly from the spectral sequence
\begin{equation}
\label{eq:specseq}
\rH^i(\mathrm{Gal}(\overline{k}/k),\rH^j(\overline{C},\overline{\mathcal{F}})) \Rightarrow \rH^{i+j}(C,\mathcal{F})
\end{equation}
together with the fact that $\mathrm{Gal}(\overline{k}/k) \cong \hat{\mathbb{Z}}$ has cohomological dimension one.

\begin{lemma} 
\label{lem:firstgroup}  
One has a split exact sequence of  $\mathbb{Z}/n$-modules
\begin{equation}\label{eq:h1seq}
0 \to \mathrm{Hom}(\mathrm{Gal}(\overline{k}/k),\mathbb{Z}/n) \to \rH^1(C,\mathbb{Z}/n) \to \rH^1(\overline{C},\mathbb{Z}/n)^{\mathrm{Gal}(\overline{k}/k)} \to 0
\end{equation}
in which  $\mathrm{Hom}(\mathrm{Gal}(\overline{k}/k),\mathbb{Z}/n)$ is cyclic of order $n$ and $\rH^1(\overline{C},\mathbb{Z}/n)$ has order  dividing $n^{2g(C)}$.  The sequence $(\ref{eq:h1seq})$ is the $\mathbb{Z}/n$-module dual of the sequence
\begin{equation}
\label{eq:picer}
0 \to \mathrm{Pic}^0(C)/n \cdot\mathrm{Pic}^0(C) \to \mathrm{Pic}(C)/n \cdot\mathrm{Pic}(C) \to \mathbb{Z}/n \to 0
\end{equation}
resulting from the degree map $\mathrm{Pic}(C) \to \mathbb{Z}$ and the Artin map $\mathrm{Pic}(C) \to \pi_1(C,\eta)^{\rab}$.
\end{lemma} 
  
The next lemma gives a description of $\rH^1(C,\mathbb{Z}/n)$ and of $\rH^i(C,\mu_n)$ for $i=1,2$.

\begin{lemma}
\label{lem:secondgroup}
There are natural isomorphisms
\begin{eqnarray}
\label{eq:h11}
\rH^1(C,\mathbb{Z}/n) &=& \mathrm{Hom}(\mathrm{Pic}(C),\mathbb{Z}/n)\quad\mbox{and}\\
\label{eq:h12} 
\rH^2(C,\mu_n) &=& \mathrm{Pic}(C)/n \cdot \mathrm{Pic}(C).
\end{eqnarray}
Define 
$$D(C) = \{b \in k(C)^*:  \mathrm{div}_C(b) \in n \cdot \mathrm{Div}(C)\}.$$
Then there is a  natural isomorphism
\begin{equation}
\label{eq:h13}
\rH^1(C,\mu_n) = D(C)/(k(C)^*)^n.
\end{equation}
Moreover, for $b\in D(C)$, choose an $n^{th}$ root $b^{1/n}$ of $b$ in $\overline{k(C)}$. Then the class $[b]\in D(C)/(k(C)^*)^n$ corresponds to the class of the one-cocycle $c_{b^{1/n}}:\pi_1(C,\eta)\to\tilde{\mu}_n$ defined by $c_{b^{1/n}}(\sigma)=\sigma(b^{1/n})/b^{1/n}$ for all $\sigma\in\pi_1(C,\eta)$.
\end{lemma}

\begin{proof}
The group $\mathrm{Hom}(\mathrm{Pic}(C),\mathbb{Z}/n) $ is identified with $\mathrm{Hom}(\pi_1(C,\eta),\mathbb{Z}/n)$, as in (\ref{eq:h11}), via the Artin map. Since $\rH^2(C,\mathbb{G}_m) = 0$,  $\mathrm{Pic}(C)/n \cdot \mathrm{Pic}(C) $ is identified with $\rH^2(C,\mu_n)$, as in (\ref{eq:h12}), via the Kummer sequence. 

By \cite[p. 125]{Milne}, $\rH^1(C,\mu_n)$ is identified with isomorphism classes of pairs $(L,\phi)$ in which $L$ is a line bundle on $C$ and $\phi$ is an isomorphism $O_C \to L^{\otimes n}$.  Given $b \in D(C)$, define $\mathfrak{b} = \mathrm{div}_C(b)/n$, and let $L_b$ be the line bundle $O_C(\mathfrak{b})$.  We then have an isomorphism $\phi_b:O_C \to L_b^{\otimes n} = O_C(\mathrm{div}_C(b))$  sending the global section $1$ to the global section $b^{-1}$ of $O_C(\mathrm{div}_C(b))$. This induces the isomorphism
$\rH^1(C,\mu_n) \to D(C)/(k(C)^*)^n$ in (\ref{eq:h13}). 

For $b\in D(C)$ and a choosen $n^{\mathrm{th}}$ root $b^{1/n}$ in $\overline{k(C)}$, the map $c_{b^{1/n}}$ satisfies the one-cocycle condition. Since the one-coboundaries are the maps $d_\zeta:\pi_1(C,\eta)\to\tilde{\mu}_n$, for $\zeta\in\tilde{\mu}_n$, with $d_\zeta(\sigma) = \sigma(\zeta)/\zeta$ for all $\sigma\in \pi_1(C,\eta)$, they naturally correspond to $b\in (k(C)^*)^n$.
\end{proof}

We now prove the formula in part (i) of Theorem \ref{thm:MSformula}
by adjusting the arguments of \cite[Thm. 2.4]{MS} in the following way.
We are concerned with the pairing
$$\rH^1(C,\mathbb{Z}/n) \times \rH^1(C,\mu_n) \to \rH^2(C,\mu_n)$$
rather than with the pairing
$$\rH^1(C,\mu_n) \times \rH^1(C,\mu_n) \to \rH^2(C,\mu_n^{\otimes 2})$$
that is the function field counterpart of the one considered in \cite[Thm. 2.4]{MS}.  As in  Theorem \ref{thm:MSformula}, suppose $\alpha \in \rH^1(C,\mathbb{Z}/n)$ and that $b \in D(C)$ has class $[b] \in \rH^1(C,\mu_n)$.  To connect the notation of \cite{MS} to our case, define $m_\sigma$ for $\sigma \in
\mathrm{Gal}(k(C_\alpha)/k(C))$ to be the smallest non-negative integer representing the residue class $\alpha'(\sigma) \in \mathbb{Z}/n$.  The computation of \cite[Lemma 2.2]{MS} holds in our case when we replace $(a,-ab)_S$ in the notation of \cite{MS} by 
\begin{equation}
\label{eq:replaceit}
\alpha \cup (-[b]) + (\alpha,\alpha)
\end{equation}
when $(\alpha,\alpha)$ is defined in the following way. The canonical exact sequence
$$0 \to \mathbb{Z}/n \to \mathbb{Z}/n^2 \to \mathbb{Z}/n \to 0$$
defines a class $\epsilon \in \rH^2(\mathbb{Z}/n,\mathbb{Z}/n)$. Let $\epsilon_\alpha \in \rH^2(C,\mathbb{Z}/n)$
be the pullback of this class via $\alpha:\pi_1(C,\eta) \to \mathbb{Z}/n$.   If $n$ is odd, let $(\alpha,\alpha) = 0$ be the trivial element of $\rH^2(C,\mu_n)$.  If $n$ is even, let
$(\alpha,\alpha) \in \rH^2(C,\mu_n)$ be the image of $\epsilon_\alpha$ under the composition of the map $\mathbb{Z}/n \to \mathbb{Z}/2$ given by multiplication by $n(n-1)/2$ followed by the unique map $\mathbb{Z}/2\to \mu_n$ sending $1$  mod $2$ to $-1 \in \mu_n$.  (We could have replaced $n(n-1)/2$ here simply by $n/2$, but we retain $n(n-1)/2$ to fit with \cite{MS}.)  The computations in \cite[Lemmas 2.2, 2.3]{MS} and \cite[Thm. 2.4]{MS} now carry over to show part (i) of Theorem \ref{thm:MSformula} in our case when we use   (\ref{eq:replaceit}) as a replacement for $(a,-ab)_S$.    

Part (ii) of Theorem \ref{thm:MSformula} now results from part (i) and the next lemma.
 
 \begin{lemma}
 \label{lem:dualer} 
 The cup product pairing
 \begin{equation}
 \label{eq:perfnexter}
 \rH^1(C,\mathbb{Z}/n) \times \rH^2(C,\mu_n) \to \rH^3(C,\mu_n) = \mathbb{Z}/n
 \end{equation}
 is a perfect duality between the groups $\rH^1(C,\mathbb{Z}/n)$ and $\rH^2(C,\mu_n)$.  This pairing
 agrees with the evaluation map
\begin{equation}
\label{eq:naturalnewer}
\mathrm{Hom}(\mathrm{Pic}(C),\mathbb{Z}/n) \times \frac{\mathrm{Pic}(C)}{n \cdot \mathrm{Pic}(C)}\to \mathbb{Z}/n\end{equation}
when we use the identifications in Lemma $\ref{lem:secondgroup}$.
 \end{lemma}

\begin{proof}The first statement is shown in \cite[Cor. II.3.3(b)]{MilneADT}.  The subtlety in verifying that (\ref{eq:perfnexter}) and (\ref{eq:naturalnewer}) agree is to check that the natural evaluation map in (\ref{eq:naturalnewer}) agrees with the canonical class isomorphism $\rH^3(C,\mu_n) = \mathbb{Z}/n$ in (\ref{eq:perfnexter}). 

We will use the following terminology.  Suppose $\chi \in \rH^1(C,\mathbb{Z}/n) = \mathrm{Hom}(\mathrm{Pic}(C),\mathbb{Z}/n) $ and that $x$ is a closed point of $C$ with class $[x] \in \mathrm{Pic}(C)/n\cdot \mathrm{Pic}(C))=  \rH^2(C,\mu_n)$.
We will say that the pairings agree for $\chi$ and $[x]$ if (\ref{eq:perfnexter}) and (\ref{eq:naturalnewer}) agree when we make the identifications in Lemma \ref{lem:dualer}.  

The first case is when $\chi \in \mathrm{Hom}(\pi_1(C,\eta),\mathbb{Z}/n)$ is trivial on $\pi_1(\overline{C},\eta)$ when $\overline{C} = \overline{k} \otimes_k C$.  Under the identification $\rH^1(C,\mathbb{Z}/n) = \mathrm{Hom}(\mathrm{Pic}(C),\mathbb{Z}/n) $, $\chi$ is then a multiple of the map
$\mathrm{deg}:\mathrm{Pic}(C) \to \mathbb{Z}/n$ induced by the degree.  We may thus reduce to the case in which $\chi = \mathrm{deg}$.  Class field theory then identifies $\chi$ with the element of $\rH^1(C,\mathbb{Z}/n) = \mathrm{Hom}(\pi_1(C,\eta),\mathbb{Z}/n)$ sending the coset $\Phi_{\overline{k}/k} \pi_1(\overline{C},\eta)$ to $1 \in \mathbb{Z}/n$ when $\Phi_{\overline{k}/k}$ is the arithmetic Frobenius of 
$G = \mathrm{Gal}(\overline{C}/C) = \mathrm{Gal}(\overline{k}/k)$.  
In the proof of \cite[Corollary V.2.3]{Milne} the Hochschild-Serre spectral sequence produces a diagram 
$$ \xymatrix @C.3pc {
0&\to&\rH^0(\overline{C},\mathbb{Z}/n)_{G}\ar[d]&\to&\rH^1(C,\mathbb{Z}/n)\ar[d]&\to&\rH^1(\overline{C},\mathbb{Z}/n)^{G}\ar[d]&\to&0\\
0&\to&\mathrm{Hom}(\rH^2(\overline{C},\mu_n)^{G},\mathbb{Z}/n)&\to& \mathrm{Hom}(\rH^2(C,\mu_n),\mathbb{Z}/n)&\to&\mathrm{Hom}(\rH^2(\overline{C},\mu_n)_{G},\mathbb{Z}/n) &\to&0
}$$
in which the vertical arrows are isomorphisms and come from duality pairings over $C$ and $\overline{C}$.  Here $\rH^0(\overline{C},\mathbb{Z}/n)_{G} = \mathbb{Z}/n$  is naturally identified with $\rH^1(G,\rH^0(\overline{C},\mathbb{Z}/n))$ by sending $1 \in \mathbb{Z}/n$ to
$\chi = \mathrm{deg}$.  This reduces the proof that the pairings agree for $\chi = \mathrm{deg}$ and the restriction of $[x]$ to showing that the value of
the pairing 
$$\rH^0(\overline{C},\mathbb{Z}/n)_{G} \times \rH^2(\overline{C},\mu_n)^{G} \to \mathbb{Z}/n$$
on $\mathrm{deg} = 1$ in $\rH^0(\overline{C},\mathbb{Z}/n)_{G} = \mathbb{Z}/n$ and the restriction of $ [x] $ to $\rH^2(\overline{C},\mu_n)^{G} = (\mathrm{Pic}(\overline{C})/n\cdot \mathrm{Pic}(\overline{C}))^{G} = \mathbb{Z}/n$ is $\mathrm{deg}(x)$. This is so because of the definition of the canonical isomorphism
$\rH^2(\overline{C},\mu_n) = \mathbb{Z}/n$ in the proof of \cite[Theorem V.2.1(a)]{Milne}. 
  Hence the pairings agree for $\chi$ trivial on $\pi_1(\overline{C},\eta)$ and all $[x]$.

Suppose now that $\chi$ is an arbitrary element of $\rH^1(C,\mathbb{Z}/n)$ and that $x$ is a closed point such that $[k(x):k]$ is prime to $n$. 
There will then be a character $\chi'$ that is trivial on $\pi_1(\overline{C},\eta)$ and has the same restriction as $\chi$ to every decomposition group associated to $x$ in $\pi_1(C,\eta)$.  Now $\chi - \chi'$
has trivial restriction to all such decomposition groups. Hence the value of the pairing 
in (\ref{eq:naturalnewer}) on $\chi$ and $[x]$ agrees with the value of this pairing on $\chi'$ and $[x]$.  We will show below the same is true for the pairing in (\ref{eq:perfnexter}). 
 Before proving this, we note that this will imply the pairings agree for $\chi$ and $[x]$ because we have shown they agree for $\chi'$ and arbitrary $[x]$.
 
 One way to prove the above claim about (\ref{eq:perfnexter}) is to use the spectral sequence
 $$\rH^p(C,R^q j_*\mu_{n,\xi}) \Rightarrow \rH^{p+q}(\xi,\mu_{n,\xi})$$ and the exact sequence
 \begin{equation}
\label{eq:okfour}
0 \to \mu_{n,C} \to j_* \mathbb{G}_{m,\xi} \to j_* \mathbb{G}_{m,\xi} \to R^1 j_* \mu_{n,\xi} \to 0
\end{equation}
associated to the inclusion $j: \xi =  \mathrm{Spec}(k(C)) \to C$ of the generic point $\xi$ of $C$ into $C$.  As in \cite[\S 5]{BCGKPT}, this gives an exact sequence
\begin{equation}
\label{eq:exactnice}
\rH^1(\xi,\mu_{n,\xi}) \to \rH^0(C,R^1 j_* \mu_{n,\xi}) \xrightarrow{\omega} \rH^2(C,\mu_{n,C}) \to 0
\end{equation} 
in which the transgression $\omega$  is a double boundary map associated to splitting (\ref{eq:okfour}) into two exact sequences.  In \cite[Cor. 5.3]{BCGKPT} it is shown how (\ref{eq:exactnice}) is naturally identified with the sequence
\begin{equation}
\label{eq:expl}
k(C)^*/(k(C)^*)^n \to \bigoplus_{x \in C^0} k(C)_x^*/T_x \xrightarrow{\omega} \rH^2(C,\mu_{n,C}) \to 0
\end{equation}
in which $C^0$ is the set of closed points of $C$ and $T_x$ for $x \in C^0$ is the subgroup of elements of
$k(C)_x^*$ having valuation divisible by $n$.  Since boundaries in spectral sequences are compatible with pairings, we see that (\ref{eq:perfnexter})
arises from (\ref{eq:exactnice}) together with the natural pairing
 $$\rH^1(C, \mathbb{Z}/n) \times \rH^0(C,R^1 j_* \mu_{n,\xi}) \to \rH^1(C,\mathbb{Z}/n \otimes R^1 j_* \mu_{n,\xi}) = \rH^1(C,R^1 j_* \mu_{n,\xi})$$
 and the transgression map
 \begin{equation}
 \label{eq:transg}
 \rH^1(C,R^1 j_* \mu_{n,\xi}) \to \rH^3(C,\mu_{n,C})
 \end{equation}
 associated to the  double boundary map in cohomology associated to (\ref{eq:okfour}).  As in the last paragraph of the proof of \cite[Lemma 5.2]{BCGKPT}, we have an isomorphism
 $$R^1 j_* \mu_{n,\xi} \cong \bigoplus_{x \in C^0} i_{x,*} i_x^* R^1 j_* \mu_{n,\xi} $$
where $i_x:x \to C$ is the closed immersion associated to $x \in C^0$. This reduces us to considering the pairing
\begin{equation}
\label{eq:Shappre}
\rH^1(C, \mathbb{Z}/n) \times \rH^0(C,i_{x,*} i_x^* R^1 j_* \mu_{n,\xi}) \to \rH^1(C,\mathbb{Z}/n \otimes i_{x,*} i_x^* R^1 j_* \mu_{n,\xi}) = \rH^1(C,i_{x,*} i_x^* R^1 j_* \mu_{n,\xi}).
\end{equation}
Pick a decomposition group $\pi_1(C,\eta)_x$ for $x$ in $\pi_1(C,\eta)$.  The sheaf
$i_x^* R^1 j_* \mu_{n,\xi}$ on $x$ is then associated to a module $M_x$ for $\pi_1(C,\eta)_x$.
The $\pi_1(C,\eta)$-module associated to $i_{x,*} i_x^* R^1 j_* \mu_{n,\xi}$ is  isomorphic the induction of $M_x$ from $\pi_1(C,\eta)_x$ to $\pi_1(C,\eta)$.   We want to show that the value of the pairing (\ref{eq:perfnexter}) on the pair $\chi$ and $[x]$ is the same as the value on $\chi'$ and $[x]$ because $\chi$ and $\chi'$ have the same restriction to every decomposition group associated to $x$.   This follows from Shapiro's Lemma applied to 
(\ref{eq:Shappre}).

We have shown that the pairings (\ref{eq:perfnexter}) and (\ref{eq:naturalnewer}) agree on every pair $\chi$ and $[x]$ for which $[k(x):k]$ has degree prime to $n$.  To finish the proof, we now show that the subgroup $T$ of \hbox{$\mathrm{Pic}(C)/n\cdot \mathrm{Pic}(C)$} generated by classes $[x]$ defined by $x$ with $[k(x):k]$ prime to $n$ is all of \hbox{$\mathrm{Pic}(C)/n\cdot \mathrm{Pic}(C)$}. 
Let $C' \to C$ be the finite abelian unramified cover of  $C$ associated to $T$.  Since there are divisors of degree $1$ on $C$, there are points with degrees that are prime to $n$.  Hence since $C$ was assumed to be geometrically irreducible with constant field $\mathbb{F}_q$ the same is true for $C'$.  By the definition of $T$,  every closed point $x \in C$ with $[k(x):k]$ prime to $n$ splits in $C'$.  Thus if $m$ is any integer prime to $n$, we have
$\#C'(\mathbb{F}_{q^m}) = d \cdot \# C(\mathbb{F}_{q^m})$ when $d$ is the covering degree of $C'$ over $C$.  Since $C$ and $C'$ are geometrically irreducible with the same constant field, the Weil conjectures force $d = 1$.  Hence $T = 
\mathrm{Pic}(C)/n\cdot \mathrm{Pic}(C)$ as required.
 \end{proof}

\section{Complements of classes coming from $k$}
\label{s:complements}

Throughout this section we will assume Hypothesis \ref{hyp:primepower}, i.e. $n=\ell^z$ for a prime $\ell$.  When
analyzing cup product pairings it will be useful to have a complement for the image of the inflation homomorphisms $\rH^1(k,\mathbb{Z}/n) \to \rH^1(C,\mathbb{Z}/n)$ and  $\rH^1(k,\mu_n) \to \rH^1(C,\mu_n)$. Let $k(O)$ be the residue field of a closed point $O$ of $C$, and let $d(O) = [k(O):k]$.

\begin{dfn}  
\label{def:normalizedone}
Suppose $O$ is a closed point of $C$ with $d(O)$ prime to $\ell \cdot (\ell -1)$.   A class  $\alpha\in \rH^1(C,\mathbb{Z}/n)$ will be said to be  normalized at $O$ if $O$ splits in the cyclic cover $C_\alpha \to C$ associated to $\alpha$. Let $\rH^1(C,\mathbb{Z}/n)_O$ be the subgroup of all such $\alpha$. Let $\pi_O$ be a uniformizing parameter in the local ring $O_{C,O}$.  An element $b\in D(C)$ will be said to be normalized at $O$ with respect to $\pi_O$ if the leading term in its Laurent expansion with respect to $\pi_O$ lies in $(k(O)^*)^n$.  A class in $\rH^1(C,\mu_n)$ will be said to be  normalized at $O$ with respect to $\pi_O$ if it has the form $[b]$ for an element $b$ of this kind. Let $\rH^1(C,\mu_n)_O$ be the subset of all such $[b]$.  
\end{dfn} 

\begin{lemma}
\label{lem:normalize}
Assume Hypothesis $\ref{hyp:primepower}$.
There is a closed point $O$ of $C$ with $d(O)$ prime to $\ell \cdot (\ell -1)$. 
\begin{enumerate}
\item[(i)] There is a direct sum decomposition 
$$\rH^1(C,\mathbb{Z}/n) = \rH^1(C,\mathbb{Z}/n)_O \oplus \rH^1(k,\mathbb{Z}/n)$$ 
where $\rH^1(k,\mathbb{Z}/n) \cong \mathbb{Z}/n$ and the restriction map  sends $\rH^1(C,\mathbb{Z}/n)_{O}$ isomorphically to $\rH^1(\overline{C},\mathbb{Z}/n)^{\mathrm{Gal}(\overline{k}/k)}$.
\item[(ii)]  The subgroup $\rH^1(C,\mu_n)_{O}$ depends on $O$ but  does not depend on the choice of uniformizer $\pi_O$ at $O$.  There is a direct sum decomposition 
$$\rH^1(C,\mu_n) = \rH^1(C,\mu_n)_{O} \oplus \rH^1(k,\mu_n)$$ 
where $\rH^1(k,\mu_n) \cong k^*/(k^*)^n$ and  the restriction map sends $\rH^1(C,\mu_n)_{O}$ isomorphically to $\rH^1(\overline{C},\mu_n)^{\mathrm{Gal}(\overline{k}/k)}$.
\end{enumerate}
\end{lemma}

\begin{proof}  
Since $C$ is geometrically irreducible over $k$,  there is a divisor of  $C$ that has degree $1$ by \cite[Cor. 5, \S VII.5]{Weil}, which implies there must a point $O$ for which $d(O)$ is prime to $\ell\cdot (\ell - 1)$.   

To prove (i), suppose $\alpha \in \rH^1(C,\mathbb{Z}/n)$.  Let $\pi_1(C,\eta)_O$ be a decomposition group associated to $O$ in $\pi_1(C,\eta)$.  Then $\pi_1(C,\eta)_O$ is procyclic, and there is a progenerator $\Phi$ of $\pi_1(C,\eta)_O$ that maps to $\Phi_{\overline{k}/k}^{d(O)}$ in $\mathrm{Gal}(\overline{k}/k)$.  Since $d(O)$ is prime to $\ell \cdot (\ell -1)$ and $n = \ell^z$, there is a character $\rho \in \rH^1(k,\mathbb{Z}/n) = \mathrm{Hom}(\mathrm{Gal}(\overline{k}/k), \mathbb{Z}/n)$ such that $\rho(\Phi_{\overline{k}/k}^{d(O)}) = \alpha(\Phi)$.  Denote by $\rho$ also the inflation of $\rho$ to $\rH^1(C,\mathbb{Z}/n)  = \mathrm{Hom}(\pi_1(C,\eta),\mathbb{Z}/n)$ via the natural surjection $\pi_1(C,\eta) \to \mathrm{Gal}(\overline{k}/k)$.  Then $\tilde{\alpha} = \alpha - \rho$ has the property that $\tilde{\alpha}(\Phi) = 0$, so $\tilde{\alpha} \in \rH^1(C,\mathbb{Z}/n)_O$ and $\alpha = \tilde{\alpha} + \rho$.  Thus $\rH^1(C,\mathbb{Z}/n)_O $ and $ \rH^1(k,\mathbb{Z}/n)$ together generate 
$\rH^1(C,\mathbb{Z}/n)$.  If $\alpha \in \rH^1(C,\mathbb{Z}/n)_O \cap \rH^1(k,\mathbb{Z}/n)$ then $C_\alpha \to C$ results from a constant field extension of $k$ of degree equal to the order of $\alpha$ and dividing $n = \ell^z$, and $O$ splits in this cover.  However $d(O) = [k(O):k]$ has degree prime to $n$, so this cover must be trivial and $\alpha$ is trivial.  This establishes the direct sum decomposition  $\rH^1(C,\mu_n) = \rH^1(C,\mu_n)_{O} \oplus \rH^1(k,\mu_n)$ in part (i).  The last statement in part (i) follows from the exact sequence (\ref{eq:h1seq}).

We now prove part (ii).
If $b \in D(C)$ then $\mathrm{ord}_{O}(b)$ is a multiple of $n$. Therefore if one replaces $\pi_O$ by another uniformizing parameter $\pi_O'$ at $O$, the leading terms in the Laurent expansions of $b$ with respect to $\pi_O$ and $\pi_O'$ differ by an element of $(k(O)^*)^n$.  Hence $\rH^1(C,\mu_n)_{O}$ does not depend on the choice of $\pi_O$.  
Since $d = d(O)$ is prime to $\ell\cdot (\ell -1)$, one sees that the ratio $\# k(O)^* / \# k^* = (q^d - 1)/(q-1) = 1 + q + \cdots + q^{d-1}$ is prime to $\ell$ by considering first the case in which $q \equiv 1$ mod $\ell$ and then the case in which $q \not \equiv 1$ mod $\ell$.  Hence the Sylow $\ell$-subgroups of $k^*$ and $k(O)^*$ are the same.  Therefore every $b \in D(C)$ is equal to $\tilde{b} \cdot s$ for some $\tilde{b}$ that is normalized at $O$ and some $s \in k^*$.  The image of $s$ in $\rH^1(k,\mu_n) = k^*/(k^*)^n = k(O)^*/(k(O)^*)^n$ is uniquely determined by the image $[b]$ of $b$ in $\rH^1(C,\mu_n) = D(C)/(k(C)^*)^n$.  It follows that we have a direct sum decomposition $\rH^1(C,\mu_n) = \rH^1(C,\mu_n)_{O} \oplus \rH^1(k,\mu_n)$.    The last statement in part (ii) follows from the spectral sequence (\ref{eq:specseq}) with $\mathcal{F} = \mu_n$.
\end{proof}

\begin{cor}
\label{lemcor:reduce} 
Assume Hypothesis $\ref{hyp:primepower}$, and let $O$ be a closed point of $C$ with $d(O)$ prime to $\ell\cdot (\ell -1)$. Every element $\alpha$ of $\rH^1(C,\mathbb{Z}/n)$ has a unique expression as a sum $\tilde{\alpha} + \rho$ with $\tilde{\alpha} \in \rH^1(C,\mathbb{Z}/n)_O$ and $\rho \in H^1(k,\mathbb{Z}/n)$. Every element $[b]$ of $\rH^1(C,\mu_n)$ has a unique expression as a product $[\tilde{b}] \cdot [s]$ with $[\tilde{b}] \in \rH^1(C,\mu_n)_O$ normalized at $O$ and $s \in k^*$, so $[s] \in H^1(k,\mu_n)$.    Then 
\begin{equation}
\label{eq:formula}
\alpha \cup [b]  = \tilde{\alpha} \cup [\tilde{b}] + \rho \cup [\tilde{b}] + \tilde{\alpha} \cup [s]  
\end{equation}
\end{cor}

\begin{proof}  
This is clear from the fact that cup products are bilinear and anti-commutative, and $\rho \cup [s] = 0 $ since $\rho \cup [s]$ is the inflation of a class in $\rH^2(k,\mu_n) = 0$ to $\rH^2(C,\mu_n)$.
\end{proof}

\begin{rem} 
Corollary $\ref{lemcor:reduce}$ reduces the computation of cup products of elements of $\rH^1(C,\mathbb{Z}/n)$ and $\rH^1(C,\mu_n)$ to two cases:  (i) both arguments are normalized classes with respect to some choice of closed point $O$ with $d(O)$ prime to $\ell\cdot (\ell -1)$, or (ii) one argument is normalized and the other argument is in $\rH^1(k,\mu_n)$ or $\rH^1(k,\mathbb{Z}/n)$.  
\end{rem} 

We will need the following characterizations of $\rH^1(C,\mathbb{Z}/n)_O$ and $\rH^1(C,\mu_n)_O$.

\begin{lemma}
\label{lem:normalizedcharacter}
Let $\overline{k(C)}$ be a separable closure of $k(C)$ containing $\overline{k}$, and let $M(C)$ be the maximal everywhere unramified extension of $k(C)$ inside $\overline{k(C)}$. Then $\pi_1(C,\eta)=\mathrm{Gal}(M(C)/k(C))$. Let $O$ be a closed point of $C$ with $d(O)$ prime to $\ell\cdot (\ell -1)$, let $\tilde{O}$ be a place in $M(C)$  over $O$, and let $\pi_1(C,\eta)_{\tilde{O}}$ be the decomposition group of the place $\tilde{O}$  in $\pi_1(C,\eta)$. Then $\pi_1(C,\eta)_{\tilde{O}}$ is procyclic, and there is a progenerator $\Phi$ of $\pi_1(C,\eta)_{\tilde{O}}$ that maps to $\Phi_{\overline{k}/k}^{d(O)}$ in $\mathrm{Gal}(\overline{k}/k)$.  
\begin{enumerate}
\item[(i)] Let $\alpha\in\rH^1(C,\mathbb{Z}/n) = \mathrm{Hom}(\pi_1(C,\eta),\mathbb{Z}/n)$. Then $\alpha\in\rH^1(C,\mathbb{Z}/n)_O$ if and only if $\alpha(\Phi)=0$. We have a direct sum decomposition 
$$\frac{\mathrm{Pic}(C)}{n\cdot \mathrm{Pic}(C)} = \langle [O]\rangle \oplus \frac{\mathrm{Pic}^0(C)}{n \cdot \mathrm{Pic}^0(C)}$$
where $[O]$ is the image of the class in $\mathrm{Pic}(C)$ of the divisor of degree $d(O)$ defined by $O$. Then $\langle [O]\rangle\cong\mathbb{Z}/n$, and we have that $\alpha\in\rH^1(C,\mathbb{Z}/n) = \mathrm{Hom}(\mathrm{Pic}(C),\mathbb{Z}/n)$ is normalized at $O$ if and only if $\alpha([O]) = 0$.
\item[(ii)] Suppose $b\in D(C)$. Then $b$ is normalized at $O$ if and only if there exists an element $b^{1/n}$ in the completion $k(C)_O$ of the function field $k(C)$ at $O$ such that $(b^{1/n})^n=b$. Let $L$ be the extension of $k(C)$ generated  by all such $b^{1/n}$ as $b$ ranges over all elements of $D(C)$ that are normalized at $O$.  Then $L \subset M(C)$, and there is a place $O'$ over $O$ with $L_{O'} = k(C)_O$.  Let  $c_{b^{1/n}}\in C^1(\pi_1(C,\eta),\tilde{\mu}_n)$ be the one-cocycle defined by $c_{b^{1/n}}(\sigma)= \sigma(b^{1/n})/b^{1/n}$ for $\sigma\in\pi_1(C,\eta)$.  Then $c_{b^{1/n}}$ represents $[b] \in \rH^1(C,\mu_n)_O$, and $c_{b^{1/n}}$ is trivial on the decomposition group $\pi_1(C,\eta)_{\tilde{O}}$ of a place $\tilde{O}$ of $M(C)$ over $O'$. In particular, $c_{b^{1/n}}(\Phi)=c_{b^{1/n}}(\Phi^{-1})=1$.  
\end{enumerate}
\end{lemma}

\begin{proof}
The first paragraph of Lemma \ref{lem:normalizedcharacter} is obvious.

To prove part (i), let $\pi_\alpha:C_\alpha \to C$ be the \'etale morphism of smooth projective curves over $k$ associated to $\alpha$. Then $O$ splits in the cyclic cover $\pi_\alpha$ if and only if $\alpha$ is trivial on the decomposition group $\pi_1(C,\eta)_{\tilde{O}}$ of any place $\tilde{O}$ in $M(C)$ over $O$. This implies the first statement of part (i) and the direct sum decomposition of $\mathrm{Pic}(C)/n\cdot \mathrm{Pic}(C)$. Since $d(O)$ is prime to $\ell$, $\langle [O]\rangle\cong\mathbb{Z}/n$, and we obtain the last statement of part (i).

To prove part (ii), let $b\in D(C)$ and choose a uniformizing parameter $\pi_O$ in the local ring $O_{C,O}$.  Then the leading term in the Laurent series expansion of $b$ with respect to $\pi_O$ lies in $(k(O)^*)^n$ if and only if there exists an $n^{\mathrm{th}}$ root $b^{1/n}$ of $b$ in the completion $k(C)_O$. The field $L$ described in part (ii) is then contained in $k(C)_O$, so there is a place $O'$ over $O$ in $L$ with $L_{O'} = k(C)_O$.  Since $b \in D(C)$ has valuation divisible by $n$ at all places of $k(C)$ and $n$ is prime to the characteristic of $k(C)$, $L$ is contained in $M(C)$.  When $\tilde{O}$ is a place of $M(C)$ over $O'$, the decomposition group $\pi_1(C,\eta)_{\tilde{O}}$ fixes $k(C)_{O} = L_{O'} \supset L$, so $c_{b^{1/n}}$ is trivial on $\pi_1(C,\eta)_{\tilde{O}}$ if $b \in D(C)$ is normalized at $O$. 
\end{proof}

We will need the following result later to reduce to the case in which $O$ has residue field $k$. 

\begin{thm}
\label{thm:basechange}
Under the assumptions of Corollary $\ref{lemcor:reduce}$, let $k' = k(O)$ be the residue field of $O$, so $[k':k]=d(O)$ is relatively prime to $\ell\cdot (\ell-1)$.
Let $\pi: C' = k' \otimes_k C \to C$ be the second projection.  The direct image homomorphism
\begin{equation}
\label{eq:directimg}
\pi_*:\rH^2(C',\mu_n) = \mathrm{Pic}(C') \otimes_{\mathbb{Z}} \mathbb{Z}/n \to \rH^2(C,\mu_n) = \mathrm{Pic}(C) \otimes_{\mathbb{Z}} \mathbb{Z}/n
\end{equation}
is induced by the norm $\mathrm{Norm}_{C'/C}:\mathrm{Pic}(C') \to \mathrm{Pic}(C)$. Let $O'$ be a point of $C'$ over $O$. Then $O'$ is a point of $C'(k')$ and $\mathrm{Norm}_{C'/C}(O') = O$.     Let  $\alpha' \in \rH^1(C',\mathbb{Z}/n)$ and $ [b]' \in \rH^1(C',\mu_n)$ be the pullbacks of $\alpha  \in \rH^1(C,\mathbb{Z}/n)$ and
$ [b] \in \rH^1(C,\mu_n)$.  One has
\begin{equation}
\label{eq:normit}
\alpha \cup [b] = \frac{1}{d(O)} (\mathrm{Norm}_{C'/C} \otimes \mathrm{Id}) (\alpha' \cup [b]').
\end{equation}
where $\mathrm{Id}$ is the identity map on $\mathbb{Z}/n$. If $\alpha \in \rH^1(C,\mathbb{Z}/n)_{O}$ and $[b] \in \rH^1(C,\mu_n)_{O}$ then $\alpha' \in 
\rH^1(C',\mathbb{Z}/n)_{O'}$ and $ [b]' \in \rH^1(C',\mu_n)_{O'}$.
\end{thm}

\begin{proof}
The  claim about (\ref{eq:directimg}) follows from the compatibility of the Kummer sequences of $C'$ and $C$ with respect to $\pi_*$.  The point $O$ splits to $C'$ so $\mathrm{Norm}_{C'/C}(O') = O$.  The composition $\pi_* \circ \pi^*$ on every cohomology group $\rH^i(C,\mathbb{Z}/n)$ and $\rH^i(C,\mu_n)$ is multiplication by the degree $d(O)$ of $C'$ over $C$.  By the compatibility of cup products with pullbacks this gives 
$$d(O) \cdot (\alpha \cup [b] ) = (\pi_* \circ \pi^*) (\alpha \cup [b] )  = \pi_* (\alpha' \cup [b']) = (\mathrm{Norm}_{C'/C} \otimes \mathrm{Id}) (\alpha' \cup [b]')$$
which shows (\ref{eq:normit}) since $d(O)$ is prime to $\ell$.  The fact that $\alpha' \in \rH^1(C',\mathbb{Z}/n)_{O'}$ if $\alpha \in \rH^1(C,\mathbb{Z}/n)_{O}$ follows from the fact that $O'$ and $O$ have the same residue field.  The fact that 
$[b]' \in \rH^1(C',\mu_n)_{O'}$ if $[b] \in \rH^1(C,\mu_n)_{O}$ follows from comparing Laurent expansions at $O$ and at $O'$.
\end{proof}

\section{The arithmetic of covers}
\label{s:arithcovers}
\setcounter{equation}{0}

In this section we will prove various results necessary for analyzing the formulas in  Theorem \ref{thm:MSformula}. We will assume Hypothesis \ref{hyp:primepower}, i.e. $n=\ell^z$.
Recall that $\Phi_{\overline{k}/k}$ is the arithmetic Frobenius of $\overline{k}$ over $k$.  As before, we fix a separable closure $\overline{k(C)}$ of $k(C)$ containing a fixed algebraic closure $\overline{k}$ of $k$.  

Define $\overline{k}(C)$ to be the compositum of $\overline{k}$ and  $k(C)$ in $\overline{k(C)}$.  Define $L(C)$ to be the maximal abelian  unramified extension of $\overline{k}(C)$ in $\overline{k(C)}$.  Then $\Phi_{\overline{k}/k}$ is a progenerator of $\mathrm{Gal}(\overline{k}(C)/k(C)) = \mathrm{Gal}(\overline{k}/k) \cong \hat{\mathbb{Z}}$ and $\mathrm{Gal}(L(C)/\overline{k}(C))$ is isomorphic to the adelic Tate module $T_{\mathbb{A}}(C) = \prod_{\ell'} T_{\ell'}(C)$, where $\ell'$ runs over all primes and $T_{\ell'}(C)$ is the $\ell'$-adic Tate module of $C$.  
The left conjugation action on $\mathrm{Gal}(L(C)/\overline{k}(C))$ of a lift $\tilde{\Phi}_{\overline{k}/k}$ of $\Phi_{\overline{k}/k}$ to $\mathrm{Gal}(L(C)/k(C))$  gives an automorphism $\Phi_{k,C}$ of $T_{\mathbb{A}}(C)$ independent of the choice of $\tilde{\Phi}_{\overline{k}/k}$.   The choice of $\tilde{\Phi}_{\overline{k}/k}$ gives an isomorphism
\begin{equation}
\label{eq:bigiso}
\mathrm{Gal}(L(C)/k(C)) = T_{\mathbb{A}}(C)  \rtimes \mathrm{Gal}(\overline{k}/k) .
\end{equation}
The Artin map then defines an injective homomorphism
\begin{equation}
\label{eq:artdef}
\mathrm{art}_C: \mathrm{Pic}(C) \to \mathrm{Gal}(L(C)/k(C))^{\rab} = T_{\mathbb{A}}(C)/(1 - \Phi_{k,C})T_{\mathbb{A}}(C) \times \mathrm{Gal}(\overline{k}/k).
\end{equation}
The image of this homomorphism consists of the elements which project to an integral power of $\Phi_{\overline{k}/k}$ in $\mathrm{Gal}(\overline{k}/k)$.
The restriction of the Artin map defines an isomorphism
\begin{equation}
\label{eq:artdef0}
\mathrm{art}^0_C:\mathrm{Pic}^0(C) \to   T_{\mathbb{A}}(C)/(1 - \Phi_{k,C})T_{\mathbb{A}}(C).
\end{equation}
Here $\mathrm{Pic}^0(C)$ is a finite group, and on the Sylow $\ell$-subgroup we get an isomorphism
\begin{equation}
\label{eq:artdef0ell}
\mathrm{art}^0_C:\mathrm{Pic}^0(C)[\ell^\infty] \to   T_{\ell}(C)/(1 - \Phi_{k,C})T_{\ell}(C)
\end{equation}
where $T_\ell(C)$ is isomorphic to $(\mathbb{Z}_\ell)^{2g(C)}$ when $g(C)$ is the genus of $C$ over $k$ since $\ell$ is prime to $q = \#k$.

Let $C'$ be a smooth projective curve such that $k(C')$ is a cyclic everywhere unramified extension of $k(C)$ of degree $d\big|n$. Let $k'$ be the constant field of $C'$, so $[k':k]\big|d$ and $k \subset k' \subset \overline{k}$.    We have $\overline{k}(C) \subset \overline{k}(C')$ and $L(C) \subset L(C') $.   We get a commutative diagram 
\begin{equation}
\label{eq:origin}
 \xymatrix @C.5pc {
\mathrm{Gal}(L(C')/k'(C'))\ar[d]&=&T_{\mathbb{A}}(C') \rtimes \mathrm{Gal}(\overline{k}/k')\ar[d]^{\overline{\pi}_{*}\  \rtimes \ \iota}\\
\mathrm{Gal}(L(C)/k(C))&=& T_{\mathbb{A}}(C) \rtimes  \mathrm{Gal}(\overline{k}/k)
}
\end{equation}
in which the left vertical homomorphism results from restricting automorphisms of $L(C')$ to $L(C)$,  $\overline{\pi}_{*}$ is induced by the morphism $\overline{\pi}:{\overline{k}}\otimes_{k'} C'  = \overline{C'} \to {\overline{k}}\otimes_{k} C = \overline{C}$ associated to $\overline{k}(C) \subset \overline{k}(C') $ and $\iota: \mathrm{Gal}(\overline{k}/k') \subset \mathrm{Gal}(\overline{k}/k)$ is the natural inclusion. By choosing compatible lifts of $\Phi_{\overline{k}/k'}$ and $\Phi_{\overline{k}/k}$, we can assume $\iota(\Phi_{\overline{k}/k'}) = \Phi_{\overline{k}/k}^{[k':k]}$.  For $y \in T_\mathbb{A}(C')$ we have
$$\overline{\pi}_*(\Phi_{k',C'}(y)) = \Phi_{k,C}^{[k':k]}(\overline{\pi}_* y).$$

Taking maximal abelian quotients of the groups in this diagram gives a commutative diagram
\begin{equation}
\label{eq:oomphah}
 \xymatrix @C.25pc {
\mathrm{Gal}(L(C')/k'(C'))^{\rab}  \ar[d]_{\mathrm{\mathcal{N}}} &=& \displaystyle \frac{T_{\mathbb{A}}(C')}{(1 - \Phi_{k',C'}) T_\mathbb{A}(C')} \times \mathrm{Gal}(\overline{k}/k') &&& \mathrm{Pic}(C') \ar[d]^{\mathrm{Norm}_{C'/C}} \ar[lll]_(.25){\mathrm{art}_{C'}}\\
\mathrm{Gal}(L(C)/k(C))^{\rab} &=& \displaystyle \frac{T_\mathbb{A}(C)}{(1 - \Phi_{k,C}) T_\mathbb{A}(C)} \times  \mathrm{Gal}(\overline{k}/k) &&&  \mathrm{Pic}(C)  \ar[lll]_(.25){\mathrm{art}_{C}}
}
\end{equation}
where $\mathcal{N}$ is induced by the previously described homomorphism $\overline{\pi}_* \ \rtimes\  \iota$ of diagram (\ref{eq:origin}).  

We must now analyze transfer maps.  Recall that if $H$ is a finite index subgroup of a group $G$, the transfer homomorphism $\mathrm{Ver}:G^{\rab} \to \rH^{\rab}$ is defined in the following way.  Let $\{x_i\}_i$ be a set of right coset representatives for $H$ in $G$.  For $\gamma \in G$, write $x_i \gamma  = h_i(\gamma) x_j$ for some index $j$ depending on $i$ and $\gamma$ and some $h_i(\gamma) \in H$.  Then $\mathrm{Ver}$ sends the image $[\gamma]$ of $\gamma$ in $G^{\rab}$ to the image in $\rH^{\rab}$ of $\prod_i h_i(\gamma)$.  We need to analyze this map when $G = \mathrm{Gal}(L(C')/k(C))$ and $H = \mathrm{Gal}(L(C')/k'(C'))$. Note that $G^{\rab} =  \mathrm{Gal}(L(C)/k(C))^{\rab}$ in this case.

\begin{dfn}
\label{def:firstTransfer}
Suppose $\overline{k}\otimes_{k'} C' = \overline{k} \otimes_k C$, so that $C' = k' \otimes_k C$ is a constant field cover of $C$.  We then have  $T_{\mathbb{A}}(C') = T_{\mathbb{A}}(C)$.  With this identification, define $V:T_{\mathbb{A}}(C) \to T_{\mathbb{A}}(C') = T_{\mathbb{A}}(C)$ by 
\begin{equation}
\label{eq:Vdef}
V = \sum_{i = 0}^{d - 1} \Phi_{k,C}^{i} 
\end{equation}
when we write the group law of $T_{\mathbb{A}}(C)$ additively, where $d = [k(C'):k(C)] = [k':k]$.  
\end{dfn}

\begin{lemma}
\label{lem:upfront}
Suppose the hypotheses of Definition $\ref{def:firstTransfer}$. There is a commutative diagram 
\begin{equation}
\label{eq:oomph}
 \xymatrix @C.25pc {
T_\mathbb{A}(C) \ar[d]_{V} \ar[rrr] &&& \mathrm{Gal}(L(C)/k(C))^{\rab}  \ar[d]_{\mathrm{Ver}} &=& \displaystyle \frac{T_{\mathbb{A}}(C)}{(1 - \Phi_{k,C}) T_{\mathbb{A}}(C)} \times \mathrm{Gal}(\overline{k}/k) &&& \mathrm{Pic}(C)\ar[d]^{\pi^*} \ar[lll]_(.25){\mathrm{art}_C}\\
T_{\mathbb{A}}(C') \ar[rrr] &&&\mathrm{Gal}(L(C')/k'(C'))^{\rab} &=& \displaystyle \frac{T_{\mathbb{A}}(C')}{(1 - \Phi_{k',C'}) T_{\mathbb{A}}(C')} \times  \mathrm{Gal}(\overline{k}/k')  &&&  \mathrm{Pic}(C') \ar[lll]_(.225){\mathrm{art}_{C'}}
}
\end{equation}
\end{lemma} 

\begin{proof}  
This is just a matter of unwinding the definition of the transfer homomorphism when $G = \mathrm{Gal}(L(C')/k(C))$ and $H = \mathrm{Gal}(L(C')/k'(C'))$.  Since $C' = k' \otimes_k C$, we can choose the set of coset representatives for $H$ in $G$ to be $\{\tilde{\Phi}_{\overline{k}/k}^{i}\}_{i = 0}^{d -1}$.  
\end{proof}  

\begin{cor} 
\label{cor:whatever} 
Suppose the hypotheses of Definition $\ref{def:firstTransfer}$. There is a commutative diagram
\begin{equation}
\label{eq:moreso}
 \xymatrix {
T_\ell(C) \;\ar[r] \ar[d]_{V} &\;  \displaystyle \frac{T_\ell(C)}{(1 - \Phi_{k,C}) T_\ell(C)}\; & \;\mathrm{Pic}^0(C)[\ell^\infty] \ar[d]^{\pi^* = \mathrm{Ver}} \ar[l]_(.4){\mathrm{art}_C^0}^(.4){\cong}\\
T_\ell(C')\; \ar[r] & \; \displaystyle \frac{T_\ell(C')}{(1 - \Phi_{k',C'}) T_\ell(C')} \;& \;\mathrm{Pic}^0(C')[\ell^\infty]\ar[l]_(.4){\mathrm{art}_{C'}^0}^(.4){\cong}
}
\end{equation}
where $\mathrm{art}_C^0$ $($resp. $\mathrm{art}_{C'}^0$$)$ is the restriction of the Artin map $\mathrm{art}_C$ $($resp. $\mathrm{art}_{C'}$$)$ as in $(\ref{eq:artdef0ell})$.
We have $T_\ell(C') = T_\ell(C)$.  With this identification, 
 \begin{equation}
 \label{eq:clearly}
 V (1 - \Phi_{k,C}) T_\ell(C) =   (1 - \Phi_{k',C'} )T_\ell(C)
 \end{equation} and both $V$ and $\pi^*$ are injective.  
\end{cor}

\begin{proof}  The equality (\ref{eq:clearly}) follows from 
$$V (1 - \Phi_{k,C}) T_\ell(C) = (\sum_{i = 0}^{d - 1} \Phi_{k,C}^{i} ) (1 - \Phi_{k,C}) T_\ell(C) =  (1 - \Phi_{k,C}^d )T_\ell(C) =  (1 - \Phi_{k',C'} )T_\ell(C).$$
Since $T_\ell(C)$ is isomorphic to $(\mathbb{Z}_\ell)^{2g(C)}$ and $(1 - \Phi_{k',C'}) T_\ell(C)$ has finite index in $T_\ell(C)$, this equality also implies that $V$ and $\pi^*$ are injective.
\end{proof}

\section{Arithmetic Frobenius and Legendre derivatives}
\label{s:restrictedFrobenius}
\setcounter{equation}{0}

In this section, we will consider restrictions of the cup product maps (\ref{eq:cupprodbi}) and (\ref{eq:cupprodtri})  that are connected to the derivative of the arithmetic Frobenius.  We will use the notation of \S \ref{s:arithcovers}.  In particular, we assume Hypothesis \ref{hyp:primepower}, i.e. $n = \ell^z$.
\begin{prop} 
\label{prop:Aut} 
There is a unique automorphism $A$ of $\mathbb{Q}_\ell \otimes_{\mathbb{Z}_\ell} T_\ell(C)$ such that $\Phi_{k,C} = 1 + nA$.  Multiplication by $A^{-1}$ defines a homomorphism
\begin{equation}
\label{eq:dPhi}
d\mathcal{L}: \mathrm{Pic}^0(C)[n] = \frac{T_\ell(C) \cap A T_\ell(C)}{n A T_\ell(C)} \to \frac{T_\ell(C)}{n T_\ell(C) + n A T_\ell(C)} = \frac{\mathrm{Pic}^0(C)}{n\cdot\mathrm{Pic}^0(C) }
\end{equation}
which we will call the Legendre derivative of Frobenius.
\end{prop}  

\begin{proof}    
By (\ref{eq:artdef0ell}) we have an isomorphism
$$\mathrm{art}_C^0:\quad \mathrm{Pic}^0(C)[\ell^\infty] \to  \frac{T_\ell(C)}{(1 - \Phi_{k,C}) T_\ell(C)}  = \frac{T_\ell(C)}{n AT_\ell(C)}.$$
Since this group is finite and  $T_\ell(C)$ is a free $\mathbb{Z}_\ell$-submodule of $\mathbb{Q}_\ell \otimes_{\mathbb{Z}_\ell} T_\ell(C)$ of rank $2 g(C)$, this implies $A$ is an automorphism of $\mathbb{Q}_\ell \otimes_{\mathbb{Z}_\ell} T_\ell(C)$.  We have  isomorphisms
$$\frac{\mathrm{Pic}^0(C)}{n\cdot\mathrm{Pic}^0(C)} = \frac{\mathrm{Pic}^0(C)[\ell^\infty] }{ n\cdot \mathrm{Pic}^0(C)[\ell^\infty]}  = \frac{T_\ell(C)}{ nT_\ell(C) + n A T_\ell(C)} $$
and
$$\mathrm{Pic}^0(C)[n] = \frac{T_\ell(C) \cap A T_\ell(C)}{n A T_\ell(C)}$$
from which the proposition is clear.
\end{proof}

\begin{rem}
\label{rem:WTF}  
The reason for the terminology is that in the classical theory of convexity, the derivative of the Legendre transform of a differentiable function is the inverse of the derivative of the function;   see \cite{Simon}. Here $\frac{1}{n}(\Phi_{k,C} - 1)$ is a formal derivative of the arithmetic Frobenius.  It would be interesting to develop a counterpart of this theory over finite fields. 

The groups $\mathrm{Pic}^0(C)[n]$ and $\mathrm{Pic}^0(C)/n\cdot\mathrm{Pic}^0(C)$ have the same order, but $d\mathcal{L}$ is not in general an isomorphism.  
For example, suppose $C$ is an elliptic curve with affine equation $y^2 = x^3 - 3$ over $k = \mathbb{Z}/7$.   Then $C$ has an automorphism $\zeta$ of order $3$ over $k$ fixing $y$ and sending $x $ to $2x$, so  $C$ has complex multiplication by $\mathbb{Z}[\zeta]$.  One finds $\# C(k) = 3$.  This implies that if $\ell = n = 3$, $T_\ell(C)$ must be a free rank one $\mathbb{Z}_\ell[\zeta]$-module, and $\Phi_{k,C}  - 1$ acts as multiplication by a uniformizing parameter in $\mathbb{Z}_\ell[\zeta]$.   Since $\mathbb{Z}_\ell[\zeta]$ is quadratically ramified over $\mathbb{Z}_\ell$, it follows that $A^{-1}$ acts by  multiplication by a uniformizer in $\mathbb{Z}_\ell[\zeta]$.  This forces $d\mathcal{L}$ to be the zero homomorphism. On the other hand, if $C$ is any curve such that $\mathrm{Pic}^0(C)[n]$ is isomorphic to $(\mathbb{Z}/n)^{2 g(C)}$, then $A$ defines an endomorphism of $T_\ell(C)$, and $d\mathcal{L}$ is an isomorphism. 
\end{rem}

\begin{thm}
\label{thm:legendre1}  
Under Hypothesis $\ref{hyp:primepower}$, the restrictions of $(\ref{eq:cupprodbi})$ $($resp. $(\ref{eq:cupprodtri})$$)$ in which the first $($resp. second\/$)$ argument lies in $\rH^1(k,\mathbb{Z}/n)$ can be computed in the following way. Suppose $\tau$, $\alpha$ and $b$ are as in parts $(i)$ and $(ii)$ of Theorem $\ref{thm:MSformula}$, and $\alpha \in \rH^1(k,\mathbb{Z}/n) = \mathrm{Hom}(\mathrm{Gal}(\overline{k}/k),\mathbb{Z}/n)$.  Let $[\mathfrak{b}]$ be the class in $\mathrm{Pic}^0(C)[n]$ of the divisor
$\mathfrak{b} = \mathrm{div}_C(b)/n$. Let $\Phi_{\overline{k}/k} \in \mathrm{Gal}(\overline{k}/k)$ be the Frobenius automorphism. One has
\begin{equation}
\label{eq:cuptwo1}
\alpha \cup [b] =  \alpha(\Phi_{\overline{k}/k}) \cdot d\mathcal{L} ([\mathfrak{b}])  \in \mathrm{Pic}(C) /n \cdot \mathrm{Pic}(C) = \rH^2(C,\mu_n).
\end{equation}
The cup product $\tau \cup \alpha \cup [b]$ depends only on the restriction of $\tau$ to $\mathrm{Pic}^0(C)$, and
\begin{equation}
\label{eq:formula1}
\tau \cup \alpha \cup [b] =\alpha(\Phi_{\overline{k}/k}) \cdot \ \tau(d\mathcal{L} ([\mathfrak{b}])) \in \mathbb{Z}/n = \rH^3(C,\mu_n).
\end{equation}
\end{thm}

\begin{rem}
\label{rem:geomversusarith}  
In  Theorem $\ref{thm:legendre1}$, it is important that  the automorphism $\Phi_{k,C}$ of $T_\ell(C)$ used to define $d\mathcal{L}$ is the one arising from the arithmetic Frobenius $\Phi_{\overline{k}/k} \otimes 1_{C/k}$ of $\overline{C} =\overline{k} \otimes_{k} C$. The geometric Frobenius endomorphism $1_{\overline{k}} \otimes F_{C/k}$ of  $\overline{C} = \overline{k} \otimes_k C$ over $\overline{k}$ is the identity map on the underlying topological space of $C$ and that is the $q^{\rth}$ power map on $O_C$. In particular, $1_{\overline{k}} \otimes F_{C/k}$ acts on $\overline{C}(\overline{k})$ by raising the coordinates of any point to the $q^{\rth}$ power (see \cite[p. 186 and p. 291-292]{Milne}). The action of $1_{\overline{k}}\otimes F_{C/k}$ on $T_\ell(C)$ is the inverse of the action of $\Phi_{k,C}$.  In particular, if one writes the action of $\Phi_{k,C}$ on $T_\ell(C)$ as $1 + nA$ as in Proposition $\ref{prop:Aut}$, then $1_{\overline{k}} \otimes F_{C/k}$ acts as $(1 + n A)^{-1}$.  
\end{rem}

\begin{proof}[Proof of Theorem $\ref{thm:legendre1}$]  
We can reduce to the case in which $\alpha$ is the generator of $\rH^1(k,\mathbb{Z}/n)$ such that $\alpha(\Phi_{\overline{k}/k}) = 1$. 
By parts (i) and (ii) of Theorem \ref{thm:MSformula}, formulas for the cup products in (\ref{eq:cuptwo1}) and (\ref{eq:formula1}) are obtained as follows.  Let $k'$ be the cyclic extension of $k$ of degree $n$.  Then $C' = k' \otimes_k C$ is a cyclic unramified cover of $C$ of degree $n$, and $\alpha$ defines an isomorphism from  $\mathrm{Gal}(C'/C) = \mathrm{Gal}(k'/k)$ to $\mathbb{Z}/n$ by sending  $\Phi_{k'/k}$ to $1$. 

We have defined $ \mathfrak{b} = \mathrm{div}_C(b)/n \in \mathrm{Div}(C)$.  There is an element $c \in k(C')$ such that $b = \mathrm{Norm}_{k(C')/k(C)}(c)$ since $k(C') = k'(C') = k' \otimes_k k(C)$ is a cyclic unramified (constant field) extension of $k(C)$. Let $\psi_k$ be the automorphism of $k(C') = k' \otimes_k k(C)$ which is the identity on $1 \otimes k(C)$ and which is the Frobenius automorphism  $\Phi_{k'/k}$ on  $k' \otimes 1 = k'$. Let $\pi:C' \to C$ be the morphism associated with the inclusion $k(C) \subset k(C')$.  There is a divisor $\mathfrak{c} \in \mathrm{Div}(C')$ such that 
\begin{equation}
\label{eq:right}
\mathrm{div}_{C'}(c) = \pi^* \mathfrak{b}+ (1 - \psi_k) \cdot \mathfrak{c}
\end{equation}
since $C' \to C$ is cyclic and unramified and the norm of the divisor $\mathrm{div}_{C'}(c) - \pi^* \mathfrak{b}$ is trivial. 

In the notation of Theorem \ref{thm:MSformula} we have $C_\alpha=C'$, $\sigma = \psi_k$, and $\alpha'(\sigma) = 1$ because $\alpha(\Phi_{\overline{k}/k}) = 1$.  We obtain that the cup product is given by 
\begin{equation}
\label{eq:cuptwo4}
\alpha \cup [b] =  [\mathrm{Norm}_{k(C')/k(C)}(\mathfrak{c})] + \frac{n}{2} [\mathfrak{b}]
\quad \mbox{in} \quad\mathrm{Pic}(C)/n \cdot \mathrm{Pic}(C) =  \rH^2(C,\mu_n) 
\end{equation}
where $[\mathfrak{d}]$ is the class in $\mathrm{Pic}(C)$ of a divisor $\mathfrak{d}$.  For $\tau \in \rH^1(C,\mathbb{Z}/n) = \mathrm{Hom}(\mathrm{Pic}(C),\mathbb{Z}/n)$, the triple product is given by 
\begin{equation}
\label{eq:triplecomp2}
\tau \cup \alpha \cup [b] = \tau \left ([\mathrm{Norm}_{k(C')/k(C)}(\mathfrak{c})] + \frac{n}{2} [\mathfrak{b}] \right ) \quad \mbox{in} \quad \mathbb{Z}/n = \rH^3(C,\mu_n).
\end{equation}

Since the class $[\mathfrak{b}]$ lies in $\mathrm{Pic}^0(C)[n]$,  (\ref{eq:right}) shows
\begin{equation}
\label{eq:pull} 
\pi^* [\mathfrak{b}] = (\psi_k - 1) [\mathfrak{c}] \quad \mathrm{in} \quad \mathrm{Pic}^0(C')[n] .
\end{equation}
We now use diagram  (\ref{eq:moreso})  of Corollary \ref{cor:whatever}.   Here, since $C' = k' \otimes_k C$ and $[k':k] = n$, $\Phi_{k,C}$ is the automorphism of the Tate module $T_\ell(C) = T_\ell(C')$ induced by $\psi_k = 1_{C} \otimes \Phi_{\overline{k}/k}$, and $\Phi_{k',C'} = \Phi_{k,C}^{n}$.   Write $\Phi_{k,C} = 1 + nA$ as in Proposition \ref{prop:Aut}.  The endomorphism $V:T_\ell(C) \to T_\ell(C)$ is
$$V = \sum_{i = 0}^{n -1} \Phi_{k,C}^i = \sum_{i = 0}^{n - 1} (1 + n A)^i.$$
From (\ref{eq:artdef0ell}) we have isomorphisms
$$\mathrm{art}_{C}^0:\mathrm{Pic}^0(C)[\ell^\infty] \xrightarrow{\cong} \frac{T_\ell(C)}{(1 - \Phi_{k,C})T_\ell(C)} =  \frac{T_\ell(C) }{ n AT_\ell(C)}$$
and
$$\mathrm{art}_C^0: \mathrm{Pic}^0(C)[n]  \xrightarrow{\cong} \frac{T_\ell(C) \cap  AT_\ell(C)}{nA T_\ell(C)}.$$
Hence $\mathrm{art}_C^0([\mathfrak{b}])  =  [AE]$  in $(T_\ell(C) \cap  AT_\ell(C))/nA T_\ell(C)$ for some $E \in T_\ell(C)$ such that $AE \in T_\ell(C)$.  Therefore (\ref{eq:moreso}) shows
$$\mathrm{art}_{C'}^0(\pi^* [\mathfrak{b}]) = [VAE] \quad \mathrm{in}\quad \frac{T_\ell(C)}{(1 - \Phi_{k',C'}) T_\ell(C)} = \frac{T_\ell(C)}{(1 - (1+ nA)^n)T_\ell(C)}.$$
Here
\begin{equation}
\label{eq:VAU}
VA = AV = A \sum_{i = 0}^{n - 1} (1 + n A)^i = A  \sum_{i = 0}^{n - 1} \sum_{j = 0}^i \left(\begin{array}{c} i \\ j \end{array}\right) (n A)^j  = n A  U
\end{equation}
where 
\begin{equation}
\label{eq:Udef}
U = 1 + A \sum_{i = 1}^{n - 1} \sum_{j = 1}^i \left(\begin{array}{c} i \\ j \end{array}\right) (n A)^{j-1}.
\end{equation}
The element $D = UE$ of $\mathbb{Q}_\ell \otimes_{\mathbb{Z}_\ell} T_\ell(C)$ lies in $T_\ell(C)$ since $E$ and $AE$ are both in $T_\ell(C)$. We have 
\begin{equation}
\label{eq:uggg}
\mathrm{art}_{C'}^0(\pi^* [\mathfrak{b}]) = [VAE]  = [n A UE] = [n A D] = (\Phi_{k,C} - 1) [D] 
\end{equation}
where on the right $[D]$ means the class of $D$ in $T_\ell(C)/(1 - \Phi_{k',C'}) T_\ell(C) = \mathrm{Pic}^0(C')[\ell^\infty]$.

The action of $\Phi_{k,C}$ on $T_\ell(C)$ gives the action of $\psi_k$ on $\mathrm{Pic}^0(C')[\ell^\infty]$. Equations (\ref{eq:pull}) and (\ref{eq:uggg}) now show 
\begin{equation}
\label{eq:fishy}
[\mathfrak{c}] - [D] \in \mathrm{Pic}^0(C')^G
\end{equation}
when $G = \mathrm{Gal}(C'/C)$ is the cyclic group of order $n$ generated by $\psi_k$. 

We have an exact sequence
\begin{equation}
\label{eq:longexact}
1 \to (k')^* \to k(C')^* \to \mathrm{Div}(C') \to \mathrm{Pic}(C') \to 1
\end{equation}
in which $k'$ is the field of constants of $k(C')$ and the map $k(C')^* \to \mathrm{Div}(C')$ is induced by taking divisors.  Splitting this into two short exact sequences and using that $\rH^1(G,k(C')^*) = 0$ by Hilbert's theorem 90, we get exact sequences
\begin{equation}
\label{eq:firstseq}
\mathrm{Div}(C')^G \to \mathrm{Pic}(C')^G \to \rH^1(G,k(C')^*/(k')^*) \to 0 = \rH^1(G,\mathrm{Div}(C'))
\end{equation}
and
\begin{equation}
\label{eq:next}
0 \to \rH^1(G,k(C')^*/(k')^*) \to \rH^2(G,(k')^*) \to \rH^2(G,k(C')^*).
\end{equation}
Since in (\ref{eq:next}), $k'/k$ is a finite Galois extension with cyclic Galois group $G$, we have $\rH^2(G,(k')^*) = \hat{\rH}^0(G,(k')^*) = 0$.  We conclude  from (\ref{eq:firstseq}) that the map $\mathrm{Div}(C')^G \to \mathrm{Pic}(C')^G$ is surjective.  However, $C' \to C$ is a cyclic unramified $G$-extension, so $\mathrm{Div}(C')^G = \pi^* \mathrm{Div}(C)$. Using this in (\ref{eq:fishy}) shows that there is a divisor $\mathfrak{e}$ on $C$ such that
\begin{equation}
\label{eq:Enice}
[\mathfrak{c}] =  [D] + \pi^* [\mathfrak{e}].
\end{equation}
We now have 
\begin{eqnarray}
\mathrm{Norm}_{k(C')/k(C)} [\mathfrak{c}] &=& \mathrm{Norm}_{k(C')/k(C)}[D] + \mathrm{Norm}_{k(C')/k(C)} \pi^*[\mathfrak{e}] \nonumber \\
\label{eq:normeq}
&=& \mathrm{Norm}_{k(C')/k(C)}[D] + n \cdot [\mathfrak{e}].
\end{eqnarray}

 Let $[D]_1$ be the image of $D\in T_\ell(C)$ in $\mathrm{Pic}^0(C)[\ell^\infty] =  T_\ell(C)/(1 - \Phi_{k,C}) T_\ell(C) = T_\ell(C)/nAT_\ell(C)$. 
By diagram (\ref{eq:oomphah}) and equation (\ref{eq:uggg}), we have
\begin{equation}
\label{eq:straight}
\mathrm{Norm}_{C'/C}[D] = [D]_1.
\end{equation}
Formula (\ref{eq:cuptwo4}) together with (\ref{eq:normeq}) and (\ref{eq:straight}) now give
\begin{equation}
\label{eq:cupper}
\alpha \cup [b] =  \mathrm{Norm}_{k(C')/k(C)} [\mathfrak{c}] + \frac{n}{2} [\mathfrak{b}]  = [D]_1 + \frac{n}{2} [\mathfrak{b}]
\end{equation}
since $n \cdot [\mathfrak{e}] $ is trivial in $\mathrm{Pic}(C)/n\mathrm{Pic}(C)$. We need to show that 
\begin{equation}
\label{eq:todolist}
[D]_1 + \frac{n}{2}[\mathfrak{b}] = d\mathcal{L} ([\mathfrak{b}]) \quad \mathrm{in}\quad \frac{T_\ell(C)}{n T_\ell(C) + n A T_\ell(C)} = \frac{\mathrm{Pic}^0(C)}{n\cdot \mathrm{Pic}^0(C)}. 
\end{equation}
Recall that we chose $E \in T_\ell(C)$ so that $AE \in T_\ell(C)$ and 
$$\mathrm{art}_C^0([\mathfrak{b}]) = [AE].$$
The definition of $d\mathcal{L}$ then shows that 
\begin{equation}
\label{eq:dellis}
d\mathcal{L} ([\mathfrak{b}]) = [E] \quad \mathrm{in}\quad \frac{T_\ell(C)}{n T_\ell(C) + n A T_\ell(C)} = \frac{\mathrm{Pic}^0(C)}{n\cdot\mathrm{Pic}^0(C)}.
\end{equation}
Concerning the left side of (\ref{eq:todolist}), $D = UE$ when $U$ is as in (\ref{eq:Udef}).   If $\ell > 2$ then $\frac{n}{2}[\mathfrak{b}] = 0$ in $\mathrm{Pic}^0(C)/n\cdot\mathrm{Pic}^0(C)$ since $2$ is then invertible mod $n$.  So (\ref{eq:todolist}) for $\ell > 2$ is equivalent to 
\begin{equation}
\label{eq:largerthantwo}
[UE] - [E] = 0 \quad  \mathrm{in}\quad \frac{T_\ell(C)}{nT_\ell(C) + n A T_\ell(C)} \quad \mathrm{when}\quad \ell > 2.
\end{equation}
We have 
$$U  - 1 =  A \sum_{i = 1}^{n - 1} \sum_{j = 1}^i \left(\begin{array}{c} i \\ j \end{array}\right) (n A)^{j-1} =  A \frac{n (n -1)}{2} + (\sum_{i = 1}^{n- 1} \sum_{j = 2}^i (n A)^{j-1}) A.$$
So
$$[UE] - [E] = \frac{n (n -1)}{2} [AE] + (\sum_{i = 1}^{n- 1} \sum_{j = 2}^i (n A)^{j-1}) [AE]  = 0 \quad \mathrm{in}\quad \frac{T_\ell(C)}{n T_\ell(C) + n A T_\ell(C)}  \quad \mathrm{when} \quad \ell > 2$$
as required since $2\big|(n -1)$, $AE \in T_\ell(C)$,  $E \in T_\ell(C)$ and $n A$ is an endomorphism of $T_\ell(C)$.  Suppose now that $\ell = 2$. Then (\ref{eq:todolist})  is equivalent to
\begin{equation}
\label{eq:urkel}
[UE] + \frac{n}{2} [AE]  - [E] = 0  \quad \mathrm{in}\quad \frac{T_\ell(C)}{n T_\ell(C) + n A T_\ell(C)} \quad \mathrm{when}\quad \ell = 2.
\end{equation}
We have 
\begin{eqnarray*}
\frac{n}{2} A + U - 1  &=&  \frac{n}{2} A + A \frac{n (n -1)}{2} + (\sum_{i = 1}^{n- 1} \sum_{j = 2}^i (n A)^{j-1}) A\\
&=&  \frac{n^2}{2} A  + (\sum_{i = 1}^{n- 1} \sum_{j = 2}^i (n A)^{j-1}) A \quad \mathrm{when}\quad \ell = 2
\end{eqnarray*}
from which (\ref{eq:urkel}) is clear because $E \in T_\ell(C)$.
\end{proof}

\begin{thm}
\label{thm:legendre2}  
Under Hypothesis $\ref{hyp:primepower}$, the restrictions of $(\ref{eq:cupprodbi})$ $($resp. $(\ref{eq:cupprodtri})$$)$ in which the second $($resp. third\/$)$ argument lies in $\rH^1(k,\mu_n)$ can be computed in the following way. Suppose $\tau$, $\alpha$ and $b$ are as in parts $(i)$ and $(ii)$ of Theorem $\ref{thm:MSformula}$, and $[b] \in \rH^1(k,\mu_n) = k^*/(k^*)^n$.     Let $d = \# \rH^1(k,\mu_n)$.  Then $d\big|n$ and there is a unique homomorphism $c \in \mathrm{Hom}(\mathrm{Gal}(\overline{k}/k),\mu_d)= \rH^1(k,\mu_d)$ such that $c(\Phi_{\overline{k}/k}) = \zeta$ is in $\mu_d(k) \subset k^*$ and  $[b]$ is represented by the one-cocycle $c$.  Let $[a] \in \rH^1(C,\mu_n)$ be the image of $\alpha$ under the homomorphism $\rH^1(C,\mathbb{Z}/n) \to \rH^1(C,\mu_n)$ induced by the unique homomorphism of sheaves   $\mathbb{Z}/n \to \mu_n$ that sends the global section $1$ to the global section $\zeta$.  Let $[\mathfrak{a}] \in \mathrm{Pic}^0(C)[n]$ be the image of $[a]$ under the map $\rH^1(C,\mu_n) \to \mathrm{Pic}^0(C)[n]$ produced by the Kummer sequence.  Then 
\begin{equation}
\label{eq:cuptwo2}
\alpha \cup [b] =  - d\mathcal{L} ([\mathfrak{a}])  \in \mathrm{Pic}(C) /n\cdot \mathrm{Pic}(C) = \rH^2(C,\mu_n).
\end{equation}
and
\begin{equation}
\label{eq:formula2}
\tau \cup \alpha \cup [b] =  - \tau(d\mathcal{L} ([\mathfrak{a}])) \in \mathbb{Z}/n = \rH^3(C,\mu_n).
\end{equation}
\end{thm}

\begin{proof} 
 Define $\beta \in \rH^1(k,\mathbb{Z}/n) = \mathrm{Hom}(\mathrm{Gal}(\overline{k}/k),\mathbb{Z}/n)$ to be the homomorphism that sends $\Phi_{\overline{k}/k}$ to $1$.  In view of Theorem \ref{thm:legendre1}, it will suffice to show 
\begin{equation}
\label{eq:reverse}
\alpha \cup [b] = -  \beta \cup [a] \quad \mathrm{and} \quad \tau \cup \alpha \cup [b] = - \tau \cup \beta \cup [a].
\end{equation}

The group $\rH^1(k,\mu_n) = k^*/(k^*)^n$ has order $d = \mathrm{gcd}(n,\# (k^*))$, so $k^*$ contains a primitive $d^{\mathrm{th}}$ root of unity.  Let $\phi_j:\rH^j(k,\mu_d) \to \rH^j(k,\mu_n)$ and $\psi_j:\rH^j(C,\mu_d) \to \rH^j(C,\mu_n)$ be the homomorphisms induced by the natural injection $\mu_d \to \mu_n$.  Moreover, let $\lambda_j:\rH^j(C,\mathbb{Z}/n) \to \rH^j(C,\mathbb{Z}/d)$ be the homomorphisms induced by the natural surjection $\mathbb{Z}/n \to \mathbb{Z}/d$.

 From the exact sequence of sheaves 
$$1 \to \mu_d \to \mu_n \to \mu_{n/d} \to 1$$
we get an exact sequence
$$\rH^0(k,\mu_{n/d}) \to \rH^1(k,\mu_d) \to \rH^1(k,\mu_n) \to \rH^1(k,\mu_{n/d}).$$
Since $\mathrm{gcd}(n/d,\# (k^*)) = 1$, we have $\rH^0(k,\mu_{n/d}) = \mu_{n/d}(k) = \{1\}$ and $\rH^1(k,\mu_{n/d}) = k^*/(k^*)^{n/d} = \{1\}$. So $\phi_1:\rH^1(k,\mu_d)  = \mathrm{Hom}(\mathrm{Gal}(\overline{k}/k),\mu_d) \to \rH^1(k,\mu_n)$ is an isomorphism.   Let $[b]' =  \phi_1^{-1}([b])$, let $\alpha'=\lambda_1(\alpha)$ and $\tau'=\lambda_1(\tau)$.  The natural maps of sheaves $\mathbb{Z}/n \otimes \mu_d \to \mu_n$ and $\mathbb{Z}/n \otimes \mathbb{Z}/n \otimes \mu_d \to \mu_n$ imply that 
\begin{equation}
\label{eq:result}
\psi_2(\alpha' \cup [b]') = \alpha \cup [b] \quad \mathrm{and} \quad \psi_3(\tau' \cup \alpha' \cup [b]') = \tau \cup \alpha \cup [b].
\end{equation}
Hence to prove Theorem \ref{thm:legendre2}, we can reduce to the case in which $d = n$.  

Now $d  = n$ implies $k^*$ contains a primitive $n^{\mathrm{th}}$ root of unity $\gamma$. There is a unique isomorphism of sheaves $\mathbb{Z}/n \to \mu_n$ sending the global section $1$ to the global section $\gamma$ and a unique homomorphism of sheaves $\mu_n \to \mu_n^{\otimes 2}$ sending $\gamma^r$ to $(\gamma \otimes \gamma)^r$. Let $\nu_j:\rH^j(C,\mathbb{Z}/n) \to \rH^j(C,\mu_n)$ and $\xi_j:\rH^j(C,\mu_n) = \rH^j(C,\mu_n^{\otimes 2})$ be the resulting isomorphisms.  Then 
$$
\nu_1(\alpha) \cup [b] = \xi_2(\alpha \cup [b]) \in \rH^2(C,\mu_n^{\otimes 2}).$$
However, we can regard $[b] \in \rH^1(k,\mu_n) = \mathrm{Hom}(\mathrm{Gal}(\overline{k}/k),\mu_n)$ as an element of $\rH^1(C,\mu_n) = \mathrm{Hom}(\pi_1(C),\mu_n)$. So since the cup product on $\rH^1(C,\mu_n)$ is anticommutative, we have
$$\nu_1(\alpha) \cup [b] = - [b] \cup \nu_1(\alpha) \quad \mathrm{in}\quad \rH^2(C,\mu_n^{\otimes 2}).$$
Now by considering cocycle representatives we see 
$$\xi_2(\beta \cup [a]) = [b] \cup \nu_1(\alpha).$$
Therefore
$$\alpha \cup [b] = \xi_2^{-1}(\nu_1(\alpha) \cup [b]) = - \xi_2^{-1}([b] \cup \nu_1(\alpha)) = - \beta \cup [a]$$
and this implies (\ref{eq:reverse}) and completes the proof.
\end{proof}

\section{Cup products of normalized classes on curves of arbitrary positive genus}
\label{s:arbitrarygenus}
\setcounter{equation}{0}

Throughout this section we will assume Hypothesis \ref{hyp:primepower}, i.e $n=\ell^z$. We let $O$ be a fixed closed point of $C$ with degree $d(O)$ prime to $\ell\cdot (\ell-1)$, which exists by Lemma \ref{lem:normalize}. The goal of this section is to prove Theorem \ref{thm:cupsizeresult}. 
In \S\ref{s:ellcurves}, we will focus on the case when the genus $g(C)=1$ and we will prove Theorem  \ref{thm:EllThm} by proving that condition (ii) of Theorem \ref{thm:cupsizeresult} holds in this case.
In \S\ref{s:genustwo}, we will give an infinite family of curves of genus $2$ for which the equivalent conditions of  Theorem $\ref{thm:cupsizeresult}$ do not hold when $n = \ell = 3$.  

\begin{proof}[Proof of Theorem $\ref{thm:cupsizeresult}$]
We consider the natural non-degenerate pairings from Lemma \ref{lem:dualer}
\begin{equation}
\label{eq:triplecup}
\xymatrix{
 \rH^1(C,\mathbb{Z}/n) \times \rH^2(C,\mu_n) \;\, \ar@{=}[d] \ar[r]& \rH^3(C,\mu_n)\ar@{=}[d]\\
 \mathrm{Hom}(\mathrm{Pic}(C),\mathbb{Z}/n) \times \displaystyle \frac{\mathrm{Pic}(C)}{n \cdot \mathrm{Pic}(C)}\ar[r] & \mathbb{Z}/n
}
\end{equation}
where the pairing in the top row is given by the cup product and in the bottom row by the evaluation map. 

The formula in part (i) of Theorem \ref{thm:cupsizeresult} implies that classes in the cup product
\begin{equation}
\label{cup:normalized}
\rH^1(C,\mathbb{Z}/n)_O \cup \rH^1(C,\mu_n)_O \quad\subseteq\quad \rH^2(C,\mu_n) = \mathrm{Pic}(C)/n \cdot \mathrm{Pic}(C)
\end{equation}
are multiples of the image of $[O]$ in $\mathrm{Pic}(C)/n \cdot \mathrm{Pic}(C)$. On the other hand, by part (i) of Lemma \ref{lem:normalizedcharacter}, elements $\alpha\in \rH^1(C,\mathbb{Z}/n)_O\subseteq \rH^1(C,\mathbb{Z}/n) = \mathrm{Hom}(\mathrm{Pic}(C),\mathbb{Z}/n)$ satisfy $\alpha([O]) = 0$. Therefore, part (i) implies part (ii).

Part (ii) of Theorem \ref{thm:cupsizeresult} implies that elements in the cup product of normalized classes in (\ref{cup:normalized}) are orthogonal to $\rH^1(C,\mathbb{Z}/n)_O$ under the non-degenerate pairing (\ref{eq:triplecup}). In fact, the direct sum decomposition of $\mathrm{Pic}(C)/n\cdot \mathrm{Pic}(C)$ in part (i) of Lemma \ref{lem:normalizedcharacter} shows that the orthogonal complement of $\rH^1(C,\mathbb{Z}/n)_O$ in $\mathrm{Pic}(C)/n \cdot \mathrm{Pic}(C)$ is equal to the cyclic subgroup of order $n$ generated by the image of $[O]$. This means that every element of the cup product in (\ref{cup:normalized}) of normalized classes is a multiple of the image of $[O]$ in $\mathrm{Pic}(C)/n \cdot \mathrm{Pic}(C)$. We can then read off this multiple by restricting to the algebraic closure and using that $d(O)$ is prime to $\ell\cdot (\ell-1)$, leading to the formula in part (i).
\end{proof}

\section{The genus one case}
\label{s:ellcurves}
\setcounter{equation}{0}

In this section, we will focus on  the case in which $C$ has genus $1$.  We make the assumptions and use the notation from \S\ref{s:complements}, \S\ref{s:arithcovers} and \S\ref{s:arbitrarygenus}. In particular, $n=\ell^z$ and $O$ is a closed point of $C$ of degree $d(O)$ prime to $\ell \cdot (\ell -1)$.  

Our goal is to prove Theorem \ref{thm:EllThm}. By Theorem \ref{thm:basechange}, it is enough to consider the case in which  $k(O) = k$, so that $d(O) = 1$. Hence we make the following assumptions throughout this section.

\begin{hypothesis}
\label{hyp:elliptic}
The genus of $C$ is $g(C)=1$, $n=\ell^z$ as in Hypothesis $\ref{hyp:primepower}$, and $O$ is a closed point of $C$ of degree $d(O)=1$.
\end{hypothesis}

A key ingredient in the proof of Theorem \ref{thm:EllThm} is the following result.

\begin{thm}
\label{thm:WeilCup}
Under Hypothesis $\ref{hyp:elliptic}$, 
let $\mathcal{F}_n$ be either the constant sheaf $\mathbb{Z}/n$ or the locally constant sheaf $\mu_n$. Let $\alpha\in \rH^1(C,\mathbb{Z}/n)_O$ and $\beta\in\rH^1(C,\mathcal{F}_n)_O$, and let $\overline{\alpha}$ and $\overline{\beta}$ be the restrictions of $\alpha$ and $\beta$ to $\rH^1(\overline{C},\mathbb{Z}/n)$ and $\rH^1(\overline{C},\mathcal{F}_n)$, respectively. If the cup product
$$\overline{\alpha} \cup_{\overline{C}} \overline{\beta} = 0 \quad \mbox{ in}\quad \rH^2(\overline{C},\mathcal{F}_n)$$
then the cup product 
$$\alpha\cup \beta =0 \quad \mbox{ in}\quad \rH^2(C,\mathcal{F}_n).$$
\end{thm}

We will defer the proof of Theorem \ref{thm:WeilCup} to \S\ref{s:cohomological} since we will use continuous group cohomology, which requires some additional notational set up.

Another key ingredient in the proof of Theorem \ref{thm:EllThm} is the following result about bilinear pairings on finite modules over $\mathbb{Z}/n$ that are generated by at most two elements.

\begin{lemma}
\label{lem:bilinear}
Suppose $M_1,M_2$ are finite modules over $\mathbb{Z}/n$ that are both generated by at most $2$ elements. Let 
$$\langle \ , \ \rangle: \quad M_1\times M_2 \to \mathbb{Z}/n$$
be a bilinear pairing. Then one of the following three conditions holds:
\begin{enumerate}
\item[(a)] $M_1$ is cyclic,
\item[(b)] $M_2$ is cyclic, say it is generated by $m_2$, $M_1$ is not cyclic, and there are generators $m_{11},m_{12}$ of $M_1$ such that $\langle m_{12}, m_2\rangle = 0$,
\item[(c)] $M_1$ and $M_2$ are both not cyclic, and there exist generators $m_{11},m_{12}$ of $M_1$ and $m_{21},m_{22}$ of $M_2$ such that $\langle m_{1i},m_{2j}\rangle = 0$ if $i \ne j$ in $\{1,2\}$.
\end{enumerate}
\end{lemma}

\begin{proof}
Suppose this is false, and that the pair $(M_1, M_2)$ is a counterexample  that minimizes $(\#M_1)\cdot(\#M_2)$. Then $M_1$ is not cyclic. If $M_2$ were cyclic and generated by $m_2$, then there would be generators $m_{11},m_{12}$ of $M_1$ such that $\langle m_{12}, m_2\rangle = 0$, contradicting that property (b) is not satisfied. Hence $M_2$ is not cyclic.

Suppose first that there exists an element  $y \in M_2$ of order $\ell$ such that $\langle M_1,y\rangle =0$. Then if we define $M_2'=M_2/\mathbb{Z} y $, the pairing $\langle \ , \ \rangle$ induces a bilinear pairing
$$\langle \ , \ \rangle' : M_1\times M_2' \to \mathbb{Z}/n.$$
By the minimality of $(M_1,M_2)$, the pairing $\langle \ , \ \rangle'$ must satisfy one of the properties (a) - (c). We know that $M_1$ and $M_2$ are both not cyclic. Note that $M_2'$ cannot be the zero module, since otherwise $\langle M_1,M_2\rangle =0$, which would imply that $\langle \ , \ \rangle$ must satisfy property (c), a contradiction. Suppose next that $M_2'$ is cyclic, which means that $M_2=\mathbb{Z} m_{21}\oplus \mathbb{Z} m_{22}$ with $y = m_{22}$ and $\ell m_{22}=0$. Since $\langle \ , \ \rangle'$ must then satisfy property (b), there are generators $m_{11},m_{12}$ of $M_1$ such that $\langle m_{12}, m_{21}\rangle = 0$. Since $\langle M_1,y\rangle =0$ and $m_{22} = y$, we also have that $\langle m_{11}, m_{22}\rangle = 0$. Hence $\langle \ , \ \rangle$ again satisfies property (c), a contradiction. Therefore, $M_2'$ is not cyclic. This means that $\langle \ , \ \rangle'$ must satisfy property (c) for some choice of generators $m_{11},m_{12}$ of $M_1$ and $m'_{21},m'_{22}$ of $M_2'$. Lifting $m'_{21},m'_{22}$ to generators $m_{21},m_{22}$ of $M_2$, we see that then $\langle \ , \ \rangle$ again satisfies property (c), a contradiction. Hence we conclude that $\langle M_1,y\rangle \ne 0$ for all $y \in M_2$ of order $\ell$. This implies that the map
\begin{eqnarray*}
M_2 &\to& \mathrm{Hom}(M_1,\mathbb{Z}/n)\\
y&\mapsto &\{ x\mapsto \langle x , y \rangle\}
\end{eqnarray*}
is injective. In particular, $\#M_2\le \# \mathrm{Hom}(M_1,\mathbb{Z}/n) = \#M_1$.

In a similar way, one can prove that there is no $x \in M_1$ of order $\ell$ such that 
$\langle x,M_2\rangle = 0$.  Therefore, the map
\begin{eqnarray*}
M_1 &\to& \mathrm{Hom}(M_2,\mathbb{Z}/n)\\
x&\mapsto &\{ y\mapsto \langle x , y \rangle\}
\end{eqnarray*}
is injective, and $\#M_1\le \# \mathrm{Hom}(M_2,\mathbb{Z}/n) = \#M_2$.

Since we also have $\#M_2\le \#M_1$, we conclude that $M_1$ maps isomorphically to $\mathrm{Hom}(M_2,\mathbb{Z}/n)=M_2^*$. Because $M_2$ is not cyclic, we can write
\begin{eqnarray*}
M_2 &=& \mathbb{Z} b_1 \oplus \mathbb{Z} b_2 \quad\mbox{ and}\\
M_1=M_2^* &=& \mathbb{Z} b_1^* \oplus \mathbb{Z} b_2^*,
\end{eqnarray*}
where $\langle b_1^*,b_2\rangle = b_1^*(b_2) = 0$ and $\langle b_2^*,b_1\rangle = b_2^*(b_1) = 0$. But this means that $\langle \ , \ \rangle$ satisfies property (c), a contradiction. In other words, no minimal counterexample $(M_1, M_2)$ exists, which completes the proof of Lemma \ref{lem:bilinear}.
\end{proof}

\begin{proof}[Proof of Theorem $\ref{thm:EllThm}$]
We use that the Weil pairing is alternating (see, for example, \cite[Theorem 1]{Howe}), which implies that the cup product $\cup_{\overline{C}}$ in (\ref{eq:ohWeil!}) is alternating when $\mathcal{F}_n=\mathbb{Z}/n$. Hence it follows by Theorem \ref{thm:WeilCup} that 
\begin{equation}
\label{eq:cupnormalizedalt}
\alpha \cup \alpha = 0 \quad\mbox{ for all $\alpha\in \rH^1(C,\mathbb{Z}/n)_O$}.
\end{equation}

We apply Lemma \ref{lem:bilinear} to $M_1=\rH^1(C,\mathbb{Z}/n)_O$, $M_2=\rH^1(C,\mu_n)_O$ and the bilinear pairing $\langle \ , \ \rangle$ that is defined as follows. Let $\alpha\in \rH^1(C,\mathbb{Z}/n)_O$ and $[b]\in\rH^1(C,\mu_n)_O$, and define
$$\langle \alpha,[b]\rangle = \overline{\alpha} \cup_{\overline{C}} \overline{[b]} \quad\mbox{ in $\mathbb{Z}/n=\rH^2(\overline{C},\mu_n)$}$$
to be the value of the pairing in $(\ref{eq:ohWeil!})$ when $\mathcal{F}_n=\mu_n$ and $\overline{\alpha}$ and $\overline{[b]}$ are the restrictions of $\alpha$ and $[b]$ to $\rH^1(\overline{C},\mathbb{Z}/n)$ and $\rH^1(\overline{C},\mu_n)$, respectively. By Lemma \ref{lem:bilinear}, we need to analyze the following three cases:

\begin{enumerate}
\item[(a)] $\rH^1(C,\mathbb{Z}/n)_O$ is cyclic. In this case $\rH^1(C,\mathbb{Z}/n)_O\cup \rH^1(C,\mathbb{Z}/n)_O$ is trivial by (\ref{eq:cupnormalizedalt}), so condition (ii) of Theorem \ref{thm:cupsizeresult} holds.

\item[(b)] $\rH^1(C,\mu_n)_O$ is cyclic, say it is generated by $[b]$, $\rH^1(C,\mathbb{Z}/n)_O$ is not cyclic, and there are generators $\alpha_1,\alpha_2$ of $\rH^1(C,\mathbb{Z}/n)_O$ such that $\overline{\alpha_2} \cup_{\overline{C}} \overline{[b]} = 0$. By Theorem \ref{thm:WeilCup}, we then have $\alpha_2\cup [b]=0$ in $\rH^2(C,\mu_n)$. Since $\alpha_1$ and $\alpha_2$ generate $\rH^1(C,\mathbb{Z}/n)_O$, it follows from (\ref{eq:cupnormalizedalt}) that the cup product $\alpha_1\cup \alpha_2$ generates $\rH^1(C,\mathbb{Z}/n)_O\cup \rH^1(C,\mathbb{Z}/n)_O$. But then 
$$(\alpha_1\cup \alpha_2)\cup [b] = \alpha_1\cup (\alpha_2\cup [b]) = 0$$
generates $\rH^1(C,\mathbb{Z}/n)_O\cup \rH^1(C,\mathbb{Z}/n)_O\cup \rH^1(C,\mu_n)_O$, which implies condition (ii) of Theorem \ref{thm:cupsizeresult}.

\item[(c)] $\rH^1(C,\mathbb{Z}/n)_O$ and $\rH^1(C,\mu_n)_O$ are both not cyclic, and there exist generators $\alpha_1,\alpha_2$ of $\rH^1(C,\mathbb{Z}/n)_O$ and $[b_1],[b_2]$ of $\rH^1(C,\mu_n)_O$ such that $\overline{\alpha_i} \cup_{\overline{C}} \overline{[b_j]} = 0$ if $i \ne j$ in $\{1,2\}$. By Theorem \ref{thm:WeilCup}, we then have $\alpha_i\cup [b_j]=0$ in $\rH^2(C,\mu_n)$ if $i \ne j$ in $\{1,2\}$. As in case (b), it follows that $\alpha_1\cup \alpha_2$ generates $\rH^1(C,\mathbb{Z}/n)_O\cup \rH^1(C,\mathbb{Z}/n)_O$. Since $\alpha_1\cup \alpha_2 = -\alpha_2\cup \alpha_1$, we obtain
$$\begin{array}{rcccccl}
(\alpha_1\cup \alpha_2)\cup [b_1] &=& \alpha_1\cup (\alpha_2\cup [b_1]) &=& 0,\\
(\alpha_1\cup \alpha_2)\cup [b_2] &=& (-\alpha_2\cup \alpha_1)\cup [b_2] &=&-\alpha_2\cup (\alpha_1\cup [b_2]) &=& 0.
\end{array}$$
Hence condition (ii) of Theorem \ref{thm:cupsizeresult} holds.
\end{enumerate}
Therefore, Theorem \ref{thm:EllThm} follows from Theorem \ref{thm:cupsizeresult} in all cases (a) - (c).
\end{proof}

\section{Cohomological calculations in the genus one case}
\label{s:cohomological}
\setcounter{equation}{0}

In this section we prove Theorem \ref{thm:WeilCup} using arguments from continuous group cohomology. We assume Hypothesis \ref{hyp:elliptic}. We need the following notation.

\begin{note}
\label{not:pi1}
Let $\overline{k(C)}$ be a fixed separable closure of $k(C)$ containing $\overline{k}$, and let $M(C)$ be the maximal everywhere unramified extension of $k(C)$ inside $\overline{k(C)}$. Then $\pi_1(C,\eta)=\mathrm{Gal}(M(C)/k(C))$. We have a short exact sequence
$$0\to \pi_1(\overline{C},\eta) \to \pi_1(C,\eta) \to \mathrm{Gal}(\overline{k}/k) \to 1.$$
The arithmetic Frobenius $\Phi_{\overline{k}/k}$ on $\overline{k}$, which the $q^{\mathrm{th}}$ power map, progenerates $\mathrm{Gal}(\overline{k}/k)$. Moreover, $\pi_1(\overline{C},\eta)$ is isomorphic to the adelic Tate module $T_{\mathbb{A}}(C) = \prod_{\ell'} T_{\ell'}(C)$, where $\ell'$ runs over all primes and $T_{\ell'}(C)$ is the $\ell'$-adic Tate module of $C$.  Defining
$$\Gamma = \frac{\pi_1(C,\eta)}{\prod_{\ell'\ne \ell} T_{\ell'}(C)}$$
we obtain a short exact sequence
$$0\to T_\ell(C) \to \Gamma \to \mathrm{Gal}(\overline{k}/k) \to 1$$
and an isomorphism
\begin{equation}
\label{eq:Gamma}
\Gamma = T_\ell(C)  \rtimes \mathrm{Gal}(\overline{k}/k) .
\end{equation}

Let $\tilde{O}$ be a place over $O$ in $M(C)$ with the properties in part (ii) of Lemma \ref{lem:normalizedcharacter}.  Let $\Gamma_{\tilde{O}}$ be the image in $\Gamma$ of the decomposition group of $\tilde{O}$ in $\pi_1(\overline{C},\eta) = \mathrm{Gal}(M(C)/k(C))$.  Let $\Phi$ be the Frobenius progenerator of $\Gamma_{\tilde O}$, so that $\Phi$ is the unique element mapping to $\Phi_{\overline{k}/k}$ in $\mathrm{Gal}(\overline{k}/k)$. 
Let $a_1$ and $a_2$ be progenerators of $T_\ell(C)$. We will write $T_\ell(C)$ multiplicatively, i.e. each element can be expressed as
$$a_1^{t_1}\cdot a_2^{t_2}$$
for unique $t_1,t_2\in \mathbb{Z}_\ell$.
We have that $\Gamma$ is progenerated by $a_1$, $a_2$ and $\Phi$.
\end{note}

\begin{rem}
\label{rem:barresolution}
Let $G$ be a profinite group, and consider the standard resolution $B_{G,\bullet}\to \widehat{\mathbb{Z}}$ of $\widehat{\mathbb{Z}}$ as a $\widehat{\mathbb{Z}}\ps{G}$-module, where
\begin{equation}
\label{eq:bar1}
\begin{array}{lccccccccccc}
B_{G,\bullet}= &  \cdots &\xrightarrow{\partial_{G,3}}& \widehat{\mathbb{Z}}\ps{G\times G\times G} &\xrightarrow{\partial_{G,2}}& \widehat{\mathbb{Z}}\ps{G\times G} &\xrightarrow{\partial_{G,1}}& \widehat{\mathbb{Z}}\ps{G} \\
&&&&&(g_0,g_1) &\mapsto& g_1-g_0\\
&&&(g_0,g_1,g_2) &\mapsto& (g_1,g_2)-(g_0,g_2)+(g_0,g_1)
\end{array}
\end{equation}
is concentrated in degrees $\ge 0$ and  $g_0,g_1,\ldots$ are arbitrary elements of $G$ (see \cite[\S 6.2]{RibesZ}). Recall that the corresponding bar resolution is obtained by introducing the ``bar notation''
$$[g_1|g_2|\cdots|g_m]=(1,g_1,g_1g_2,\ldots, g_1g_2\ldots g_m)$$
for all $g_1,g_2,\ldots,g_m$.

Note that we have a contracting homotopy $H_{G,m}: B_{G,m}=\widehat{\mathbb{Z}}\ps{G^{m+1}} \to B_{G,m+1}=\widehat{\mathbb{Z}}\ps{G^{m+2}}$, given by $H_{G,m}(g_0,g_1,\ldots,g_m)=(1,g_0,g_1,\ldots,g_m)$ (see \cite[\S I.5]{brown}), which satisfies 
\begin{equation}
\label{eq:contracthom}
(\partial_{G,m+1}\circ H_{G,m}) (w) = w\quad\mbox{ for all $w\in\mathrm{Image}(\partial_{G,m+1})\subseteq B_m$.}
\end{equation}

Let $M$ and $N$ be two discrete $\widehat{\mathbb{Z}}\ps{G}$-modules, and let $[\phi_M]\in \rH^1(G,M)$ and $[\phi_N]\in \rH^1(G,N)$ be classes that are represented by continuous one-cocycles $\phi_M\in C^1(G,M)$ and $\phi_N\in C^1(G,N)$. Then the cup product $[\phi_M]\cup[\phi_N]\in \rH^2(G,M\otimes_{\mathbb{Z}} N)$ is represented by the continuous two-cycle $\phi_M\cup\phi_N$ given by
\begin{equation}
\label{eq:cupbar}
(\phi_M\cup\phi_N)([g_1|g_2]) = - \phi_M([g_1])\otimes g_1\phi_N([g_2])
\end{equation}
for all $g_1,g_2\in G$ (see \cite[\S V.3]{brown}).
\end{rem}

Let $\overline{\alpha}\in \rH^1(\overline{C},\mathbb{Z}/n)=\mathrm{Hom}(T_\ell(C),\mathbb{Z}/n)$ and let $\overline{\beta}\in \rH^1(\overline{C},\mathcal{F}_n) = \mathrm{Hom}(T_\ell(C),\tilde{\mathcal{F}}_n)$, where the sheaf $\mathcal{F}_n$ is $\mathbb{Z}/n$ or $\mu_n$, and the group $\tilde{\mathcal{F}}_n$ is $\mathbb{Z}/n$ or $\tilde{\mu}_n$, respectively.
We want to use a free resolution $W_\bullet\to \widehat{\mathbb{Z}}\to 0$ of $\widehat{\mathbb{Z}}$ as a $\widehat{\mathbb{Z}}\ps{T_\ell(C)}$-module that is different from the standard resolution $B_{T_\ell(C),\bullet}\to \widehat{\mathbb{Z}}$ to provide a different presentation of the cup product $\overline{\alpha}\cup_{\overline{C}}\overline{\beta} \in \rH^2(\overline{C},\mathcal{F}_n)$. We define $W_\bullet$ to be the (Koszul) complex
\begin{equation}
\label{eq:resolvealgclosure}
\begin{array}{lccccccccccc}
W_\bullet= &  \cdots \; 0 &\xrightarrow{\delta_3=0}& \widehat{\mathbb{Z}}\ps{T_\ell(C)}h &\xrightarrow{\delta_2}& \widehat{\mathbb{Z}}\ps{T_\ell(C)} f_1\oplus \widehat{\mathbb{Z}}\ps{T_\ell(C)}f_2 &\xrightarrow{\delta_1}& \widehat{\mathbb{Z}}\ps{T_\ell(C)}e \\
&&&&&f_1\oplus 0 \;\,&\mapsto& (a_1-1)e\\
&&&&& \;\, 0\oplus f_2  &\mapsto& (a_2-1)e\\
&&&h &\mapsto& (a_2-1)f_1 \oplus (1-a_1)f_2
\end{array}
\end{equation}
which only has nonzero terms in degrees $0$, $1$ and $2$. Note that the basis elements $h,f_1,f_2,e$ are only used to more readily distinguish between the different nonzero terms of $W_\bullet$. 

Writing $B_\bullet = B_{T_\ell(C),\bullet}$ and $\partial_m=\partial_{T_\ell(C),m}$, we obtain a morphism of complexes
$$\xymatrix{
W_\bullet\ar[d]_{\lambda_\bullet} &0 \ar[r]^(.4){0} & \widehat{\mathbb{Z}}\ps{T_\ell(C)}h \ar[d]_{\lambda_2}\ar[r]^(.4){\delta_2} & \widehat{\mathbb{Z}}\ps{T_\ell(C)} f_1\oplus \widehat{\mathbb{Z}}\ps{T_\ell(C)}f_2 \ar[d]_{\lambda_1}\ar[r]^(.6){\delta_1} & \widehat{\mathbb{Z}}\ps{T_\ell(C)}e\ar[d]_{\lambda_0}\\
B_\bullet & \cdots \ar[r]^(.2){\partial_3}& \widehat{\mathbb{Z}}\ps{T_\ell(C)\times T_\ell(C)\times T_\ell(C)} \ar[r]^(.55){\partial_2}& \widehat{\mathbb{Z}}\ps{T_\ell(C)\times T_\ell(C)} \ar[r]^(.6){\partial_1}& \widehat{\mathbb{Z}}\ps{T_\ell(C)}
}$$
where
\begin{eqnarray}
\label{eq:lambda0}
\lambda_0(e) &=& 1,\\
\label{eq:lambda11}
\lambda_1(f_1\oplus 0) &=& (1,a_1),\\
\label{eq:lambda12}
\lambda_1(0\oplus f_2) &=& (1,a_2),\\
\label{eq:lambda2}
\lambda_2(h) &=& (1,a_2,a_2a_1) - (1,1,a_1) + (1,1,a_2) - (1,a_1,a_1a_2).
\end{eqnarray}
These equalities, and in particular (\ref{eq:lambda2}), can be found using property (\ref{eq:contracthom}) of the contracting homotopy $H_{T_\ell(C),m}$.

\begin{lemma}
\label{lem:cupalgclosure}
Let $\mathcal{F}_n$ be the sheaf $\mathbb{Z}/n$ or $\mu_n$, and let $\tilde{\mathcal{F}}_n$ be the group $\mathbb{Z}/n$ or $\tilde{\mu}_n$, respectively. We have an isomorphism
$$\rH^2(\overline{C},\mathcal{F}_n)=
\mathrm{Hom}_{\widehat{\mathbb{Z}}\ps{T_\ell(C)}}(\widehat{\mathbb{Z}}\ps{T_\ell(C)}h,\tilde{\mathcal{F}}_n).$$
If $\overline{\alpha}\in \rH^1(\overline{C},\mathbb{Z}/n)=\mathrm{Hom}(T_\ell(C),\mathbb{Z}/n)$ and $\overline{\beta}\in \rH^1(\overline{C},\mathcal{F}_n) = \mathrm{Hom}(T_\ell(C),\tilde{\mathcal{F}}_n)$ then
$\overline{\alpha} \cup_{\overline{C}} \overline{\beta}\in \mathrm{Hom}_{\widehat{\mathbb{Z}}\ps{T_\ell(C)}}(\widehat{\mathbb{Z}}\ps{T_\ell(C)}h,\tilde{\mathcal{F}}_n)$ is given by
\begin{equation}
\label{eq:Weil1}
\left(\overline{\alpha} \cup_{\overline{C}} \overline{\beta}\right)(h) = 
\overline{\alpha}(a_1)\overline{\beta}(a_2) - \overline{\alpha}(a_2)\overline{\beta}(a_1) 
\quad\mbox{ if}\quad \mathcal{F}_n=\mathbb{Z}/n
\end{equation}
and by
\begin{equation}
\label{eq:Weil2}
\left(\overline{\alpha} \cup_{\overline{C}} \overline{\beta}\right)(h) = 
\overline{\beta}(a_2)^{\overline{\alpha}(a_1)}  \cdot \overline{\beta}(a_1)^{-\overline{\alpha}(a_2)}  \quad\mbox{ if}\quad \mathcal{F}_n=\mu_n.
\end{equation}
\end{lemma}

\begin{proof}
Since $n=\ell^z$, we have $\rH^2(\overline{C},\mathcal{F}_n) = \rH^2(T_\ell(C),\tilde{\mathcal{F}}_n)$ by Lemma \ref{lem:achinger}.
To prove the first statement, we consider the map
$$\delta_2^*: \mathrm{Hom}_{\widehat{\mathbb{Z}}\ps{T_\ell(C)}}(\widehat{\mathbb{Z}}\ps{T_\ell(C)}f_1\oplus \widehat{\mathbb{Z}}\ps{T_\ell(C)}f_2,\tilde{\mathcal{F}}_n)\to \mathrm{Hom}_{\widehat{\mathbb{Z}}\ps{T_\ell(C)}}(\widehat{\mathbb{Z}}\ps{T_\ell(C)}h,\tilde{\mathcal{F}}_n)$$
given by precomposition with $\delta_2$. Since $a_1$ and $a_2$ act trivially on $\tilde{\mathcal{F}}_n$, the image of $\delta_2^*$ is trivial. The first statement now follows from the isomorphism
$$\rH^2(\overline{C},\mathcal{F}_n)=
\frac{\mathrm{Hom}_{\widehat{\mathbb{Z}}\ps{T_\ell(C)}}(\widehat{\mathbb{Z}}\ps{T_\ell(C)}h,\tilde{\mathcal{F}}_n)}{\mathrm{Image}(\delta_2^*)}.$$

Suppose first that $\mathcal{F}_n=\mathbb{Z}/n$. We use the natural isomorphism $\mathbb{Z}/n\otimes_{\mathbb{Z}} \mathbb{Z}/n = \mathbb{Z}/n$ given by multiplication. Since $T_\ell(C)$ acts trivially on $\mathbb{Z}/n$, we have, by (\ref{eq:cupbar}), that $\overline{\alpha} \cup_{\overline{C}} \overline{\beta}$ is represented by the two-cocycle $c_{\overline{\alpha},\overline{\beta}}\in C^2(T_\ell(C),\mathbb{Z}/n)$ satisfying for all $g_1,g_2\in T_\ell(C)$, 
$$c_{\overline{\alpha},\overline{\beta}}([g_1|g_2]) = 
- \overline{\alpha}(g_1)\overline{\beta}(g_2).$$
Using $[g_1|g_2]=(1,g_1,g_1g_2)$ and (\ref{eq:lambda2}), we obtain that, as an element of $\mathrm{Hom}_{\widehat{\mathbb{Z}}\ps{T_\ell(C)}}(\widehat{\mathbb{Z}}\ps{T_\ell(C)}h,\tilde{\mathcal{F}}_n)$, $\overline{\alpha} \cup_{\overline{C}} \overline{\beta}$ is given by
\begin{eqnarray*}
\left(\overline{\alpha} \cup_{\overline{C}} \overline{\beta}\right)(h) &=&
-\overline{\alpha}(a_2)\overline{\beta}(a_1) + \overline{\alpha}(1)\overline{\beta}(a_1) -\overline{\alpha}(1)\overline{\beta}(a_2) + \overline{\alpha}(a_1)\overline{\beta}(a_2)\\
&=& -\overline{\alpha}(a_2)\overline{\beta}(a_1) + \overline{\alpha}(a_1)\overline{\beta}(a_2)
\end{eqnarray*}
which implies (\ref{eq:Weil1}). The case (\ref{eq:Weil2}) when $\mathcal{F}_n=\mu_n$ follows by using the natural isomorphism $\mathbb{Z}/n\otimes_{\mathbb{Z}}\tilde{\mu}_n\to \tilde{\mu}_n$ given by exponentiation and that $T_\ell(C)$ also acts trivially on $\tilde{\mu}_n$.
\end{proof}

Let now $\alpha\in\rH^1(C,\mathbb{Z}/n)_O$ and $\beta\in \rH^1(C,\mathcal{F}_n)_O$, where $\Gamma$ is defined as in (\ref{eq:Gamma}) and $\tilde{\mathcal{F}}_n$ is as above.
We want to use the induced complex $V_\bullet = \mathrm{Ind}_{T_\ell(C)}^\Gamma W_\bullet$ to construct a particular free resolution of $\widehat{\mathbb{Z}}$ as a $\widehat{\mathbb{Z}}\ps{\Gamma}$-module in order to provide a presentation of the cup product $\alpha\cup \beta \in \rH^2(C,\mathcal{F}_n)$.

We consider the following commutative diagram of induced complexes from $T_\ell(C)$ to $\Gamma$
$$\xymatrix{
V_\bullet = \mathrm{Ind}_{T_\ell(C)}^\Gamma W_\bullet \ar[r]\ar[d]_{\Phi-1} & \mathrm{Ind}_{T_\ell(C)}^\Gamma \widehat{\mathbb{Z}} = \widehat{\mathbb{Z}}\ps{\widehat{\langle \Phi\rangle}}\ar[d]_{\Phi-1}\\
V_\bullet = \mathrm{Ind}_{T_\ell(C)}^\Gamma W_\bullet \ar[r]& \mathrm{Ind}_{T_\ell(C)}^\Gamma \widehat{\mathbb{Z}} = \widehat{\mathbb{Z}}\ps{\widehat{\langle \Phi\rangle}}
}$$
where the right vertical morphism $\Phi-1$ is injective with cokernel $\widehat{\mathbb{Z}}$. It follows that the mapping cone of the morphism 
$$V_\bullet = \mathrm{Ind}_{T_\ell(C)}^\Gamma W_\bullet\xrightarrow{\Phi-1} V_\bullet = \mathrm{Ind}_{T_\ell(C)}^\Gamma W_\bullet$$
provides a free resolution $P_\bullet\to \widehat{\mathbb{Z}}$ of $\widehat{\mathbb{Z}}$ as a $\widehat{\mathbb{Z}}\ps{\Gamma}$-module. By \cite[\S2.6]{Bourbaki10}, we have $P_i=V_{i-1}\oplus V_i$ with $i^{\mathrm{th}}$ differential $\delta_{P,i}:P_i\to P_{i-1}$ given by
$$\delta_{P,i}=\left(\begin{array}{cc} -\delta_{V,i-1} & 0\\ -(\Phi-1) & \delta_{V,i}\end{array}\right).$$
More precisely, we have that
\begin{equation}
\label{eq:resolve}
\begin{array}{lc}
P_\bullet= &  \cdots \;0 \to
P_3=V_2 \xrightarrow{\delta_{P,3}} 
P_2 = V_1\oplus V_2 \xrightarrow{\delta_{P,2}} 
P_1=V_0\oplus V_1  \xrightarrow{\delta_{P,1}}
P_0=V_0
\end{array}
\end{equation}
is concentrated in degrees $0$, $1$, $2$ and $3$ with differentials
$$\begin{array}{rcccl}
P_1= &\widehat{\mathbb{Z}}\ps{\Gamma}e \oplus \widehat{\mathbb{Z}}\ps{\Gamma}f_1\oplus \widehat{\mathbb{Z}}\ps{\Gamma}f_2 & \xrightarrow{\delta_{P,1}}& \widehat{\mathbb{Z}}\ps{\Gamma}e &= P_0\\
&e\oplus 0 \oplus 0 & \mapsto & (1-\Phi)e\\
&0\oplus f_1\oplus 0 &\mapsto & (a_1-1)e\\
&0\oplus 0\oplus f_2 &\mapsto & (a_2-1)e,
\end{array}$$
$$\begin{array}{rcccl}
P_2= &\widehat{\mathbb{Z}}\ps{\Gamma}f_1 \oplus \widehat{\mathbb{Z}}\ps{\Gamma}f_2\oplus \widehat{\mathbb{Z}}\ps{\Gamma}h & \xrightarrow{\delta_{P,2}}& \widehat{\mathbb{Z}}\ps{\Gamma}e \oplus \widehat{\mathbb{Z}}\ps{\Gamma}f_1\oplus \widehat{\mathbb{Z}}\ps{\Gamma}f_2  &= P_1\\
&f_1\oplus 0 \oplus 0 & \mapsto & (1-a_1)e\oplus (1-\Phi)f_1\oplus 0\\
&0\oplus f_2\oplus 0 &\mapsto & (1-a_2)e\oplus 0\oplus (1-\Phi)f_2\\
&0\oplus 0\oplus h &\mapsto & 0\oplus (a_2-1)f_1 \oplus (1-a_1)f_2,
\end{array}$$
and
$$\begin{array}{rcccl}
P_3= &\widehat{\mathbb{Z}}\ps{\Gamma}h & \xrightarrow{\delta_{P,3}}& \widehat{\mathbb{Z}}\ps{\Gamma}f_1 \oplus \widehat{\mathbb{Z}}\ps{\Gamma}f_2\oplus \widehat{\mathbb{Z}}\ps{\Gamma}h  &= P_2\\
& h &\mapsto & (a_2-1)f_1 \oplus (1-a_1)f_2 \oplus (1-\Phi)h.
\end{array}$$

Writing $A_\bullet = B_{\Gamma,\bullet}$ and $\partial_{A,m}=\partial_{\Gamma,m}$, we obtain a morphism of complexes
$$\xymatrix @C1.75pc {
P_\bullet\ar[d]_{\nu_\bullet} &\cdots \ar[r]^(.2){\delta_{P,3}} & \widehat{\mathbb{Z}}\ps{\Gamma}f_1 \oplus \widehat{\mathbb{Z}}\ps{\Gamma}f_2\oplus \widehat{\mathbb{Z}}\ps{\Gamma}h \ar[d]_{\nu_2}\ar[r]^{\delta_{P,2}} & \widehat{\mathbb{Z}}\ps{\Gamma}e \oplus \widehat{\mathbb{Z}}\ps{\Gamma}f_1\oplus \widehat{\mathbb{Z}}\ps{\Gamma}f_2 \ar[d]_{\nu_1}\ar[r]^(.7){\delta_{P,1}} & \widehat{\mathbb{Z}}\ps{\Gamma}e\ar[d]_{\nu_0}\\
A_\bullet & \cdots \ar[r]^(.4){\partial_{A,3}}& \widehat{\mathbb{Z}}\ps{\Gamma\times \Gamma\times \Gamma} \ar[r]^{\partial_{A,2}}& \widehat{\mathbb{Z}}\ps{\Gamma\times \Gamma} \ar[r]^(.6){\partial_{A,1}}& \widehat{\mathbb{Z}}\ps{\Gamma}
}$$
where
\begin{eqnarray}
\label{eq:nu0}
\nu_0(e) &=& 1,\\
\label{eq:nu10}
\nu_1(e\oplus 0\oplus 0) &=& (\Phi,1),\\
\label{eq:nu11}
\nu_1(0\oplus f_1\oplus 0) &=& (1,a_1),\\
\label{eq:nu12}
\nu_1(0\oplus 0\oplus f_2) &=& (1,a_2),\\
\label{eq:nu21}
\nu_2(f_1\oplus 0 \oplus 0) &=& (1,\Phi,1)- (1, a_1\Phi,a_1) +(1,1,a_1)-(1,\Phi,\Phi a_1),\\
\label{eq:nu22}
\nu_2(0 \oplus f_2 \oplus 0) &=& (1,\Phi,1)- (1, a_2\Phi,a_2) +(1,1,a_2)-(1,\Phi,\Phi a_2),\\
\label{eq:nu20}
\nu_2(0\oplus 0 \oplus h) &=& (1,a_2,a_2 a_1) - (1,1,a_1)  + (1,1,a_2) - (1,a_1,a_1a_2).
\end{eqnarray}
These equalities, and in particular (\ref{eq:nu21}) - (\ref{eq:nu20}), can be found using again property (\ref{eq:contracthom}) of the contracting homotopy $H_{\Gamma,m}$.

\begin{lemma}
\label{lem:cupelliptic}
Let $\mathcal{F}_n$ and $\tilde{\mathcal{F}}_n$ be as in Lemma $\ref{lem:cupalgclosure}$. We have an isomorphism
\begin{equation}
\label{eq:isocupell}
\rH^2(C,\mathcal{F}_n)=
\frac{\mathrm{Hom}_{\widehat{\mathbb{Z}}\ps{\Gamma}}(\widehat{\mathbb{Z}}\ps{\Gamma}f_1\oplus \widehat{\mathbb{Z}}\ps{\Gamma}f_2\oplus \widehat{\mathbb{Z}}\ps{\Gamma}h,\tilde{\mathcal{F}}_n)}{\mathrm{Image}(\delta_{P,2}^*)}
\end{equation}
where
\begin{equation}
\label{eq:diffcupell}
\delta_{P,2}^*: \mathrm{Hom}_{\widehat{\mathbb{Z}}\ps{\Gamma}}(\widehat{\mathbb{Z}}\ps{\Gamma}e\oplus \widehat{\mathbb{Z}}\ps{\Gamma}f_1\oplus \widehat{\mathbb{Z}}\ps{\Gamma}f_2,\tilde{\mathcal{F}}_n)\to \mathrm{Hom}_{\widehat{\mathbb{Z}}\ps{\Gamma}}(\widehat{\mathbb{Z}}\ps{\Gamma}f_1\oplus \widehat{\mathbb{Z}}\ps{\Gamma}f_2\oplus \widehat{\mathbb{Z}}\ps{\Gamma}h,\tilde{\mathcal{F}}_n)
\end{equation}
is given by precomposition with $\delta_{P,2}$.
Suppose $\alpha\in \rH^1(C,\mathbb{Z}/n)=\mathrm{Hom}(\Gamma,\mathbb{Z}/n)$ and $\beta\in \rH^1(C,\mathcal{F}_n)$ is given by
$$\left\{ \begin{array}{ccl}\beta\in \mathrm{Hom}(\Gamma,\mathbb{Z}/n)&\mbox{if}&\mathcal{F}_n=\mathbb{Z}/n,\\
{[c_b]}\in \rH^1(\Gamma,\tilde{\mu}_n)&\mbox{if}&\mathcal{F}_n=\mu_n,
\end{array}\right.$$
where $b\in D(C)$, $b^{1/n}$ is a chosen $n^{th}$ root of $b$ in $\overline{k(C)}$, and $c_b\in C^1(\Gamma,\tilde{\mu}_n)$ is the one-cocycle defined by $c_b(\gamma)=\gamma(b^{1/n})/b^{1/n}$ for all $\gamma\in \Gamma$. 
Then $\alpha\cup\beta\in \rH^2(C,\mathcal{F}_n)$ is represented by the homomorphism $\xi_{\alpha,\beta}\in \mathrm{Hom}_{\widehat{\mathbb{Z}}\ps{\Gamma}}(\widehat{\mathbb{Z}}\ps{\Gamma}f_1\oplus \widehat{\mathbb{Z}}\ps{\Gamma}f_2\oplus \widehat{\mathbb{Z}}\ps{\Gamma}h,\tilde{\mathcal{F}}_n)$ given by
\begin{equation}
\label{eq:Zn}
\left\{\begin{array}{rcl}
\xi_{\alpha,\beta}(f_1\oplus 0 \oplus 0) &=&\alpha(\Phi)\beta(a_1)-\alpha(a_1)\beta(\Phi),\\
\xi_{\alpha,\beta}(0\oplus f_2 \oplus 0) &=&\alpha(\Phi)\beta(a_2)-\alpha(a_2)\beta(\Phi),\\
\xi_{\alpha,\beta}(0\oplus 0 \oplus h) &=&\alpha(a_1)\beta(a_2) - \alpha(a_2)\beta(a_1),
\end{array}\right.\quad\mbox{ if}\quad \mathcal{F}_n=\mathbb{Z}/n,
\end{equation}
and by
\begin{equation}
\label{eq:mun}
\left\{\begin{array}{rcl}
\xi_{\alpha,\beta}(f_1\oplus 0 \oplus 0) &=&c_b(a_1)^{q\,\alpha(\Phi)}\cdot c_b(\Phi^{-1})^{q\,\alpha(a_1)},\\
\xi_{\alpha,\beta}(0\oplus f_2 \oplus 0) &=&c_b(a_2)^{q\,\alpha(\Phi)}\cdot c_b(\Phi^{-1})^{q\,\alpha(a_2)},\\
\xi_{\alpha,\beta}(0\oplus 0 \oplus h) &=&c_b(a_2)^{\alpha(a_1)} \cdot c_b(a_1)^{-\alpha(a_2)},
\end{array}\right.\quad\mbox{ if}\quad \mathcal{F}_n=\mu_n.
\end{equation}
\end{lemma}

\begin{proof}
Since $n=\ell^z$, we have $\rH^2(C,\mathcal{F}_n) = \rH^2(\Gamma,\tilde{\mathcal{F}}_n)$ by Lemma \ref{lem:achinger}, which implies the isomorphism (\ref{eq:isocupell}). If $\mathcal{F}_n=\mu_n$, then $\rH^1(C,\mu_n)=D(C)/(k(C)^*)^n$. By Lemma \ref{lem:secondgroup}, if $\beta\in\rH^1(C,\mu_n)$ then there exists an element $b\in D(C)$ such that $\beta=[c_b]$ for the one-cocycle $c_{b}\in C^1(\Gamma,\tilde{\mu}_n)$ given in the statement of Lemma \ref{lem:cupelliptic}.

Suppose first that $\mathcal{F}_n=\mathbb{Z}/n$. As above, we use the natural isomorphism $\mathbb{Z}/n\otimes_{\mathbb{Z}} \mathbb{Z}/n = \mathbb{Z}/n$ given by multiplication. Since $T_\ell(C)$ and $\Phi$ act trivially on $\mathbb{Z}/n$, so does $\Gamma$. Therefore, we have, by (\ref{eq:cupbar}), that $\alpha \cup \beta\in \rH^2(C,\mathbb{Z}/n)$ is represented by the two-cocycle $c_{\alpha,\beta}\in C^2(\Gamma,\mathbb{Z}/n)$ satisfying for all $\gamma_1,\gamma_2\in \Gamma$, 
$$c_{\alpha,\beta}([\gamma_1|\gamma_2]) = 
- \alpha(\gamma_1)\beta(\gamma_2).$$
Using $[\gamma_1|\gamma_2]=(1,\gamma_1,\gamma_1\gamma_2)$ and (\ref{eq:nu21}) - (\ref{eq:nu20}), we obtain that $\alpha\cup \beta$ is represented by the homomorphism $\xi_{\alpha,\beta}\in \mathrm{Hom}_{\widehat{\mathbb{Z}}\ps{\Gamma}}(\widehat{\mathbb{Z}}\ps{\Gamma}f_1 \oplus \widehat{\mathbb{Z}}\ps{\Gamma}f_2\oplus \widehat{\mathbb{Z}}\ps{\Gamma}h,\mathbb{Z}/n)$ given by
\begin{eqnarray}
\label{eq:ohyeah1}
\quad\xi_{\alpha,\beta}(f_1\oplus 0 \oplus 0) &=&
-\alpha(\Phi)\beta(\Phi^{-1}) + \alpha(a_1\Phi)\beta(\Phi^{-1}) - \alpha(1)\beta(a_1) + \alpha(\Phi)\beta(a_1)\\ 
\nonumber
&=&\alpha(\Phi)\beta(\Phi)-\alpha(a_1)\beta(\Phi)-\alpha(\Phi)\beta(\Phi)+ \alpha(\Phi)\beta(a_1)\\
\nonumber
&=&-\alpha(a_1)\beta(\Phi)+ \alpha(\Phi)\beta(a_1),\\
\label{eq:ohyeah2}
\quad\xi_{\alpha,\beta}(0\oplus f_2 \oplus 0) &=&
-\alpha(\Phi)\beta(\Phi^{-1}) + \alpha(a_2\Phi)\beta(\Phi^{-1}) - \alpha(1)\beta(a_2) +\alpha(\Phi)\beta(a_2)\\ 
\nonumber
&=&\alpha(\Phi)\beta(\Phi)-\alpha(a_2)\beta(\Phi)-\alpha(\Phi)\beta(\Phi)+ \alpha(\Phi)\beta(a_2)\\
\nonumber
&=&-\alpha(a_2)\beta(\Phi)+ \alpha(\Phi)\beta(a_2),\\
\label{eq:ohyeah3}
\quad\xi_{\alpha,\beta}(0\oplus 0 \oplus h) &=& 
-\alpha(a_2)\beta(a_1) + \alpha(1)\beta(a_1) - \alpha(1)\beta(a_2) + \alpha(a_1)\beta(a_2)\\
\nonumber
&=& -\alpha(a_2)\beta(a_1) + \alpha(a_1)\beta(a_2),
\end{eqnarray}
where the second equalities in (\ref{eq:ohyeah1}) - (\ref{eq:ohyeah3}) follow since $\alpha$ and $\beta$ are group homomorphisms. Therefore we obtain (\ref{eq:Zn}).

Suppose next that $\mathcal{F}_n=\mu_n$. As above, we use the natural isomorphism $\mathbb{Z}/n\otimes_{\mathbb{Z}} \tilde{\mu}_n = \tilde{\mu}_n$ given by exponentiation. We have that $T_\ell(C)$ acts trivially on $\tilde{\mu}_n$ and that $\Phi$ acts on $\tilde{\mu}_n$ as the $q^{\mathrm{th}}$ power map. Therefore, we have, by (\ref{eq:cupbar}), that $\alpha \cup \beta\in \rH^2(C,\mu_n)$ is represented by the two-cocycle $c_{\alpha,\beta}\in C^2(\Gamma,\tilde{\mu}_n)$ satisfying for all $\gamma_1,\gamma_2\in \Gamma$, 
$$c_{\alpha,\beta}([\gamma_1|\gamma_2]) = 
\left(\gamma_1 c_b(\gamma_2)\right)^{- \alpha(\gamma_1)}.$$
Using $[\gamma_1|\gamma_2]=(1,\gamma_1,\gamma_1\gamma_2)$ and (\ref{eq:nu21}) - (\ref{eq:nu20}), we obtain that $\alpha\cup \beta$ is represented by the homomorphism $\xi_{\alpha,\beta}\in \mathrm{Hom}_{\widehat{\mathbb{Z}}\ps{\Gamma}}(\widehat{\mathbb{Z}}\ps{\Gamma}f_1 \oplus \widehat{\mathbb{Z}}\ps{\Gamma}f_2\oplus \widehat{\mathbb{Z}}\ps{\Gamma}h,\tilde{\mu}_n)$ given by
\begin{eqnarray}
\label{eq:ohyeah1mu}
\quad\quad\xi_{\alpha,\beta}(f_1\oplus 0 \oplus 0) &=&
(\Phi c_b(\Phi^{-1}))^{-\alpha(\Phi)} \cdot (a_1\Phi c_b(\Phi^{-1}))^{\alpha(a_1\Phi)} \cdot (c_b(a_1))^{- \alpha(1)} \cdot (\Phi c_b(a_1))^{\alpha(\Phi)}\\ 
\nonumber
&=&(\Phi c_b(\Phi^{-1}))^{-\alpha(\Phi)} \cdot (\Phi c_b(\Phi^{-1}))^{\alpha(a_1)} \cdot (\Phi c_b(\Phi^{-1}))^{\alpha(\Phi)} \cdot (\Phi c_b(a_1))^{\alpha(\Phi)}\\
\nonumber
&=&c_b(\Phi^{-1})^{q\,\alpha(a_1)}\cdot c_b(a_1)^{q\,\alpha(\Phi)},\\
\label{eq:ohyeah2mu}
\quad\xi_{\alpha,\beta}(0\oplus f_2 \oplus 0) &=&
(\Phi c_b(\Phi^{-1}))^{-\alpha(\Phi)} \cdot (a_2\Phi c_b(\Phi^{-1}))^{\alpha(a_2\Phi)} \cdot (c_b(a_2))^{- \alpha(1)} \cdot (\Phi c_b(a_2))^{\alpha(\Phi)}\\ 
\nonumber
&=&(\Phi c_b(\Phi^{-1}))^{-\alpha(\Phi)} \cdot (\Phi c_b(\Phi^{-1}))^{\alpha(a_2)} \cdot (\Phi c_b(\Phi^{-1}))^{\alpha(\Phi)} \cdot (\Phi c_b(a_2))^{\alpha(\Phi)}\\
\nonumber
&=&c_b(\Phi^{-1})^{q\,\alpha(a_2)}\cdot c_b(a_2)^{q\,\alpha(\Phi)},\\
\label{eq:ohyeah3mu}
\quad\xi_{\alpha,\beta}(0\oplus 0 \oplus h) &=& (a_2 c_b(a_1))^{-\alpha(a_2)} \cdot (c_b(a_1))^{\alpha(1)} \cdot (c_b(a_2))^{- \alpha(1)} \cdot (a_1 c_b(a_2))^{\alpha(a_1)}\\
\nonumber
&=& c_b(a_1)^{- \alpha(a_2)} \cdot c_b(a_2)^{\alpha(a_1)}, 
\end{eqnarray}
where the second equalities in (\ref{eq:ohyeah1mu}) - (\ref{eq:ohyeah3mu}) follow since $\alpha$ is a group homomorphism and since $a_1$ and $a_2$ act trivially on $\tilde{\mu}_n$. Therefore we obtain (\ref{eq:mun}).
\end{proof}

\begin{proof}[Proof of Theorem $\ref{thm:WeilCup}$]
Let $\alpha\in\rH^1(C,\mathbb{Z}/n)_O$ and let $\beta\in \rH^1(C,\mathcal{F}_n)_O$.
We view $\alpha\in \mathrm{Hom}(\Gamma,\mathbb{Z}/n)$. By Lemma \ref{lem:normalizedcharacter}(i), we have $\alpha(\Phi)=\alpha(\Phi^{-1})=0$. 

Suppose first that $\mathcal{F}_n=\mathbb{Z}/n$. In this case, we use again Lemma \ref{lem:normalizedcharacter}(i) to view $\beta\in \mathrm{Hom}(\Gamma,\mathbb{Z}/n)$ with $\beta(\Phi)=\beta(\Phi^{-1})=0$. By Lemma \ref{lem:cupelliptic}, $\alpha \cup \beta\in \rH^2(C,\mathbb{Z}/n)$ is represented by the homomorphism $\xi_{\alpha,\beta}\in \mathrm{Hom}_{\widehat{\mathbb{Z}}\ps{\Gamma}}(\widehat{\mathbb{Z}}\ps{\Gamma}f_1 \oplus \widehat{\mathbb{Z}}\ps{\Gamma}f_2\oplus \widehat{\mathbb{Z}}\ps{\Gamma}h,\mathbb{Z}/n)$ given by (\ref{eq:Zn}). Since $\alpha(\Phi)=\beta(\Phi^{-1})=0$ and since $\overline{\alpha}$ and $\overline{\beta}$ are the restrictions of $\alpha$ and $\beta$, respectively, to $T_\ell(C)$, we obtain
\begin{eqnarray*}
\quad\xi_{\alpha,\beta}(f_1\oplus 0 \oplus 0) &=&0,\\
\quad\xi_{\alpha,\beta}(0\oplus f_2 \oplus 0) &=&0,\\
\quad\xi_{\alpha,\beta}(0\oplus 0 \oplus h) &=& \overline{\alpha}(a_1)\overline{\beta}(a_2) - \overline{\alpha}(a_2)\overline{\beta}(a_1).
\end{eqnarray*}
Therefore, (\ref{eq:Weil1}) implies that if $\overline{\alpha}\cup_{\overline{C}}\overline{\beta}$ is trivial in $\rH^1(\overline{C},\mathbb{Z}/n)$ then $\xi_{\alpha,\beta}$ is the trivial homomorphism. This implies Theorem \ref{thm:WeilCup} when $\mathcal{F}_n=\mathbb{Z}/n$.

Suppose next that $\mathcal{F}_n=\mu_n$. By Lemma \ref{lem:normalizedcharacter}(ii), there exists an element $b\in D(C)$ that is normalized at $O$ and an $n^{\mathrm{th}}$ root $b^{1/n}$ of $b$ in the completion $k(C)_O$ at $O$ such that $\beta$ is represented by the one-cocycle $c_{b}\in C^1(\Gamma,\tilde{\mu}_n)$ given by $c_b(\gamma)= \gamma(b^{1/n})/b^{1/n}$ for all $\gamma\in \Gamma$. Moreover, $c_b(\Phi)=c_b(\Phi^{-1})=1$. Since $T_\ell(C)$ acts trivially on $\tilde{\mu}_n$, it follows that the restriction of $c_b$ to $T_\ell(C)$ defines a homomorphism $T_\ell(C)\to \tilde{\mu}_n$ that is independent of the choice of $n^{\mathrm{th}}$ root of $b$ and that is therefore equal to the restriction $\overline{\beta}$ of $\beta$ to $\rH^1(\overline{C},\mu_n)=\mathrm{Hom}(T_\ell(C),\tilde{\mu}_n)$. By Lemma \ref{lem:cupelliptic}, $\alpha \cup \beta\in \rH^2(C,\mu_n)$ is represented by the homomorphism $\xi_{\alpha,\beta}\in \mathrm{Hom}_{\widehat{\mathbb{Z}}\ps{\Gamma}}(\widehat{\mathbb{Z}}\ps{\Gamma}f_1 \oplus \widehat{\mathbb{Z}}\ps{\Gamma}f_2\oplus \widehat{\mathbb{Z}}\ps{\Gamma}h,\tilde{\mu}_n)$ given by (\ref{eq:mun}). Since $\alpha(\Phi)=0$ and $c_b(\Phi^{-1})=1$, and since $T_\ell(C)$ acts trivially on $\tilde{\mu}_n$, we obtain
\begin{eqnarray*}
\quad\xi_{\alpha,\beta}(f_1\oplus 0 \oplus 0) &=&1,\\
\quad\xi_{\alpha,\beta}(0\oplus f_2 \oplus 0) &=&1,\\
\quad\xi_{\alpha,\beta}(0\oplus 0 \oplus h) &=& \overline{\beta}(a_2)^{\overline{\alpha}(a_1)} \cdot \overline{\beta}(a_1)^{- \overline{\alpha}(a_2)}.
\end{eqnarray*}
Therefore, (\ref{eq:Weil2}) implies that if $\overline{\alpha}\cup_{\overline{C}}\overline{\beta}$ is trivial in $\rH^1(\overline{C},\mu_n)$ then $\xi_{\alpha,\beta}$ is the trivial homomorphism. This implies Theorem \ref{thm:WeilCup} when $\mathcal{F}_n=\mu_n$.
\end{proof}

\begin{rem}
\label{rem:Legendre}
Lemma \ref{lem:cupelliptic} can also be used to get different presentations for the cup product when one of the arguments lies in $\rH^1(k,\mathbb{Z}/n)$ or $\rH^1(k,\mu_n)$, respectively.
\begin{enumerate}
\item[(i)] 
Suppose $\alpha\in \rH^1(k,\mathbb{Z}/n)$ and $\beta\in \rH^1(C,\mu_n)$. Then $\alpha(a_1)=\alpha(a_2)=0$. Therefore, the homomorphism $\xi_{\alpha,\beta}\in \mathrm{Hom}_{\widehat{\mathbb{Z}}\ps{\Gamma}}(\widehat{\mathbb{Z}}\ps{\Gamma}f_1 \oplus \widehat{\mathbb{Z}}\ps{\Gamma}f_2\oplus \widehat{\mathbb{Z}}\ps{\Gamma}h,\tilde{\mu}_n)$ given by (\ref{eq:mun}) satisfies
\begin{eqnarray*}
\xi_{\alpha,\beta}(f_1\oplus 0 \oplus 0) &=&\overline{\beta}(a_1)^{q\,\alpha(\Phi)},\\
\xi_{\alpha,\beta}(0\oplus f_2 \oplus 0) &=&\overline{\beta}(a_2)^{q\,\alpha(\Phi)},\\
\xi_{\alpha,\beta}(0\oplus 0 \oplus h) &=&1.
\end{eqnarray*}
Here the first two equalities follow, since the restriction of $c_b$ to $T_\ell(C)$ defines a homomorphism $T_\ell(C)\to \tilde{\mu}_n$ that is independent of the choice of $n^{\mathrm{th}}$ root of $b$ and that is therefore equal to the restriction $\overline{\beta}$ of $\beta$ to $\rH^1(\overline{C},\mu_n)=\mathrm{Hom}(T_\ell(C),\tilde{\mu}_n)$.

\item[(ii)]
Suppose $\alpha\in \rH^1(C,\mathbb{Z}/n)$ and $\beta\in \rH^1(k,\mu_n)$. Then $\beta$ is represented by a one-cocycle $c_b$ associated to an element $b\in k$, which means that any choice of $n^{\mathrm{th}}$ root of $b$ in $\overline{k(C)}$ lies in $\overline{k}$. Hence $c_b(a_1)=c_b(a_2)=1$. Therefore, the homomorphism $\xi_{\alpha,\beta}\in \mathrm{Hom}_{\widehat{\mathbb{Z}}\ps{\Gamma}}(\widehat{\mathbb{Z}}\ps{\Gamma}f_1 \oplus \widehat{\mathbb{Z}}\ps{\Gamma}f_2\oplus \widehat{\mathbb{Z}}\ps{\Gamma}h,\tilde{\mu}_n)$ given by (\ref{eq:mun}) satisfies
\begin{eqnarray*}
\xi_{\alpha,\beta}(f_1\oplus 0 \oplus 0) &=&c_b(\Phi^{-1})^{q\,\overline{\alpha}(a_1)},\\
\xi_{\alpha,\beta}(0\oplus f_2 \oplus 0) &=&c_b(\Phi^{-1})^{q\,\overline{\alpha}(a_2)},\\
\xi_{\alpha,\beta}(0\oplus 0 \oplus h) &=&1.
\end{eqnarray*}
Here the first two equalities follow, since the restriction of $\alpha$ to $T_\ell(C)$ is equal to $\overline{\alpha}$.
\end{enumerate}
In both cases (i) and (ii), $\xi_{\alpha,\beta}$ induces a homomorphism in $\mathrm{Hom}_{\widehat{\mathbb{Z}}\ps{\Gamma}}(\widehat{\mathbb{Z}}\ps{\Gamma}f_1 \oplus \widehat{\mathbb{Z}}\ps{\Gamma}f_2,\tilde{\mu}_n)$. Since $T_\ell(C)$ acts trivially on $\tilde{\mu}_n$, it follows that the class of $\xi_{\alpha,\beta}$ in the quotient (\ref{eq:isocupell}) defines an element in the quotient group
$$\frac{\mathrm{Hom}_{\widehat{\mathbb{Z}}\ps{T_\ell(C)}}(\widehat{\mathbb{Z}}\ps{T_\ell(C)}f_1 \oplus \widehat{\mathbb{Z}}\ps{T_\ell(C)}f_2,\tilde{\mu}_n)}{(\Phi-1)\mathrm{Hom}_{\widehat{\mathbb{Z}}\ps{T_\ell(C)}}(\widehat{\mathbb{Z}}\ps{T_\ell(C)}f_1 \oplus \widehat{\mathbb{Z}}\ps{T_\ell(C)}f_2,\tilde{\mu}_n)} = \frac{\rH^1(\overline{C},\mu_n)}{(\Phi-1)\rH^1(\overline{C},\mu_n)} $$
which equals $\rH^1(\mathrm{Gal}(\overline{k}/k),\rH^1(\overline{C},\mu_n))$. On the other hand, at least one of $\alpha$ or $\beta$ has trivial restriction to $\overline{C}$, so $\alpha \cup \beta$ restricts to $0$ in $\rH^2(\overline{C},\mu_n)$.  Thus the above expression for $\xi_{\alpha,\beta}$ is consistent with the split exact sequence
\begin{equation}
\label{eq:spliteq}
0\to \rH^1(\mathrm{Gal}(\overline{k}/k),\rH^1(\overline{C},\mu_n)) \to \rH^2(C,\mu_n) \to \rH^0(\mathrm{Gal}(\overline{k}/k),\rH^2(\overline{C},\mu_n)) \to 0
\end{equation}
that results from the degeneration of the spectral sequence
$$\rH^p(\mathrm{Gal}(\overline{k}/k),\rH^q(\overline{C},\mu_n))
\Rightarrow \rH^{p+q}(C,\mu_n).$$
This gives another approach to computing $\alpha \cup \beta$ than the Legendre transform method in Theorems \ref{thm:legendre1}   and  \ref{thm:legendre2}.
\end{rem}

\begin{rem}
\label{rem:normalstuff}
Since the point $O$ has degree $d(O)$ prime to $\ell \cdot (\ell - 1)$,  there is a canonical splitting of the sequence (\ref{eq:spliteq}) sending $d(O) \in \mathbb{Z}/n = \rH^2(\overline{C},\mu_n)^{\mathrm{Gal}(\overline{k}/k)}$ to the class $[O]$ of $O$ in $\rH^2(C,\mu_n) = \mathrm{Pic}(C)/n\cdot \mathrm{Pic}(C)$.   Theorem \ref{thm:EllThm} shows that if $\alpha\in \rH^1(C,\mathbb{Z}/n)_O$ and $\beta\in \rH^1(C,\mu_n)_O$, then $\alpha \cup \beta$ lies in the image of this splitting.
\end{rem}

\section{An infinite family of genus two examples}
\label{s:genustwo}
\setcounter{equation}{0}

The goal of this section is to construct an infinite family of curves of genus $2$ for which the equivalent conditions of  Theorem \ref{thm:cupsizeresult} do not hold when $n=\ell = 3$.  We introduce the following notation.

\begin{note}
\label{not:genus2}
Let $\zeta \in \mathbb{C}$ be a primitive cube root of unity, and let $F = \mathbb{Q}(\zeta) = \mathbb{Q}(\sqrt{-3})$. 
\begin{itemize}
\item[(a)]
Let $E$ be the elliptic curve over $F$ defined by the affine equation $y^2 = x^3 - 3$, and let $\pi:E \to \mathbb{P}^1_{F}$ be the morphism associated to the inclusion of function fields $F(x) \subset F(E) = F(x)[y]/(y^2 - x^3 + 3)$. 
\item[(b)] 
Let $Y$ be the smooth projective curve with function field $F(E)(x^{1/2})$, and let $\theta:Y  \to E$ be the natural morphism associated to the containment $F(E)\subset F(Y)$.
\item[(c)]
Define $\tilde{y} = y/x^{3/2} \in F(Y)$ and $w=x^{-1}$. Let $E'$ be the elliptic curve over $F$ defined by $\tilde{y}^2 = 1 - 3w^3$, and let $\pi':E' \to \mathbb{P}^1_{F}$ be the morphism associated to the inclusion of function fields $F(x)=F(w) \subset F(E') = F(w)[\tilde{y}]/(\tilde{y}^2 - 1 + 3w^3)$. Moreover, let $\theta':Y  \to E'$ be the natural morphism associated to the containment $F(E')\subset F(Y)$.
\item[(d)] If $N$ is a finite extension of $F$, we define $Y_N = N \otimes_F Y$, $E_N = N \otimes_F E$ and $E'_N = N\otimes_F E'$. 
\end{itemize}
\end{note}

We collect some obvious properties of $Y$, $E$ and $E'$ in the following remark.

\begin{rem}
\label{rem:genus2}
Both $\theta$ and $\theta'$ are covers of degree 2. 
It follows that $F(Y)$ is a biquadratic extension of $F(x)=F(w)$ with intermediate fields $F(E)$, $F(x^{1/2})$ and $F(E')$.  
Moreover, $\theta$ is ramified over the points $Q_1 = (0,\sqrt{-3})$ and $Q_2 = (0,- \sqrt{-3})$ in $(x,y)$ coordinates on $E$, so $Y$ has genus $2$ by the Hurwitz formula.
The point $x = \infty$ on $\mathbb{P}^1_F$ is a branch point of $\pi:E \to \mathbb{P}^1_F$ and splits under $\pi':E' \to \mathbb{P}^1_F$.  So there are two points $0_Y$ and $0'_Y$ of $Y$ over $x =\infty $ on $\mathbb{P}^1_F$.
\end{rem}

\begin{lemma}
\label{lem:picnice}  
Let $N$ be a finite extension of $F$.  The direct image homomorphisms $\theta_*:\mathrm{Pic}^0(Y_N) \to \mathrm{Pic}^0(E_N)$ and $\theta'_*:\mathrm{Pic}^0(Y_N) \to \mathrm{Pic}^0(E'_N)$ give a homomorphism
\begin{equation}
\label{eq:Picnice}
\theta_* \times \theta'_*:\mathrm{Pic}^0(Y_N) \to \mathrm{Pic}^0(E_N) \times \mathrm{Pic}^0(E'_N)
\end{equation}
whose kernel and cokernel are finite groups annihilated by $2$.
\end{lemma}

\begin{proof}  
We have pullback maps $\theta^*:\mathrm{Pic}^0(E_N) \to \mathrm{Pic}^0(Y_N)$ and $\theta'^*:\mathrm{Pic}^0(E'_N) \to \mathrm{Pic}^0(Y_N)$ such that $\theta_* \circ \theta^*$ and $\theta'_* \circ \theta'^*$ are multiplication by $2$.  Furthermore, $\theta'_* \circ \theta^* = \pi'^* \circ \pi_*$ and $\theta_* \circ \theta'^* = \pi^* \circ \pi'_*$ are trivial since $\mathrm{Pic}^0(\mathbb{P}^1_{N})$ is trivial. Hence the composition
$$(\theta_* \times \theta'_* )\circ (\theta^* \times \theta'^*): \mathrm{Pic}^0(E_N) \times \mathrm{Pic}^0(E'_N) \to \mathrm{Pic}^0(E_N) \times \mathrm{Pic}^0(E'_N)$$
is multiplication by $2$.  Since all the groups appearing in (\ref{eq:Picnice}) are finitely generated abelian groups, it will suffice to show that the kernel $\mathcal{K}$  of $\theta_* \times \theta'_*$ is annihilated by $2$.  View $Y_N$ as a Galois cover of $\mathbb{P}^1_N$ with Galois group $G$ a Klein four group.  The three intermediate quadratic covers are $E_N$, $E'_N$ and the projective line over $N$ with function field $N(x^{1/2})$.  Since the latter projective line has trivial Jacobian, we see $\mathcal{K}$ is annihilated by the group ring element $1 + \sigma$ for each non-trivial $\sigma \in G$.  The sum of these three group ring elements is $2 + \mathrm{Trace}_G$.  Hence $2$ annihilates $\mathcal{K}$ because the action of $\mathrm{Trace}_G$ on $\mathrm{Pic}^0(Y_N)$ factors through $\mathrm{Pic}^0(\mathbb{P}^1_N) = 0$.
\end{proof}

The following lemma is clear from the functorial properties of Jacobians. 

\begin{lemma}
\label{cor:ecor}  
Let $X$ be a geometrically integral curve of genus $1$ over a perfect field $K$, and suppose $\tilde{0}_X$ is an arbitrary point of $X(K)$.  Then $X$ becomes an elliptic curve with origin $\tilde{0}_X$ via the morphism from $X$ to its Jacobian sending $\tilde{0}_X$ to the origin.  
\begin{enumerate}
\item[(i)]  
Suppose $ 1 \le n \in \mathbb{Z}$. The $n$-torsion of $\mathrm{Pic}^0(X)$ has order $n^2$ if and only  if the set of elements $P \in X(K)$ such that $nP = \tilde{0}_X$ with respect to the group law on $X$ has order $n^2$.
\item[(ii)]  Suppose $P_1, P_2 \in X(K)$.   There is a class $\mathfrak{d} \in \mathrm{Pic}^0(X)$ such that $n\mathfrak{d} = [P_1] - [P_2]$ in $\mathrm{Pic}^0(X)$ if and only if there is a point $P \in X(K)$ such that $nP = P_1 - P_2$ with respect to the group law on $X$.
\end{enumerate}
\end{lemma}
 
We now return to our curves $E$, $E'$ and $Y$ from Notation \ref{not:genus2}. 
 
\begin{lemma}
\label{lem:deffields}
Let $N$ be a finite extension of $F$.
\begin{enumerate}
\item[(i)]  The $3$-torsion of $\mathrm{Pic}^0(E_N)$ has order $9$ if and only if $F(12^{1/3}) \subset N$.
\item[(ii)]  The $3$-torsion of $\mathrm{Pic}^0(E'_N)$ has order $9$ if and only if $F((4/3)^{1/3}) \subset N$.
\item[(iii)]  Let  $0_Y$ and $0'_Y$ be the two points of $Y$ defined in Remark $\ref{rem:genus2}$. Then $\{0_Y,0'_Y\}$ is the inverse image under $\theta:Y \to E$ of the point $0_E$ at infinity on the curve $E:y^2 = x^3 - 3$.  We can label these points so that under $\theta':Y \to E'$ one has $\theta'(0_Y) = (0,1) = P_1$ and $\theta'(0'_Y) = (0,-1) = P_2$ relative to the $(w,\tilde{y})$ coordinates of the curve $E':\tilde{y}^2 = 1 - 3w^3$.  The divisor class $[P_1] - [P_2]$ in $\mathrm{Pic}^0(E'_N)$ lies in $3 \cdot \mathrm{Pic}^0(E'_N)$ if and only if $F((4/3)^{1/3},\zeta^{1/3}) \subset N$.
 \end{enumerate}
 \end{lemma}

\begin{proof} 
The complex multiplication of $\zeta \in F$ on $E:y^2 = x^3 - 3$ is defined by $\zeta\cdot (x,y) = (\zeta x, y)$, while $-(x,y) = (x, -y)$ in the group law of $E$.  Thus $(x,y) = (0,\sqrt{-3})$ and $(x,y) = (0,-\sqrt{-3})$ are fixed by $\zeta$ and therefore sent to $0$ in the group law of $E$ by $(\zeta - 1)$.  Since $(\zeta^2 - 1)(\zeta - 1) = 3$ in $F$, these points are $3$-torsion points of $E$ over $\mathbb{Q}(\sqrt{-3}) = F$.  To find the remaining $3$-torsion points over $\overline{F}$, we just need a point $(x,y) \in E(\overline{F})$ such that $(\zeta^2 - 1) (x,y) = (0,\sqrt{-3})$. Here $\zeta^2 (x,y) = (\zeta^2x,y)$, $-(x,y) = (x,-y)$ and $-(0,\sqrt{-3}) = (0, -\sqrt{-3})$.  So we are looking for a line of the form $cx + dy - e = 0$ containing three points of the form $(\zeta^2 x,y)$, $(x,-y)$ and $(0,-\sqrt{-3})$. Furthermore, one should not have $x = 0$, since such a solution leads to the previous points $(0,\sqrt{-3})$ and $(0,\sqrt{-3})$. Writing out the constraints on $c$, $d$ and $e$ we find $d(-\sqrt{-3}) = e$ and that there is a matrix equation
$$\left(\begin{array}{cr} \zeta^2 x & y + \sqrt{-3}\\ x& -y + \sqrt{-3}\end{array}\right)\cdot  \begin{pmatrix} c \\ d\end{pmatrix} =  \begin{pmatrix} 0 \\ 0\end{pmatrix}.  $$
Since $c$ and $d$ are not both $0$ we find from this that
$$0 = \mathrm{det} \left(\begin{array}{cr} \zeta^2 x & y + \sqrt{-3}\\ x& -y + \sqrt{-3}\end{array}\right) = y(-\zeta^2 x -x) + \sqrt{-3}(\zeta^2 -1)x .$$
Since $x \ne 0$ and $\zeta^2 + 1 = -\zeta$ we can divide by $x$ to have
$$y = -\zeta^2 \sqrt{-3}(\zeta^2 -1).$$ 
On squaring we get
$$ x^3 - 3 = y^2 = -3 \zeta (\zeta^2 - 1)^2 = -3 \zeta (\zeta^4 - 2\zeta^2 + 1) = -3 (\zeta^5 - 2\zeta^3 + \zeta) = -3(\zeta^2 - 2 + \zeta) = 9.$$
Thus 
$$x^3=12 \quad \mathrm{and} \quad y = \pm 3$$
which shows part (i) of Lemma \ref{lem:deffields}.

For part (ii) we proceed similarly using the model $E':\tilde{y}^2 = 1 - 3w^3$, with complex multiplication defined by $\zeta(w,\tilde{y}) = (\zeta w, \tilde{y})$ and $-(w,\tilde{y}) = (w, -\tilde{y})$.  As the origin of $E'$ we will use the point $\tilde{0}_{E'}$ at infinity relative to the above affine $(w,\tilde{y})$ model of $E'$.  Note that because of Lemma \ref{cor:ecor} this does not make a difference as far as part (ii) of Lemma \ref{lem:deffields} is concerned. The points $(w,\tilde{y}) = (0,1)$ and $(0,-1)$ are fixed by $\zeta $ and hence annihilated by $\zeta - 1$, so they are $3$-torsion points.  To find the remaining $3$-torsion points, we look for a $(w_0,\tilde{y}_0)$ with $(\zeta^2 - 1)\cdot (w_0,\tilde{y}_0) = (0,1)$.  Thus we need a line $cw_0 + d\tilde{y}_0 -e = 0$ containing the points $(\zeta^2 w_0,\tilde{y}_0)$, $(w_0,-\tilde{y}_0)$ and $-(0,1) = (0,-1)$.  This implies $-d-e = 0$ so $-e = d$ and
$$\left(\begin{array}{cr} \zeta^2 w_0 & \tilde{y}_0 + 1\\ w_0& -\tilde{y}_0 + 1\end{array}\right)\cdot  \begin{pmatrix} c \\ d\end{pmatrix} =  \begin{pmatrix} 0 \\ 0\end{pmatrix}. $$
Since $c$ and $d$  are not both $0$, this gives
$$0 = \mathrm{det} \left(\begin{array}{cr} \zeta^2 w_0  & \tilde{y}_0 + 1\\ w_0& -\tilde{y}_0 + 1\end{array}\right) = \tilde{y}_0(-\zeta^2 w_0 -w_0) + (\zeta^2 -1)w_0.$$
The solution we are looking for does not have $w_0 = 0$, so we can divide by $w_0$ and then square to find
$$1 - 3w_0^3 = \tilde{y}_0^2 = \left ( \frac{\zeta^2 - 1}{\zeta^2 + 1}\right ) ^2 = \zeta^{-2} (\zeta^4 - 2 \zeta^2 + 1) = (\zeta^2 - 2 + \zeta) = -3.$$
Thus
$$w_0^3 = 4/3 \quad \mathrm{and}\quad \tilde{y}_0 = \pm \left ( \frac{\zeta^2 - 1}{\zeta^2 + 1}\right )   = \mp \sqrt{-3} \in F$$
where we set $\sqrt{-3} = 2\zeta + 1$, which leads to part (ii) of Lemma \ref{lem:deffields}.

Finally, for part (iii), note that we have shown $P_1 = (0,1)$ and $P_2 = (0,-1)$ are $3$-torsion points on $E'$ with $P_2 = -P_1$, so $P_1 - P_2 = 2 P_1$ and $P_1 = 2 \cdot (2P_1)$.  So in view of Lemma \ref{cor:ecor}, it will suffice to determine the extension of $F$ generated by the coordinates of a point $Q$ with $3Q = P_1 = (0,1)$ when the group law on $E'$ is the one coming from the map to the Jacobian that sends the point $\tilde{0}_{E'}$ at infinity on $E'$ to the origin.  We have found above a point $(w_0,\tilde{y}_0)$ with $(\zeta^2 - 1) (w_0,\tilde{y}_0) = (0,1)$.  This point has $w_0^3 = 4/3$ and $\tilde{y}_0$ a particular square root of $-3$ depending on $\zeta$.  So we now look for a point $Q = (w_1,\tilde{y}_1)$ with $(\zeta - 1)(w_1,\tilde{y}_1) = (w_0,\tilde{y}_0)$;   then $3 Q = (\zeta^2 - 1)(\zeta -1)(w_1,\tilde{y}_1) = (0,1)$.  Here $\zeta (w_1,\tilde{y}_1) = (\zeta w_1, \tilde{y}_1)$ and $-(w_1,\tilde{y}_1) = (w_1,-\tilde{y}_1)$ and $-(w_0,\tilde{y}_0) = (w_0, -\tilde{y}_0)$.  Hence we are looking for a line
\begin{equation}
\label{eq:lastline}
c(w - w_0) + d(\tilde{y} + \tilde{y}_0) - e = 0
\end{equation}
containing $(\zeta w_1, \tilde{y}_1)$, $(w_1,-\tilde{y}_1)$ and $(w_0, -\tilde{y}_0)$. Here if $w_1 = w_0$ then $(w_1,\tilde{y}_1) = \pm (w_0,\tilde{y}_0)$ and 
$$(\zeta - 1) (w_1,\tilde{y}_1) = \pm (\zeta - 1) (w_0,\tilde{y}_0) = \mp \zeta (\zeta^2 -1)(w_0,\tilde{y}_0) = \mp \zeta (0,1) =  (0,\mp 1) \ne (w_0,\tilde{y}_0).$$
Thus we can assume $w_1 \ne w_0$, and similarly we can assume $\zeta w_1 \ne w_0$, so the three points   $(\zeta w_1, \tilde{y}_1)$, $(w_1,-\tilde{y}_1)$ and $(w_0, -\tilde{y}_0)$ are distinct.  In order for the point $(w_0,-\tilde{y}_0)$ to be on the line (\ref{eq:lastline}) we must have $-e = 0$.  The remaining two points $(\zeta w_1, \tilde{y}_1)$, $(w_1,-\tilde{y}_1)$ are on the line if and only if
$$\left(\begin{array}{cr} \zeta w_1 - w_0 & \tilde{y}_1 + \tilde{y}_0\\ w_1 - w_0& -\tilde{y}_1 + \tilde{y}_0\end{array}\right)\cdot  \begin{pmatrix} c \\ d\end{pmatrix} =  \begin{pmatrix} 0 \\ 0\end{pmatrix}  .$$
Since $c$ and $d$ are not both $0$, we conclude
\begin{eqnarray*}
0 = \mathrm{det} \left(\begin{array}{cr} \zeta w_1 - w_0 & \tilde{y}_1 + \tilde{y}_0\\ w_1 - w_0& -\tilde{y}_1 + \tilde{y}_0\end{array}\right) &=& (\zeta w_1 - w_0)\cdot (-\tilde{y}_1 + \tilde{y}_0) - ( \tilde{y}_1 + \tilde{y}_0) \cdot (w_1 - w_0) \\
&=& \tilde{y}_1((-\zeta -1)w_1 + 2w_0) + (\zeta - 1)w_1 \tilde{y}_0.
\end{eqnarray*}
Since $-\zeta - 1 = \zeta^2$ we get 
\begin{equation}
\label{eq:nicely}
\tilde{y}_1 (\zeta^2 w_1 + 2w_0) = (1 -\zeta) w_1 \tilde{y}_0
\end{equation}
so on squaring this we find
$$(1 - 3w_1^3)(\zeta^2 w_1 + 2w_0)^2 =   (1 - \zeta)^2 w_1^2 (-3).$$
Writing $(\zeta^2 w_1 + 2w_0)^2  = \zeta^4 w_1^2 + 4\zeta^2 w_0 w_1 + 4 w_0^2$ we end up with an equality
$$-3 (\zeta^4 w_1^5 + 4\zeta^2 w_0 w_1^4 + 4 w_0^2 w_1^3) + (\zeta^4 + 3(1 -\zeta)^2) w_1^2 + 4\zeta^2 w_0 w_1 + 4 w_0^2 = 0.$$
We know $w_1 \ne w_0$ and $w_1 \ne \zeta^{-1} w_0 = \zeta^2 w_0$, so we divide the left hand side by $(w_1 - w_0)(w_1 - \zeta^2 w_0) = w_1^2 -(1 + \zeta^2) w_0 w_1 + \zeta^2 w_0^2 = w_1^2 + \zeta w_0 w_1 + \zeta^2 w_0^2$.  This gives
$$-3\zeta w_1^3 - 9 \zeta^2 w_0 w_1^2 + 4\zeta = 0.$$
Dividing by $-3 \zeta$ we obtain
\begin{equation}
\label{eq:needtheroots}
w_1^3 + 3\zeta w_0 w_1^2 - 4/3 = 0.
\end{equation}
We now write down the roots of (\ref{eq:needtheroots}) using Cardano's formulas. Define
$$a = 3 \zeta w_0, \quad b = 0, \quad c = -4/3,$$
and write 
\begin{eqnarray*}
r &=& \frac{1}{3}(3b - a^2) =\frac{1}{3} (-9 \zeta^2 w_0^2) = -3 \zeta^2 w_0^2= -3 \zeta^2 (4/3)^{2/3} ,
\\
s &=& \frac{1}{27}(2 a^3 - 9ab + 27 c) = 2 w_0^3 - 4/3 = 2 (4/3) - 4/3 = 4/3,
\\
D &=& - 4 r^3 - 27 s^2  = -4 (-27 w_0^6) - 27 (4/3)^2 = 4 \cdot 27 \cdot (4/3)^2  - 27 \cdot (4/3)^2  = 4^2 3^2,
\\
A^3 &=& \frac{-27 s}{2} + \frac{3}{2} \sqrt{-3D} = - \frac{27}{2} \cdot \frac{4}{3} + \frac{3}{2} \sqrt{-3^3 4^2} =  36 \zeta,
\\
B^3 &=& \frac{-27 s}{2} - \frac{3}{2} \sqrt{-3D} = 36 \zeta^2.
\end{eqnarray*}
The roots $w_1$ of (\ref{eq:needtheroots}) are then 
$$ \frac{A+B}{3};\quad \frac{\zeta^2 A + \zeta B}{3};\quad \frac{\zeta A + \zeta^2 B}{3}.$$
We have from (\ref{eq:nicely}) that
$$\tilde{y}_1/w_1 = (1- \zeta) \tilde{y}_0 /(\zeta^2 w_1 + 2w_0).$$
Since $\tilde{y}_0 = \pm \sqrt{-3} \in F$, we conclude that $w_1$ and $\tilde{y}_1$ generate the same extension over $F$ as $w_1$ and $w_0$ do.  So the extension of $F$ generated by $w_1$ and $\tilde{y}_1$ is $ F((4/3)^{1/3},\zeta^{1/3})$. This completes the proof of part (iii) of Lemma \ref{lem:deffields}.
\end{proof}

\begin{cor}
\label{cor:modp}  
Let $\mathcal{E}$, $\mathcal{E}'$ and $\mathcal{Y}$ be the minimal projective models over $\mathbb{Z}$ of the curves  over $\mathbb{Z}$ associated to the affine equations used to define $E$, $E'$ and $Y$.  We assume $q$ is a rational prime such that for $\mathcal{C} =  \mathcal{E}, \mathcal{E}'$ or $\mathcal{Y}$, $\mathcal{C}$ has good reduction $\mathcal{C}_q$ modulo $q$.  Suppose $q$ splits in the field $N = F(4^{1/3},3^{1/3})$ but does not split in the extension $N(\zeta^{1/3})$.  
\begin{enumerate}
\item[(i)]  The $3$-torsion subgroup of $\mathrm{Pic}^0(\mathcal{Y}_q)$ is isomorphic to $(\mathbb{Z}/3)^4 = (\mathbb{Z}/3)^{2 g(\mathcal{Y}_q)}$.
\item[(ii)] The point $0_{\mathcal{E}_q}$ at infinity associated to the affine model $\mathcal{E}_q:y^2 = x^3 - 3$ splits into two points $0_{\mathcal{Y}_q}$ and $0'_{\mathcal{Y}_q}$ via the morphism $\mathcal{Y}_q \to \mathcal{E}_q$ associated to the field embedding $(\mathbb{Z}/q)(\mathcal{E}_q) \subset (\mathbb{Z}/q)(\mathcal{Y}_q)$.  
\item[(iii)] The divisor class $[0_{\mathcal{Y}_q}] - [0'_{\mathcal{Y}_q}]$ in $\mathrm{Pic}^0(\mathcal{Y}_q)$ does not lie in $3 \cdot \mathrm{Pic}(\mathcal{Y}_q)$.
\end{enumerate}
\end{cor}

\begin{proof} 
By Lemma \ref{lem:deffields}, the $3$-torsion on the general fibers of $\mathcal{E}$ and $\mathcal{E}'$ is defined over $N$, so these torsion points define sections of the natural morphisms $\mathcal{E}_N = \mathrm{Spec}(O_N) \otimes_{\mathbb{Z}} \mathcal{E} \to \mathrm{Spec}(O_N)$ and $\mathcal{E}'_N = \mathrm{Spec}(O_N) \otimes_{\mathbb{Z}} \mathcal{E}' \to \mathrm{Spec}(O_N)$.  Since $q$ must be larger than $3$, these sections specialize to distinct $3$-torsion points of the fibers of $\mathcal{E}_N$ and $\mathcal{E}'_N$ over a prime of $O_N$ over $q$.  Since $q$ splits in $O_N$, these fibers are isomorphic to $\mathcal{E}_q$ and $\mathcal{E}'_q$, respectively.  This shows part (i) because the same arguments used in the proof of Lemma \ref{lem:picnice} show the natural direct image homomorphism  $\mathrm{Pic}^0(\mathcal{Y}_q) \to \mathrm{Pic}^0(\mathcal{E}_q) \times \mathrm{Pic}^0(\mathcal{E}'_q)$ has kernel and cokernel equal to finite abelian groups annihilated by $2$.  Moreover, part (ii) follows from the corresponding fact on the generic fibers of $\mathcal{Y}$ and $\mathcal{E}$.  

Finally, for (iii), it suffices to show that the image $\mathfrak{d}$ of $[0_{\mathcal{Y}_q}] - [0'_{\mathcal{Y}_q}]$ in $\mathrm{Pic}^0(\mathcal{E}'_q)$ does not lie in $3 \cdot \mathrm{Pic}^0(\mathcal{E}'_q)$.  Setting $w = x^{-1}$ defines the affine model $\mathcal{E}'_q: \tilde{y}^2 = 1 - 3w^3$ when $\tilde{y}= y/x^{3/2}$.   Let $0_{\mathcal{E}'_q}$ be the point at infinity for this model.  Then $\mathfrak{d} = [P_{1,q}] - [P_{2,q}]$ when $P_{1,q} = (0,1)$ and $P_{2,q} = (0,-1)$ in $(w,\tilde{y})$ coordinates. Here $[P_{2,q}] = -[P_{1,q}]$ in $\mathrm{Pic}^0(\mathcal{E}'_q)$ and $(\zeta - 1)[P_{1,q}] = 0$ relative to the complex multiplication action of $\mathbb{Z}[\zeta] = O_F$ on $\mathcal{E}'_q$ defined by $\zeta (w,\tilde{y}) = (\zeta w ,\tilde{y})$.  So $[P_{1,q}]$ is a $3$-torsion point when we use $0_{\mathcal{E}'_q}$ as the origin of the group law of $\mathcal{E}'_q$.   Lemma \ref{cor:ecor} shows that $\mathfrak{d} = [P_{1,q}] - [P_{2,q}] $ lies in $3 \cdot \mathrm{Pic}(\mathcal{E}'_q)$ if and only if $P_{1,q} - P_{2,q} = 2 P_{1,q}$ lies in $3 \cdot \mathcal{E}'_q(\mathbb{Z}/q)$ relative to the group law of $\mathcal{E}'_q(\mathbb{Z}/q)$.  Since $P_{1,q}$ is a $3$-torsion point, this will be true if and only if $P_{1,q}$ lies in $3 \cdot \mathcal{E}'_q(\mathbb{Z}/q)$.  However, $P_{1,q}$ is the intersection of the fiber of $\mathcal{E}'_N$ over a chosen prime $\mathfrak{q}$ of $O_N$ over $q$ with the corresponding point $P_1$ on the general fiber of $\mathcal{E}'_N$, i.e., the point $P_1$ with $(w,\tilde{y})$ coordinates $(0,1)$.  Multiplication by $3$ defines an \'etale morphism of abelian schemes $\mathbb{Z}[\frac{1}{6}] \otimes \mathcal{E}'_N \to \mathbb{Z}[\frac{1}{6}] \otimes \mathcal{E}'_N$.  The pullback of the section defined by $P_1$ is a divisor $\mathcal{P}$ on $\mathbb{Z}[\frac{1}{6}] \otimes \mathcal{E}'_N$  which is \'etale over this section.  Therefore $\mathcal{P}$ must be a disjoint union of divisors of the form $\mathrm{Spec}(\mathcal{O})$ in which $\mathcal{O}$ is the integral closure of $\mathbb{Z}[\frac{1}{6}]$ in the residue field $L$ of a closed point $P'$ of $E'_N$ such that $3 \cdot P' = P_1$.  By Lemma \ref{lem:deffields}, $L$ must contain $N(\zeta^{1/3})$, and by construction $\mathfrak{q}$ is a prime of $O_N$ inert to $N(\zeta^{1/3})$ and $\mathfrak{q}$ is prime to $6$.  Any point of the fiber of $\mathcal{E}'_N$ over $\mathfrak{q}$ which when multiplied by $3$ gives $P_{1,q}$ must lie on the intersection of $\mathcal{P}$ with this fiber. Because $N(\zeta^{1/3}) \subset L$,  there is no point of this intersection with residue field $O_N/\mathfrak{q} = \mathbb{Z}/q$. This implies (iii). 
\end{proof} 

\begin{rem} 
By the Cebotarev density theorem, the set of rational primes $q$ that have the properties in Corollary $\ref{cor:modp}$  has Dirichlet density $1/18 - 1/54 = 1/27$.  The prime $q = 439$ is an example.\end{rem}

\begin{thm}  
\label{thm:genus2examples}
The equivalent conditions of Theorem $\ref{thm:cupsizeresult}$ need not hold if $C$ is allowed to have genus greater than $1$.  More specifically, let $n=\ell=3$. With the notation of Corollary $\ref{cor:modp}$, let $C$ be the curve $\mathcal{Y}_q$ and let $O$ be either one of the points $0_{\mathcal{Y}_q}$ or $0'_{\mathcal{Y}_q}$.  Let $\overline{\mathcal{E}_q} = \overline{k} \otimes_k \mathcal{E}_q$ and $\overline{C} = \overline{k} \otimes_k C$. There are normalized classes $\alpha\in \rH^1(\mathcal{E}_q,\mathbb{Z}/3)_{0_{\mathcal{E}_q}}$ and $[b]\in\rH^1(\mathcal{E}_q,\mu_3)_{0_{\mathcal{E}_q}}$ such that the restrictions $\overline{\alpha}$ and $\overline{[b]}$ to $\overline{\mathcal{E}_q}$ satisfy $\overline{\alpha}\cup_{\overline{\mathcal{E}_q}}\overline{[b]}\ne 0$ with respect to the pairing $\cup_{\overline{\mathcal{E}_q}}$ in $(\ref{eq:ohWeil!})$ on $\overline{\mathcal{E}_q}$. Let $\theta_q:C = \mathcal{Y}_q \to \mathcal{E}_q$ be the morphism associated to the construction of $\mathcal{Y}_q$.  The pullbacks $\theta_q^*(\alpha)\in \rH^1(C,\mathbb{Z}/3)$ and $\theta_q^*([b])\in\rH^1(C,\mu_3)$ are normalized classes at $O$, but 
$$\theta_q^*(\alpha) \cup \theta_q^*([b]) \neq \frac{1}{d(O)} \cdot \left( \overline{\theta_q^*(\alpha)}\cup_{\overline{C}}\overline{\theta_q^*([b])}\right)\cdot [O] \quad\mbox{in}\quad \mathrm{Pic}(C) /3\cdot \mathrm{Pic}(C) = \rH^2(C,\mu_3).$$
\end{thm}

\begin{proof}  
Let $k=\mathbb{F}_q$, and let $\overline{k}$ be a fixed algebraic closure of $k$.
By Corollary \ref{cor:modp}, the $3$-torsion of $\mathrm{Pic}(\mathcal{E}_q)$ is isomorphic to $\mathbb{Z}/3 \times \mathbb{Z}/3$. In particular, $\mathrm{Pic}(\mathcal{E}_q)[3]=\mathrm{Pic}(\overline{\mathcal{E}_q})[3]$, which implies $\tilde{\mu}_3\subset k^*$. Therefore, $\rH^1(\mathcal{E}_q,\mu_3)=\rH^1(\mathcal{E}_q,\mathbb{Z}/3)\otimes \tilde{\mu}_3$. Since the Weil pairing (\ref{eq:Weildef2}) on $\overline{\mathcal{E}_q}$ is non-degenerate, we can use (\ref{eq:Weildef}) and (\ref{eq:ohWeil!}), together with Lemmas \ref{lem:normalize} and \ref{lem:normalizedcharacter}, to find normalized classes $\alpha\in \rH^1(\mathcal{E}_q,\mathbb{Z}/3)_{0_{\mathcal{E}_q}}$ and $[b]\in\rH^1(\mathcal{E}_q,\mu_3)_{0_{\mathcal{E}_q}}$ such that the restrictions $\overline{\alpha}$ and $\overline{[b]}$ to $\overline{\mathcal{E}_q}$ satisfy $\overline{\alpha}\cup_{\overline{\mathcal{E}_q}}\overline{[b]}\ne 0$  in $\mathbb{Z}/3$, with respect to the pairing $\cup_{\overline{\mathcal{E}_q}}$ in (\ref{eq:ohWeil!}) on $\overline{\mathcal{E}_q}$.

Cup products respect restrictions and pullbacks by $\theta_q$, so we find that
\begin{equation}
\label{eq:relation0}
\overline{\theta_q^*(\alpha)}\cup_{\overline{C}}\overline{\theta_q^*([b])} = \theta_q^* ( \overline{\alpha}\cup_{\overline{\mathcal{E}_q}}\overline{[b]} )
\end{equation}
when $\theta_q^*$ on the right is the pullback map
\begin{equation}
\label{eq:twoisom}
\theta_q^*:\rH^2(\overline{\mathcal{E}_q},\mu_3) \to \rH^2(\overline{C},\mu_3).
\end{equation}
When we identify the domain and range of $\theta_q^*$ in (\ref{eq:twoisom}) with $\mathbb{Z}/3$, the map $\theta_q^*$ becomes multiplication by $2$, since $\theta_q:C \to \mathcal{E}_q$ is a degree two map of curves with constant field $k$.  This and (\ref{eq:relation0}) imply $\overline{\theta_q^*(\alpha)}\cup_{\overline{C}}\overline{\theta_q^*([b])}\ne 0$ in $\mathbb{Z}/3$. Suppose now that in fact,
\begin{equation}
\label{eq:nottrue}
\theta_q^*(\alpha) \cup \theta_q^*([b]) = \frac{1}{d(O)} \cdot \left( \overline{\theta_q^*(\alpha)}\cup_{\overline{C}}\overline{\theta_q^*([b])}\right)\cdot [O] \quad\mbox{in}\quad \mathrm{Pic}(C) /3\cdot \mathrm{Pic}(C) = \rH^2(C,\mu_3)
\end{equation}
where $d(O) = 1$. 
Let $\sigma$ be the non-trivial automorphism of $C$ over $\mathcal{E}_q$. Then the action of $\sigma$ fixes $\theta_q^*(\alpha)$ and $\theta_q^*([b])$ and it is equivariant with respect to cup products.  Moreover, $\sigma$ acts on $\rH^2(C, \mu_3) = \mathrm{Pic}(C) / 3\cdot \mathrm{Pic}(C)$ via the automorphism of $\mathrm{Pic}(C)$ induced by $\sigma$.   Hence (\ref{eq:nottrue}) would imply
\begin{eqnarray}
\label{eq:oppsey}
\theta_q^*(\alpha) \cup \theta_q^*([b])&=& \sigma\left(\theta_q^*(\alpha)\right) \cup \sigma\left(\theta_q^*([b])\right) = \sigma\left( \theta_q^*(\alpha) \cup \theta_q^*([b]) \right) = \sigma\left(\left( \overline{\theta_q^*(\alpha)}\cup_{\overline{C}}\overline{\theta_q^*([b])}\right)\cdot [O]\right)\nonumber\\
 &=& \left( \overline{\theta_q^*(\alpha)}\cup_{\overline{C}}\overline{\theta_q^*([b])}\right)\cdot [\sigma(O)].
\end{eqnarray}
Subtracting the right side of (\ref{eq:oppsey}) from the right side of (\ref{eq:nottrue}), using $d(O)=1$, we get
$$0 = \left( \overline{\theta_q^*(\alpha)}\cup_{\overline{C}}\overline{\theta_q^*([b])}\right)\cdot \left( [O]-[\sigma(O)]\right) \quad \mathrm{in} \quad \mathrm{Pic}(C) / 3\cdot \mathrm{Pic}(C).$$
Since $\overline{\theta_q^*(\alpha)}\cup_{\overline{C}}\overline{\theta_q^*([b])}\ne 0$ in $\mathbb{Z}/3$, this would force $[O] - [\sigma(O)]$ to lie in $3 \cdot \mathrm{Pic}(C)$. However, $\{O,\sigma({O)}\} = \{0_{\mathcal{Y}_q},0'_{\mathcal{Y}_q}\}$ and we have shown in Corollary \ref{cor:modp}  that the difference $[0_{\mathcal{Y}_q}]-[0'_{\mathcal{Y}_q}]$ is not in $3 \cdot \mathrm{Pic}(C)$.  So the contradiction shows that (\ref{eq:nottrue}) cannot be true.
\end{proof}


\begin{thebibliography}{99}

\bibitem{AchingerThesis} Achinger, P.,
\newblock $K(\pi,1)$-Spaces in Algebraic Geometry.
\newblock Thesis (Ph.D.) -- University of California, Berkeley, 2015.

\bibitem{Achinger2015} Achinger, P., 
\newblock $K(\pi,1)$-neighborhoods and comparison theorems. 
\newblock {\em Compos. Math.} 151 (2015), 1945--1964.

\bibitem{BCGKPT} Bleher, F. M., Chinburg, T., Greenberg, R., Kakde, M., Pappas, G. and Taylor, M. J.,
\newblock Cup products in the \'etale cohomology of number fields.
\newblock {\em New York J. Math.} 24 (2018),  514--542.

\bibitem{BS} Boneh, D. and Silverberg, A.,
\newblock Applications of multilinear forms to cryptography.
\newblock {\em Contemp. Math.} 324 (2003), 71--90.

\bibitem{Bourbaki10} Bourbaki, N.,
\newblock \'El\'ements de Math\'ematique. Alg\`ebre, Chapitre 10.
\newblock N. Bourbaki et Springer-Verlag, 2007.

\bibitem{brown} Brown, K.,
\newblock Cohomology of Groups.
\newblock Graduate Texts in Mathematics, 87, Springer-Verlag, 1982.

\bibitem{SGA4.5} Deligne, P. (with Boutot, J.F., Grothendieck, A., Illusie, L. and Verdier, J. L.),
\newblock Cohomologie \'Etale.
\newblock S\'eminaire de G\'eometrie Alg\'ebrique du Bois-Marie SGA 4$\frac{1}{2}$,
Lecture Notes in Mathematics, 569, Springer-Verlag,  1977.

\bibitem{Howe} Howe, E. W.,
\newblock The Weil pairing and the Hilbert symbol.
\newblock {\em Math. Ann.} 305 (1996), 387--392.

\bibitem{Joux}  Joux, A.,
\newblock A one round protocol for tripartite Diffie-Hellman.
\newblock Lecture Notes in Comput. Sci., 1838, Springer-Verlag, 385--393.


\bibitem{LuRo2015} Lubicz, D. and Robert, D.,
\newblock A generalisation of Miller's algorithm and applications to pairing computations 
on abelian varieties.
\newblock {\em J. Symbolic Comput.} 67 (2015), 68--92. 

\bibitem{MayleWang} Mayle, J. and Wang, T.,
\newblock On the effective version of Serre's open image theorem.
\newblock {\em Bull. Lond. Math. Soc.} 56 (2024), 1399--1416.

\bibitem{MS} McCallum, W. and Sharifi, R.,
\newblock A cup product in the Galois cohomology of number fields.
\newblock {\em Duke Math. J.} 120 (2003), 269--310.


\bibitem{Miller} Miller, V. S.,
\newblock The Weil pairing and its efficient calculation.
\newblock {\em J. Cryptology} 17 (2004), 235--261.

\bibitem{Milne} Milne, J., 
\newblock \'Etale cohomology.
\newblock Princeton University Press, 1980.

\bibitem{MilneADT} Milne, J.,  
\newblock Arithmetic Duality Theorems, second edition. 
\newblock Academic Press, 2006.

\bibitem{Mumford} Mumford, J., 
\newblock Abelian Varieties. 
\newblock Oxford University Press, 1970.

\bibitem{RibesZ}  Ribes, L. and Zalesskii, P.,
\newblock Profinite groups, second edition. 
\newblock Springer-Verlag, 2010.

\bibitem{Serre} Serre, J. P.,
\newblock Propri\'et\'es galoisiennes des points d'order fini des courbes elliptiques.
\newblock {\em Invent. Math.} 15 (1972), 259--331

\bibitem{Simon} Simon, B.,
\newblock Convexity. 
\newblock Cambridge Tracts in Math., 187, Cambridge University Press,  2011. 

\bibitem{Weil} Weil, A., 
\newblock Basic Number Theory, third edition.
\newblock Springer-Verlag, 1974.
 
\end{thebibliography}
\end{document}